%% file: manuscript.tex
\documentclass[preprint,12pt, a4paper]{elsarticle}

\usepackage{subfiles}
\journal{Chemical Engineering Science}

\usepackage{graphicx}
\extrafloats{200}

\usepackage{amsmath}
\usepackage{amsthm}
\usepackage{amsfonts}
\usepackage{amssymb}
\usepackage{mathtools}
\newcommand{\defeq}{\vcentcolon=}

\DeclareSymbolFontAlphabet{\amsmathbb}{AMSb}%

\newtheorem{theorem}{Theorem}
\newtheorem{lemma}{Lemma}
\newtheorem{definition}{Definition}
\newtheorem{proposition}{Proposition}
\newtheorem{corollary}{Corollary}

\newtheorem{remark}{Remark}

\usepackage[plainpages=false, pdfpagelabels]{hyperref}
\usepackage{cleveref}

\usepackage{xcolor}
\usepackage[normalem]{ulem}

\makeatletter
\newcommand*{\rom}[1]{\expandafter\@slowromancap\romannumeral #1@}
\makeatother

\usepackage{algorithm}
\usepackage{algpseudocode}
\usepackage{pifont}

\begin{document}

\begin{frontmatter}

\title{A geometric reformulation of the bilevel parameter optimization problem to a single level non-linear programming problem with applications to phase equilibria}

\author[bre,iwt]{Stefan C. Endres\corref{cor1}}
\ead{s.endres@iwt.uni-bremen.de}
\cortext[cor1]{Corresponding author.}
\author[bre,iwt]{Lutz M\"adler}


\affiliation[bre]{organization={University of Bremen, Faculty of Production Engineering}, 
  city={Bremen},
  postcode={28359},
  country={Germany}}

\affiliation[iwt]{organization={Leibniz Institute for Materials Engineering - IWT}, 
  city={Bremen},
  postcode={28359},
  country={Germany}}

\begin{abstract}
\input{sections/abstract}
\end{abstract}


\end{frontmatter}

\input{sections/nomenclature}

\section{Introduction and Background}
\input{sections/introduction_background}

\section{Methods}\label{sec:methods}
\input{sections/methods}

\section{Results}
\input{sections/results}
\input{sections/results_lle}
\input{sections/results_sour_gas}
\input{sections/results_ccs}
\input{sections/results_refrigerant}


\section{Conclusions}
\input{sections/conclusions}

\section*{Supplementary material}
The complete per-data-point Gibbs free energy and dual-manifold
archives that certify the stability record of every fitted case are
provided as three separate supplementary documents:
Supplementary Material~S1 (per-temperature plots for the four binary
LLE benchmark cases), Supplementary Material~S2 (per-tie-line plots
for the sour-gas ternary case), and Supplementary Material~S3
(per-tie-line plots for the CCS and refrigerant validation cases).
Representative example plots from each archive are reproduced in the
corresponding results sections of the main text.

\bibliographystyle{elsarticle-num}
\bibliography{references}

\appendix
\section{The dual manifold: formulation and global-consistency proofs}\label{appendix:proofs}
\input{sections/appendix_esurf}
\input{sections/appendix_phi_zero_proof}

\section{The two-stage geometric bilevel solver: algorithmic specification and convergence}\label{appendix:solver}
\input{sections/appendix_algorithms}
\input{sections/appendix_convergence_proof}

\section{Validation of the CEOS inner solver against Glass et al.}\label{appendix:ceos-validation-sec}
\input{sections/appendix_ceos_validation}


\end{document}


\begin{frontmatter}
\title{Supplementary Material S1: Per-temperature Gibbs and dual-manifold
plots for the four binary LLE benchmark cases\\[1ex]
\normalsize{Supplement to: A geometric reformulation of the bilevel
parameter optimization problem to a single level non-linear programming
problem with applications to phase equilibria}}
\end{frontmatter}

\section{Per-data-point Gibbs and dual-manifold plots for the LLE benchmark cases}
\label{supp:lle-per-point}
\input{sections/supp_lle_per_point}

\bibliographystyle{elsarticle-num}
\bibliography{references}

%% file: sections/abstract.tex
Many problems in science and engineering are modelled with an underlying energy potential surface from which more complex models and phenomena are derived. Phase equilibrium problems in particular are central to chemical engineering, underpinning tasks ranging from separation process design to the development of thermodynamic models. A particularly challenging problem from the perspective of computational engineering is the generation of phase envelopes: the inverse problem of rigorously fitting thermodynamic models to real world data with well behaved predictions requires finding the solution of a computationally expensive bilevel optimization. We present a geometric reformulation of the parameter bilevel optimization problem for fitting thermodynamic models to equilibrium data. The reformulation allows for the bilevel problem to be restated as a single level problem that is significantly easier to solve. The solution to the single level problem is proven to be the globally optimal solution of the bilevel problem when specialized global optimization solvers are used. In addition, the method retains the well behaved resultant model guaranteed by its constraints such as enforcing the correct number of phase splits by excluding spurious phases and ensuring the stability and correct number of phases in regions of instability. This allows the practitioner to reliably and efficiently fit mathematically complex thermodynamic models to data and potentially allows for the highly accurate and rigorous modelling and simulation of problems in computational thermodynamics that were previously intractable. Finally an algorithm is presented that is proven to converge for any black box thermodynamic model, therefore only an expression of the Gibbs free energy is required, the model's derivatives are not required and convergence is guaranteed for the broadest class of non-smooth, non-continuous models.

%% file: sections/nomenclature.tex

\section*{Nomenclature}
\addcontentsline{toc}{section}{Nomenclature}

\subsection*{Sets and indices}

\noindent
\begin{tabular}{@{}p{0.18\linewidth}p{0.78\linewidth}@{}}
$n$              & Number of mole-fraction coordinates of the energy surface ($c+1$). \\
$c = n-1$        & Number of independent chemical components. \\
$\Delta^{n}$     & Composition simplex,
                   $\{\mathbf{x}\in\mathbb{R}^{n}_{\geq 0} : \sum_{j} x_{j} = 1\}$. \\
$i = 1,\dots,k$  & Index of an experimental tie-line; $k$ is the number of
                   tie-lines at fixed $(P,T)$. \\
$j$              & Index over chemical components, $j = 1,\dots,n$. \\
$s$              & Index labelling distinct regions of instability that share
                   the same $(P,T)$. \\
$\alpha,\beta$   & Phase superscripts on coexisting compositions
                   $\mathbf{x}^{\alpha,i}, \mathbf{x}^{\beta,i}$;
                   synonymous with the binary-equilibrium labels
                   ``Phase~I'' and ``Phase~II''. \\
$\mathcal{D}$    & Set of experimental temperatures in an isobaric dataset. \\
$\mathcal{X}_{D}^{-}(\mathbf{p})$ & Set of \emph{defective minima} of the dual
                   manifold at parameter $\mathbf{p}$, equation~\eqref{eq:XDmin}. \\
$\mathcal{X}_{D}^{+}(\mathbf{p})$ & Set of \emph{defective maxima} of
                   $-\mathcal{M}$ at parameter $\mathbf{p}$. \\
\end{tabular}

\subsection*{Decision variables}

\noindent
\begin{tabular}{@{}p{0.18\linewidth}p{0.78\linewidth}@{}}
$\mathbf{x}\in\Delta^{n}$
                 & Mole-fraction (composition) vector. \\
$x_{j}$          & Mole fraction of component $j$. \\
$\mathbf{x}^{\alpha,i},\mathbf{x}^{\beta,i}$
                 & Equilibrium endpoint compositions of tie-line $i$. \\
$\mathbf{x}_{\mathrm{near}}$
                 & Nearest equilibrium endpoint to a defective minimum,
                   $\arg\min_{\mathbf{y}\in\{\mathbf{x}^{\alpha,i},\mathbf{x}^{\beta,i}\}}
                   \|\mathbf{y}-\mathbf{x}\|_{2}$. \\
$\mathbf{x}^{+}_{\mathrm{near}}$
                 & Nearest defective maximum in $\mathcal{X}_{D}^{+}$
                   measured in the joint composition--energy space
                   $(\mathbf{x},g(\mathbf{x},\mathbf{p}))$. \\
$\mathbf{p}\in P\subset\mathbb{R}^{m}$
                 & Vector of thermodynamic-model parameters
                   (the upper-level decision variable). \\
$\mathbf{p}^{\star}$
                 & Optimal parameter vector returned by the bilevel solver. \\
$\mathbf{p}_{\mathrm{fit}}$
                 & Synonym for $\mathbf{p}^{\star}$ (the bilevel-fitted
                   optimum) used in the Supplementary Material S1 figure
                   captions. \\
$\mathbf{d} = \mathbf{x}^{\beta,i}-\mathbf{x}^{\alpha,i}$
                 & Tie-line direction. \\
$\lambda \in \mathbb{R}$
                 & Affine parameter along a tie-line:
                   $\mathbf{x}(\lambda)=\mathbf{x}^{\alpha}+\lambda(\mathbf{x}^{\beta}-\mathbf{x}^{\alpha})$. \\
\end{tabular}

\subsection*{Thermodynamic state and physical parameters}

\noindent
\begin{tabular}{@{}p{0.18\linewidth}p{0.78\linewidth}@{}}
$T$              & Absolute temperature $[\mathrm{K}]$. \\
$P$              & Pressure $[\mathrm{Pa}]$. \\
$T_{\mathrm{ref}}$
                 & Reference temperature for the NRTL temperature-dependence
                   form (\,$T_{\mathrm{ref}} = 298.15~\mathrm{K}$ throughout). \\
$\alpha_{ij}$    & NRTL non-randomness parameter (binary). \\
$\tau_{ij}$      & NRTL binary-interaction energy parameter,
                   $\tau_{ij}=A_{ij}+B_{ij}(T_{\mathrm{ref}}/T-1)
                                 +C_{ij}(T/T_{\mathrm{ref}}-1)$. \\
$A_{ij},B_{ij},C_{ij}$
                 & Temperature-coefficient triplet of $\tau_{ij}$;
                   the six values $(A_{12},B_{12},C_{12},A_{21},B_{21},C_{21})$
                   form $\mathbf{p}$ for binary NRTL fits. \\
$\mu_{j}$        & Chemical potential of component $j$,
                   $\mu_{j}=\partial G/\partial x_{j}$. \\
$\varepsilon,\sigma$
                 & Pair-potential depth and characteristic distance of the
                   Lennard--Jones. \\
\end{tabular}

\subsection*{Energy surfaces, hyperplanes and the dual manifold}

\noindent
\begin{tabular}{@{}p{0.18\linewidth}p{0.78\linewidth}@{}}
$U(\mathbf{x},\mathbf{p})$
                 & Generic potential energy surface. \\
$G(\mathbf{x},\mathbf{p})$
                 & Total Gibbs free energy of the mixture at fixed $(P,T)$. \\
$g(\mathbf{x},\mathbf{p})$
                 & Reduced (dimensionless) Gibbs free energy,
                   $g = G/(RT)$; the universal solver argument carried
                   through both the outer parameter level and the inner
                   equilibrium level of the bilevel formulation. \\
$h_{i}(\mathbf{x})$
                 & Affine interpolant (chord / hyperplane) through the two
                   endpoints of tie-line $i$. \\
$h(\mathbf{x}) = \min_{i} h_{i}(\mathbf{x})$
                 & Composite supporting hyperplane over the simplex
                   (\,\autoref{def:h-composite-appendix},
                   Equation~\eqref{eq:hsum}). \\
$\mathcal{M}(\mathbf{x};\mathbf{p})$
                 & Dual manifold,
                   $\mathcal{M} \defeq g(\mathbf{x},\mathbf{p}) - h(\mathbf{x})$,
                   Equation~\eqref{eq:dual-manifold}. \\
$f_{\mathcal{M}}$
                 & Vector embedding $\Delta^{n}\!\to\Delta^{n+1}\subset\mathbb{R}^{n+1}$
                   of $\mathcal{M}$. \\
$[\,\cdot\,]_{-}$
                 & Negative part operator, $[r]_{-} \defeq \min(r,0)$,
                   Equation~\eqref{eq:negpart}. \\
$[\,\cdot\,]_{+}$
                 & Positive part operator, $[r]_{+} \defeq \max(r,0)$. \\
\end{tabular}

\subsection*{Objective functions and components}

\noindent
\begin{tabular}{@{}p{0.18\linewidth}p{0.78\linewidth}@{}}
$F(\mathbf{x},\mathbf{p})$
                 & Generic upper-level objective in the bilevel program
                   (Equation~\eqref{eq:bilevel_general}). \\
$J(\mathbf{p})$  & Single-tie-line line-trace objective
                   Equation~\eqref{eq:single-tieline-objective}. \\
$\Phi(\mathbf{p})$
                 & Single-level objective of the geometric reformulation,
                   Equation~\eqref{eq:phi-total}. \\
$\Phi_{\mathrm{tangent}}$
                 & Tangent-plane sub-sampling residual at the four
                   bracketing samples per tie-line. \\
$\Phi_{\mathrm{plane}}$
                 & Plane-violation sum over $\mathcal{X}_{D}^{-}$. \\
$\Phi_{\mathrm{topo}}$
                 & Topological pull from each defective minimum to its
                   nearest equilibrium endpoint. \\
$\Phi_{\mathrm{local}}$
                 & Local-topology amplitude (joint composition--energy
                   distance from each defective minimum to its paired
                   maximum). \\
$\Phi_{\mathrm{phase}}$
                 & Phase-assignment partition of $\arg\min\mathcal{M}$ into
                   defective minima vs.\ equilibrium endpoints. \\
$\epsilon$       & Bracketing offset of the four directional samples in
                   $\Phi_{\mathrm{tangent}}$; finite-difference scale
                   ($1/\epsilon$) used to promote the residual to a
                   directional derivative. \\
$\eta$           & Adaptive-$\epsilon$ scale factor,
                   $\epsilon_{i} = \eta \,\|\mathbf{x}^{\beta,i}-\mathbf{x}^{\alpha,i}\|$. \\
\end{tabular}

\subsection*{Solver and error metrics}

\noindent
\begin{tabular}{@{}p{0.18\linewidth}p{0.78\linewidth}@{}}
$N_{\mathrm{tan}}$
                 & SHGO sampling density used in the tangent-plane stage
                   (default 1000). \\
$N_{\mathrm{def}}$
                 & SHGO sampling density used in the inner defect-detection
                   stage (default 20). \\
$\epsilon_{a}$
                 & Absolute least-squares residual,
                   $\sum_{i}\|\mathbf{x}^{\mathrm{exp}}_{i}-\mathbf{x}^{\mathrm{pred}}_{i}\|^{2}$;
                   matches the ``Error'' column of
                   \cite{Mitsos2009LLE}~Table~1. \\
$\epsilon_{r,\mathrm{I}},\epsilon_{r,\mathrm{II}}$
                 & Per-phase relative errors on Phase~I ($\alpha$-rich) and
                   Phase~II ($\beta$-rich) compositions; match the
                   ``Rel-error$_{1}$'' / ``Rel-error$_{2}$'' columns of
                   \cite{Mitsos2009LLE}~Table~1. \\
\end{tabular}

\subsection*{Acronyms}

\noindent
\begin{tabular}{@{}p{0.18\linewidth}p{0.78\linewidth}@{}}
LLE              & Liquid--liquid equilibrium. \\
VLE              & Vapour--liquid equilibrium. \\
VLLE             & Vapour--liquid--liquid equilibrium. \\
EOS / CEOS       & Equation of state / cubic equation of state. \\
NRTL             & Non-Random Two-Liquid activity-coefficient model. \\
UNIQUAC          & Universal Quasi-Chemical activity-coefficient model. \\
UNIFAC           & UNIQUAC Functional-group Activity Coefficients
                   (group-contribution extension). \\
SRK, PR, PR78    & Soave--Redlich--Kwong, Peng--Robinson (1976) and
                   Peng--Robinson (1978) cubic EOS. \\
TWU              & Twu (1991) volume-translated alpha-function EOS. \\
PC-SAFT          & Perturbed-Chain Statistical Associating Fluid Theory. \\
DWPM             & Dual-Weighted Phase-Minimisation equilibrium solver. \\
TPD / TPDF       & Tangent-plane distance / tangent-plane distance function
                   \cite{Michelsen1982}. \\
KKT              & Karush--Kuhn--Tucker (first-order optimality conditions). \\
NLP              & Nonlinear program. \\
MINLP            & Mixed-integer nonlinear program. \\
SIP              & Semi-infinite program. \\
SHGO             & Simplicial-Homology Global Optimisation
                   \cite{Endres2018}. \\
PDIFF            & Piecewise-differentiable function class. \\
\end{tabular}

\bigskip

%% file: sections/introduction_background.tex

Many problems in science and engineering are modelled with an underlying energy potential surface from which more complex models and phenomena are derived. Phase equilibrium problems in particular are central to chemical engineering, underpinning tasks ranging from the design of separation processes to the development of new thermodynamic models \cite{kontogeorgis2009thermodynamic, SmithVanNessAbbott2018}. The accurate prediction of phase behaviour for both familiar and industrially relevant mixtures relies on thermodynamic models such as NRTL, UNIQUAC, Wilson, cubic equations of state, and PC-SAFT, whose parameters must be fitted to experimental data in a manner that is consistent with the laws of thermodynamics governing phase coexistence. Determining these parameters is, however, an instance of a computationally expensive \emph{bilevel} optimization problem, in which an outer parameter-fitting objective is constrained at every iteration by an inner global Gibbs free energy minimization \cite{Mitsos2009LLE, Bollas2009VLE}. The reliable, efficient solution of this bilevel problem, together with a geometric reformulation that recasts it as a single-level non-linear programming problem with provable global optimality guarantees, are the central concerns of this work.

\subsection{Mixture Thermodynamics and the Inverse Problem}

In thermodynamics, the equilibrium state of a system at given conditions is determined by minimizing an appropriate energy (e.g. Gibbs free energy at constant $T,P$), subject to constraints like material balances \cite{Zhang2011, Mitsos2007}. Conversely, in \emph{inverse} problems one seeks to determine model parameters or design variables such that the resulting energy surface matches observed behavior or yields a desired outcome. The inverse problem therefore inherits the bilevel structure introduced above: the outer level adjusts model parameters or design variables, while the inner level is the energy-minimizing equilibrium state for each candidate parameter set. These problems are challenging because the inner energy minimization is typically nonconvex and may have multiple local minima, so a naive approach can converge to suboptimal or physically incorrect solutions \cite{Michelsen1982, McDonald1995, Bollas2009VLE}. In particular, phase equilibrium computations, stability analysis, and parameter estimation in thermodynamic models are all recognized as global optimization problems due to highly nonlinear, multimodal energy surfaces \cite{Bollas2009VLE, Zhang2011}.

\subsection{Use of Global Optimization in Phase Equilibria Thermodynamics}

The role of global optimization in phase equilibrium problems has been well established since Michelsen's seminal work, with the introduction of the tangent plane distance (TPD) stability criterion \cite{Michelsen1982}, which cast phase stability as a global Gibbs energy minimization problem. He showed that a single-phase mixture is stable if and only if no other phase composition yields a lower Gibbs free energy: in practice, if the minimum TPD is non-negative; a negative TPD indicates the mixture will split into multiple phases. This insight revealed the nonconvexity of the free energy surface (multiple local minima corresponding to possible phase splits). Michelsen's criterion, though not a bilevel formulation itself, underpins later methods by formalizing phase equilibrium as a global optimization problem. However, due to the highly nonlinear nature of excess Gibbs free energy models, ensuring that equilibrium calculations do not converge to a metastable state is difficult. McDonald and Floudas \cite{McDonald1995} demonstrated that deterministic global optimization is necessary to reliably solve phase stability problems, as standard local optimization methods may lead to incorrect predictions by failing to detect alternative phase splits. Their work motivated later bilevel optimization frameworks that embedded phase equilibrium calculations as rigorous global minimization problems \cite{Bollas2009VLE, Mitsos2009LLE}.

In the context of Gibbs free energy models, phase stability is verified using the \textit{tangent plane criterion}. This criterion was originally developed by Gibbs~\cite{gibbs1878} and reformulated for modern computational use by Michelsen~\cite{Michelsen1982}. A composition $\mathbf{x}^{*}$ is globally stable if the Gibbs free energy surface $G(\mathbf{x})$ lies above its tangent plane at $\mathbf{x}^{*}$ for all $\mathbf{x} \in \Delta^{n}$:
\begin{equation}
G(\mathbf{x}) \geq G(\mathbf{x}^{*}) + \sum_{j} \mu_{j}(\mathbf{x}^{*})\,(x_{j} - x^{*}_{j}).
\end{equation}
Equivalently, this can be expressed via a \textit{thermodynamic inequality} using the chemical potentials:
\begin{equation}
\mu_{j}(\mathbf{x}^{*}) - \mu_{j}(\mathbf{x}) \leq 0 \quad \forall \mathbf{x} \in \Delta^{n}.
\end{equation}
This condition must hold for all trial compositions $\mathbf{x}$ to ensure that $\mathbf{x}^{*}$ corresponds to a true equilibrium, not just a local extremum, but one satisfying global thermodynamic consistency~\cite{gibbs1878, Michelsen1982}.

In inverse problems, this constraint defines a \textit{feasibility region} in parameter space. That is, for a given set of model parameters $\mathbf{p}$, we only admit solutions $\mathbf{x}_{e}(\mathbf{p})$ that satisfy the tangent plane criterion. This results in a \textit{nonconvex bilevel problem}, where global optimization is required even to verify feasibility.

This formulation is particularly useful in ensuring that the model does not overfit data by predicting spurious phase splits or unstable equilibria, a key challenge highlighted in phase behavior classification studies~\cite{Mitsos2009LLE}.

Robust global optimization techniques are essential to reliably solve such \textit{inverse free energy surface problems}. The outer loop adjusts the thermodynamic parameters $\mathbf{p}$ to match experiments, while the inner loop performs global Gibbs energy minimization and enforces stability via the tangent plane condition. This forms the computational backbone of the Mitsos--Barton dual extremum approach to phase equilibrium.

\subsection{Formal Mathematical Statements of the Problem} \label{sec:bilevel}
\subsubsection{Formalizing the Inverse Free Energy Surface Problem}
A general inverse or bilevel optimization problem in thermodynamics can be formulated as follows:
\begin{equation}
\min_{\mathbf{x}, \mathbf{p}} F(\mathbf{x}, \mathbf{p}),
\end{equation}
where $\mathbf{p}$ denotes the thermodynamic-model parameters (e.g., binary interaction parameters in excess Gibbs energy models), and $\mathbf{x}$ represents the system configuration, such as phase compositions. The objective function $F(\mathbf{x}, \mathbf{p})$ typically quantifies the discrepancy between experimental measurements and model predictions:
\begin{equation}
F(\mathbf{p}) = \left\|\mathbf{x}^{\mathrm{exp}}_{e} - \mathbf{x}_{e}(\mathbf{p})\right\|,
\end{equation}
where $\mathbf{x}^{\mathrm{exp}}_{e}$ is a measured equilibrium composition and $\mathbf{x}_{e}(\mathbf{p})$ is the equilibrium composition predicted by the model at parameter set $\mathbf{p}$.

To obtain $\mathbf{x}_{e}(\mathbf{p})$, one must solve a \textit{lower-level optimization problem} that corresponds to minimizing the total Gibbs free energy of the system:
\begin{equation} \label{eq:min Gibbs}
\min_{\mathbf{x}} G(\mathbf{x}, \mathbf{p}),
\end{equation}
subject to phase and mole balance constraints.

Alternatively, this can be reformulated as a \textit{tangent plane condition} that serves as a thermodynamic stability criterion:
\begin{equation}
\mu_j(\mathbf{x}^{*}) - \mu_j(\mathbf{x}) \leq 0 \quad \forall \mathbf{x} \in \Delta^{n}, \quad \forall j,
\end{equation}
where $\mu_j(\mathbf{x})$ is the chemical potential of component $j$ at composition $\mathbf{x}$, and $\mathbf{x}^{*}$ is a candidate equilibrium phase. This constraint ensures that no tangent plane constructed at $\mathbf{x}^{*}$ lies above the free energy surface, a necessary and sufficient condition for phase stability~\cite{gibbs1878, Michelsen1982}.

The full bilevel optimization problem can be written as:
\begin{equation} \label{eq:bilevel_general}
\begin{aligned}
\min _{\mathbf{x} \in \Delta^{n},\,\mathbf{p} \in P} & \quad F(\mathbf{x}, \mathbf{p}) \\
\text{subject to:} & \quad C^{\mathrm{out}}_{i}(\mathbf{x}, \mathbf{p}) \leq 0, \quad \forall i, \\
& \quad \mathbf{x} \in \arg \min _{\bar{\mathbf{x}} \in \Delta^{n}}\left\{ G(\bar{\mathbf{x}}, \mathbf{p}) : C^{\mathrm{in}}_{j}(\bar{\mathbf{x}}, \mathbf{p}) \leq 0,\ \forall j \right\}.
\end{aligned}
\end{equation}

Equation~\eqref{eq:bilevel_general}
is usually multimodal, and robust global optimization techniques are therefore essential for reliably solving such inverse energy-surface problems \cite{Esposito1998, Gau2002, Bollas2009VLE, Zhang2011}.

\subsubsection{Necessary and Sufficient Stability Criteria} \label{sec:neccessary_sufficient_stability}

It is essential to distinguish \emph{necessary} from \emph{sufficient} stability criteria, a distinction sharpened by Baker, Pierce \& Luks~\cite{Baker1982} 
in the context of cubic equations of state. The classical isopotential constraint $\mu_i^{\mathrm{I}} = \mu_i^{\mathrm{II}}$, together with material balance, is necessary for phase equilibrium but is satisfied by trivial single-phase solutions and by false multi-phase tangencies that lie above the global Gibbs envelope. The \emph{sufficient} criterion is the global non-negativity of the tangent plane distance function (TPDF) constructed at every candidate phase composition~\cite{Baker1982,Michelsen1982}. 
Glass, Djelassi \& Mitsos~\cite{Glass2018CEOS} 
recently emphasised that almost all standard regression formulations enforce only the necessary criterion, with the sufficient criterion checked (when it is checked at all) only \emph{a posteriori}.

For activity-coefficient ($\Delta G^E$) models the stability picture is purely \emph{compositional}: the only state variable on the lower level is the mole fraction $\mathbf{x}$. Cubic equations of state (CEOS), by contrast, introduce an additional stability dimension. The cubic in molar volume can admit up to three real roots, of which only the smallest (liquid-like) and largest (vapor-like) are physically admissible; the intermediate root is mechanically unstable and must be excluded by a \emph{mechanical}-stability test alongside the diffusive (compositional) test~\cite{Glass2018CEOS,Kamath2010}. 
Imposing the cubic itself as an equality constraint inside the bilevel program ties parameter regression to a non-trivial root-discrimination criterion, ranging from the implicit reformulation of Nagarajan, Cullick \& Griewank~\cite{NagarajanCullick1991} 
to the mixed-integer big-$M$ formulation of Kamath, Biegler \& Grossmann~\cite{Kamath2010}. 
Both formulations live below the parameter-regression upper level and must be solved to global optimality if the resulting parameter set is to be guaranteed thermodynamically consistent in subsequent process simulation.

\subsection{The Inverse Problem of Fitting Thermodynamic Models to Phase Equilibrium Data}
Accurate prediction of phase equilibria (vapor--liquid, liquid--liquid, etc.) is critical in chemical engineering for design and simulation of separation processes. The thermodynamic models (e.g. activity coefficient models such as NRTL, UNIQUAC, Wilson) contain parameters that must be fitted to experimental phase equilibrium data. This parameter estimation is a typical bilevel problem as described in the previous section~\ref{sec:bilevel}: the upper level seeks the best-fit parameters (typically by minimizing a least-squares discrepancy between predicted and measured phase compositions or other properties), and the lower level corresponds to the phase equilibrium calculation for each experimental condition \cite{Bollas2009VLE, Mitsos2009LLE}. In traditional approaches, the phase equilibrium computations are embedded as equations (e.g. enforcing equality of fugacities or chemical potentials in each phase) and the entire problem is solved as a single-level nonlinear program. However, standard local optimization methods often converge to \emph{locally} optimal parameters that yield qualitatively wrong predictions \cite{Bollas2009VLE}. Notably, a model fit to data could reproduce the measurements well yet predict spurious phase behavior outside the measured range. Examples of such pitfalls include predicting additional phases or phase splits that were not actually observed, or misclassifying a homogeneous azeotrope as heterogeneous \cite{Bollas2009VLE, Mitsos2007}. These errors arise when the parameter estimation converges to a solution that is only a local optimum of the model's Gibbs energy surface, not the true global minimum \cite{Bollas2009VLE, Zhang2011}.

Researchers have long been aware of the need to consider global phase stability in model fitting. The tangent plane criterion introduced above is effectively a lower-level optimization problem embedded in phase equilibrium calculations. Early work implemented tangent plane stability analysis using methods like interval Newton solvers \cite{Piro2016} and homotopy continuation to ensure that all possible incipient phases are detected. Likewise, direct Gibbs energy minimization for equilibrium (with multiple phases allowed) was tackled via global optimization algorithms as early as the 1990s. For example, Nichita \emph{et al.} (2002) formulated the phase equilibrium problem as a nonconvex minimization of total $G$ and applied deterministic global solvers to reliably find the global minimum \cite{Piro2016, Zhang2011}. These efforts addressed the \emph{inner} problem of phase equilibrium. However, integrating such global approaches into parameter estimation (the \emph{outer} problem) was nontrivial, and for many years parameter fitting was often done with local methods that did not guarantee thermodynamically consistent results \cite{Bollas2009VLE, Zhang2011}.

A breakthrough came with the work of Mitsos, Barton, Bollas and co-workers, who formulated parameter estimation in phase equilibrium models as an explicit bilevel optimization problem \cite{Bollas2009VLE, Mitsos2009LLE}. In their approach, the upper-level objective typically minimizes the discrepancy between experimental and predicted phase compositions (e.g. a least-squares error). The lower level consists of multiple Gibbs free energy minimization problems (essentially one for each experiment, and for each potential phase split scenario) which enforce that the model's predictions correspond to true equilibrium states consistent with the observed phase behavior. In simple terms, the fitting procedure ``asks'' at every step: \emph{given the current parameters, does the model predict the correct number of phases and phase types for each experiment, and are those phases globally stable?} These questions are answered by solving the inner Gibbs minimization problems to global optimality. If the model with a given parameter set would predict an extra spurious phase or the wrong type of phase split, the global minimum of the appropriate inner problem will reveal this, and such a parameter set will be rejected in the bilevel optimization. By searching over parameter space with this rigorous inner check, the algorithm can identify parameter values that reproduce \emph{exactly} the experimentally observed phase behavior: not just the measured data points, but the correct qualitative regime (no missing or extra phases, correct azeotrope type, etc.) \cite{Bollas2009VLE, Mitsos2009LLE}. Bollas \emph{et al.} (2009) demonstrated this bilevel formulation on vapor--liquid and vapor--liquid--liquid equilibrium data for binary mixtures \cite{Bollas2009VLE}. They showed that conventional fitting methods often converge to solutions that appear to fit the data well yet predict an extra liquid--liquid split or misclassify a homogeneous azeotrope as heterogeneous, precisely the issues noted above \cite{Bollas2009VLE, Mitsos2007}. The bilevel approach successfully avoids those erroneous fits by explicitly excluding any parameter set that fails the global stability tests. As a result, the parameters obtained from the bilevel formulation yield thermodynamically consistent predictions across the entire composition range \cite{Mitsos2007, Mitsos2009LLE}.

Mathematically, the bilevel formulation for parameter estimation can be written in a simplified form as:
\begin{equation}\label{eq:bilevel}
\begin{aligned}
    \min_{\mathbf{p}} \; & F(\mathbf{p}) \;=\; \sum_{i=1}^{k} w_{i}\,\big\|\mathbf{x}^{\mathrm{pred}}_{i}(\mathbf{p}) - \mathbf{x}^{\mathrm{exp}}_{i}\big\|^{2}, \\[0.5em]
    \text{s.t.}\;& \mathbf{x}^{\mathrm{pred}}_{i}(\mathbf{p}) = \arg\min_{\mathbf{x},\;\mathrm{phases}} \; G_{i}\!\big(\mathbf{x}; \mathbf{p}\big), \qquad i=1,\dots,k,
\end{aligned}
\end{equation}
where $\mathbf{p}$ represents the model parameters (e.g.\ binary interaction parameters in an activity-coefficient model), $\mathbf{x}^{\mathrm{exp}}_{i}$ are the measured equilibrium compositions for experiment $i$, and $G_{i}(\mathbf{x};\mathbf{p})$ is the Gibbs free energy of the system in experiment $i$'s conditions as a function of the phase split $\mathbf{x}$ (phase compositions and amounts). The inner minimization $\min G_{i}$ is subject to material balance constraints and can produce one or multiple equilibrium phases; essentially, it encodes the condition that the system's free energy is at a global minimum for the given $\mathbf{p}$. Solving \eqref{eq:bilevel} is challenging: each inner problem is a nonconvex optimization (and may itself require a global solver to avoid false solutions), and the outer problem must propagate the global optimality conditions. Mitsos \emph{et al.} addressed this by using advanced global optimization techniques. In their implementation, the bilevel problem is transformed into a single-level formulation by replacing each inner minimization with equivalent constraints (such as the Karush--Kuhn--Tucker conditions of the Gibbs minimization, plus additional constraints to enforce global optimality) \cite{Mitsos2009LLE, Mitsos2007}. This leads to a large, highly nonconvex nonlinear programming (NLP) problem. They then apply deterministic global NLP solvers (for example, BARON) within a branch-and-bound framework to solve it \cite{Gumus2001}. Notably, global solution techniques are necessary even to find a feasible initial point for these problems, because if any inner problem is not solved to global optimality, the constraints can be violated. The result is a computationally intensive but rigorously reliable approach: if the algorithm converges, it has found a set of parameters $\mathbf{p}$ that globally minimizes the fitting objective while guaranteeing that each experiment's phase equilibrium condition is satisfied with no spurious phases. This rigor was unprecedented in thermodynamic model fitting. The bilevel formulation thereby ensures \emph{thermodynamic consistency} of the fitted model by embedding the second law (global energy minimum) directly into the estimation procedure. Subsequent case studies showed that this method can fit models like NRTL to complex phase diagrams, such as mixtures with multiple azeotropes or liquid--liquid splitting, without the qualitative errors that plagued earlier methods \cite{Mitsos2007, Bollas2009VLE}.

It is worth noting that enforcing global phase equilibrium in the fit effectively turns on/off certain model features. For example, if an incorrect heterogeneous azeotrope were predicted by a given parameter set, the inner problem would find a lower-energy two-phase solution, signaling that the parameter set is not acceptable \cite{Bollas2009VLE}. Thus, the optimizer will favor parameter sets that do not allow that false phase split. In a sense, the algorithm searches for a Gibbs energy surface topology that exactly matches the observed reality, a concept very much in line with the \emph{topological analysis} of Gibbs surfaces. In fact, around the same time other researchers were examining the topology of Gibbs free energy surfaces to understand modeling challenges. Olaya \emph{et al.} (2008) analyzed ``island type'' ternary LLE systems, where an immiscible region is completely enclosed by a miscible region in the composition space \cite{Olaya2008}. Their study revealed that standard models (NRTL/UNIQUAC) might struggle to fit these islands unless the parameter values produce very specific surface curvatures. This underscores the importance of considering global features of the energy surface (such as multiple local minima separated by barriers) during parameter fitting. The bilevel framework inherently does this by searching the energy surface for alternate minima at every iteration. In summary, the work of Mitsos, Barton, Bollas and co-workers in 2007--2009 established a rigorous foundation for treating phase equilibrium parameter estimation as a bilevel (or even a semi-infinite) optimization problem, thereby overcoming key consistency issues.

\subsection{Spurious phases from standard regression: documented case studies}\label{sec:spurious-phases}

The qualitative failure mode highlighted by Bollas, Barton \& Mitsos~\cite{Bollas2009VLE,Mitsos2009LLE}, namely a parameter set that interpolates the data points yet predicts a phase split that was never observed, is not a hypothetical concern. It is documented for several mainstream cubic equation of state (CEOS) regressions in the literature, and Glass, Djelassi \& Mitsos~\cite{Glass2018CEOS} 
recently catalogued three of them in detail. We summarise each here because they motivate the present work just as concretely as they motivate the Glass formulation.

\textbf{CF$_4$/CHF$_3$ with SRK (Asselineau, Bogdani\'c \& Vidal, 1978).} The Soave--Redlich--Kwong fit reported by Asselineau \emph{et al.}~\cite{Asselineau1978} 
predicts a liquid-phase Gibbs free energy that is mildly nonconvex at $T = 145.16$~K and $P = 100.6$~kN\,m$^{-2}$. A flash on this surface using Aspen Plus' RGibbs returns a spurious LLE with $x_{\mathrm{CF_4,I}} = 0.38499$ and $x_{\mathrm{CF_4,II}} = 0.5938$~\cite{Glass2018CEOS}, 
yet the experimental upper critical solution temperature of the binary, measured by Croll \& Scott~\cite{CrollScott1964}, 
is well above the Asselineau-fit operating temperature.

\textbf{Acetone/water with PR (Trebble \& Bishnoi, 1988).} The Peng--Robinson regression of the narrow-boiling acetone/water VLE reported by Trebble \& Bishnoi~\cite{TrebbleBishnoi1988} 
yields a liquid-phase Gibbs free energy that is locally nonconvex on the acetone-rich side. Flashing the corresponding feed at $T=373.124$~K and $P=266.83$~kN\,m$^{-2}$ produces a false LLE with $x_{\mathrm{C_3H_6O,I}} = 0.27448$ and $x_{\mathrm{C_3H_6O,II}} = 0.44173$~\cite{Glass2018CEOS}, 
even though only VLE is observed in the original measurements of Griswold \& Wong~\cite{GriswoldWong1952}. 
A similar pathology is reported by Trebble \& Bishnoi for the methanol/benzene system, where PR and SRK both predict a spurious second liquid phase even though only VLE has been measured (Strubl \emph{et al.}, 1973), as captured by Glass \emph{et al.}~\cite{Glass2018CEOS}. 

\textbf{Process-simulation impact.} The above examples are not academic curiosities. Each of the three flagged regressions is reproducible in a commercial flowsheet simulator: Glass \emph{et al.}~\cite{Glass2018CEOS} verified them with Aspen Plus 8.8 RGibbs and showed that the engineer downstream of such a regression sees physically meaningless phase splits at design conditions. 
The root cause is the same in every case: standard regression formulations enforce only the necessary equilibrium criterion, so the optimiser is free to settle on any parameter set that places the model curve through the data, including parameter sets whose Gibbs surface admits low-lying tangent planes that the regression objective is blind to~\cite{Glass2018CEOS,Bollas2009VLE,Mitsos2009LLE}. 
Constructive remedies that make the sufficient criterion part of the optimisation, rather than an \emph{a posteriori} check, are reviewed next.

\subsection{Advances and Extensions in Bilevel Optimization for Energy Surfaces}
The introduction of a bilevel formulation for phase equilibrium problems spurred further research and provided a benchmark for reliability. Subsequent efforts can be grouped into a few categories: (i) improvements and alternative algorithms for solving the bilevel problem or its single-level reformulation, (ii) broader applications of bilevel optimization to related problems (e.g. reactive systems, higher-dimensional mixtures), and (iii) general advances in global optimization of energy surfaces (including potential energy landscapes) that can support these applications.

\textbf{Deterministic global optimization approaches:} Deterministic methods guarantee finding the true global optimum given sufficient computational effort, and several such methods have been adapted to phase equilibrium calculations. Prior to the bilevel formulation, deterministic global solvers were already applied to the inner problems (stability analysis and Gibbs energy minimization) \cite{Piro2016}. Gau \emph{et al.} (2000) combined interval Newton methods with classical parameter estimation for VLE data, showing that multiple distinct sets of model parameters could satisfy local optimality conditions \cite{Piro2016}. This highlighted the existence of \emph{multiple} solutions to the fitting problem (due to nonconvex error surfaces), reinforcing the need for global methods. Later, in the bilevel context, researchers explored tailored branch-and-bound strategies and convex relaxations to reduce the computational burden. One notable contribution is the work on solving nonlinear bilevel programs to global optimality by G\"um\"us and Floudas \cite{Gumus2001}. They presented a general branch-and-bound algorithm based on convex underestimation (the $\alpha$BB method) to handle twice-differentiable bilevel problems. While their formulation was generic (not specific to thermodynamics), such advances have been influential. Mitsos \cite{Mitsos2009LLE} also addressed the closely related class of semi-infinite programs (SIPs) by proposing a method to globally solve SIPs without convexity assumptions. This is directly relevant because the Gibbs stability constraint, that no lower-energy phase exists, is effectively a semi-infinite condition. The bounding algorithm of Mitsos, Lemonidis \& Barton~\cite{Mitsos2008BLP} 
generalised this idea to bilevel programs with nonconvex inner programs without convexity assumptions, providing the algorithmic backbone subsequently used by Glass \emph{et al.}~\cite{Glass2018CEOS} for the CEOS regression problem. 
Recent surveys of nonconvex semi-infinite programming~\cite{djelassi2021recent} 
catalogue the discretisation, lower-bounding, and adaptive-restriction strategies that this lineage of work has produced.

\textbf{Stochastic and hybrid optimization approaches:} In parallel, many researchers have applied stochastic or heuristic optimization methods to energy surface problems, trading guarantees of optimality for efficiency. Numerous studies have reported applications of metaheuristics such as simulated annealing, genetic algorithms, particle swarm optimization, differential evolution, and hybrid methods to phase equilibrium calculations and parameter estimation \cite{Zhang2011}. Although stochastic methods can often locate a near-global solution efficiently, purely heuristic approaches cannot guarantee convergence to the true global optimum. Therefore, for critical applications where a false optimum could lead to erroneous designs, deterministic global methods or thorough verification are preferred.

\textbf{Applications to complex phase equilibrium problems:} With the availability of better algorithms, researchers attacked more complex equilibrium problems as bilevel or global optimization tasks. One direction is reactive phase equilibrium, where the inner problem includes chemical reaction extents in addition to phase splitting. Although more complex, the principle remains the same: the Gibbs free energy (now including chemical contributions) is minimized at equilibrium \cite{Bollas2009VLE}. Another notable application is the prediction of azeotropes and critical points via global optimization. Reviews have shown that global methods can map out entire phase diagrams and locate features such as azeotropes by systematically exploring the energy landscape \cite{Zhang2011}. Integrating global phase stability checks into parameter estimation greatly improves the reliability of phase diagram predictions.

\textbf{Bilevel programs for cubic equations of state:} The Mitsos--Bollas--Barton bilevel formulation~\cite{Mitsos2009LLE,Bollas2009VLE} treats the lower-level state variable as the composition vector $\mathbf{x}$ and is therefore tailored to $\Delta G^E$ activity-coefficient models. Imposing the same rigour on a CEOS regression is harder because the cubic in molar volume is itself an equality constraint that the lower-level optimiser must respect, and because the cubic admits up to three real roots that the regression must discriminate between. The earliest constrained-regression attempts date back to Schwartzentruber, Galivel-Solastiouk \& Renon~\cite{Schwartzentruber1987} 
and the Englezos--Kalogerakis--Bishnoi line~\cite{Englezos1989,Englezos1990}, 
which used the implicit Gibbs-stationarity criterion of Nagarajan, Cullick \& Griewank~\cite{NagarajanCullick1991} as the in-problem stability test but stopped short of solving the resulting bilevel problem to global optimality.

Glass, Djelassi \& Mitsos~\cite{Glass2018CEOS} closed this gap by extending the Mitsos--Bollas--Barton framework to CEOS through three nested lower-level problems: a Baker-style isopotential and tangent-plane-distance test (LLP$_1$), a mechanical-stability test on the volume root (LLP$_2$), and a phase-split-count test (LLP$_3$). 
The cubic itself is enforced as an equality constraint, so the resulting program is solved using the convergent bounding algorithm of Mitsos, Lemonidis \& Barton~\cite{Mitsos2008BLP} 
extended by Djelassi, Glass \& Mitsos to handle equality constraints on the lower level~\cite{djelassi2021recent}, 
and implemented in the BARON-backed BOARPET tool. The method is demonstrated on the C$_5$H$_{12}$/H$_2$S binary, on which the BLP converges in roughly 4.5\,h per case study; Glass \emph{et al.} explicitly note that this cost is prohibitive for industrial-scale problems with extensive measurement data or many components~\cite{Glass2018CEOS}. 
Complementary lines of work attack the in-problem root discrimination directly: the convex-hull test of Joseph \emph{et al.}~\cite{Joseph2017}, 
the global-tangent-plane minimisation of Harding \& Floudas via $\alpha$BB~\cite{HardingFloudas2000}, 
and the direct multiphase Gibbs minimisation of Nichita, G\'omez \& Luna~\cite{Nichita2002}. 
Glass \emph{et al.} close their outlook by noting that an extension to non-cubic equations of state, including PC-SAFT, awaits an appropriate root-discrimination criterion~\cite{Glass2018CEOS}: 
this open question motivates the model-agnostic geometric reformulation pursued in the present work.

\textbf{Potential energy surfaces and materials design:} Beyond classical thermodynamics, bilevel optimization has promising applications in materials science and molecular engineering. In these problems the energy surface is a potential energy landscape (e.g., for a molecule or crystal), and equilibrium configurations correspond to minima on that surface. For example, force-field calibration involves tuning interatomic potential parameters so that the predicted properties (e.g. lattice constants, binding energies) match experiments or quantum calculations. Often the property evaluation requires an inner optimization (geometry relaxation), mirroring the phase equilibrium fitting problem. Although such studies might not explicitly use bilevel programming terminology, the structure is analogous. More broadly, the energy landscape exploration community has developed global optimization methods to find multiple low-lying minima on complex potential energy surfaces \cite{Endres2018}. Techniques such as basin hopping and other global search strategies have proven invaluable for mapping energy landscapes and ensuring that no candidate equilibrium state is overlooked. While the simplicial homology global optimization (SHGO) algorithm was introduced to efficiently locate all local and global minima \cite{Endres2018}, here its role is secondary except where it offers new insights into energy surface exploration.

In summary, the field of bilevel optimization applied to energy surface problems has matured significantly over the past two decades. Early foundational ideas like Gibbs' phase rule and the tangent plane criterion laid the groundwork by framing equilibrium as an extremum principle \cite{Mitsos2007}. Building on this, global optimization methods were developed to reliably solve the inner phase stability and equilibrium problems \cite{Bollas2009VLE, Zhang2011}. The crucial leap to bilevel formulations enabled the direct inclusion of those inner problems into parameter estimation and design, ensuring consistency between model predictions and physical reality \cite{Bollas2009VLE, Mitsos2009LLE}. Deterministic and stochastic global optimization strategies have since been refined, and recent innovations such as topological and homological methods exemplified by SHGO promise even more robust exploration of complex energy surfaces \cite{Endres2018}. Ultimately, incorporating global features of energy surfaces into the optimization framework is essential for both accurate model fitting and innovative materials design.

\subsection{Challenges and motivation}
Many open questions remain in the modelling of phase equilibria. In addition to the numerical challenges that are intrinsic to bilevel optimisation, there are physical-modelling challenges. For example, Olaya \emph{et al.}
\cite{Olaya2008} demonstrated that the popular NRTL and UNIQUAC models do not admit any parameter set that correctly fits a particular class of ternary island systems. Bollas \emph{et al.} \cite{Mitsos2009LLE, Bollas2009VLE} further demonstrated the need for a well-behaved resultant model whose constraints enforce the correct number of phase splits, exclude spurious phases, and ensure that the predicted phases are globally stable in regions of instability. These types of defects are best understood by analysing the underlying Gibbs free-energy surface predicted by the model. As an illustration,
\autoref{fig:octanol-water-tangent-plane} below shows two parameter sets for the same NRTL model that both reproduce the experimental tie-line, yet only one of them represents the true physical solution of the underlying energy surface; the other introduces a spurious additional phase split.

In our reformulation we define the applicability of a given thermodynamic model to a particular physical system more rigorously than has previously been done. Our reformulation also generalises earlier rigorous statements of the problem first formulated by Bollas \emph{et al.} \cite{Mitsos2009LLE, Bollas2009VLE} and most recently extended to cubic equations of state by Glass, Djelassi \& Mitsos~\cite{Glass2018CEOS}, who solve a bilevel program with three nested lower-level problems (a tangent-plane criterion, a mechanical-stability test, and a phase-split-count test) using BARON-backed branch-and-bound at a cost of approximately 4.5\,h per binary case study, a cost that the authors themselves identify as prohibitive for industrial-scale problems 
. In contrast, our reformulation collapses these nested lower-level problems into a single geometric object, the dual manifold, evaluated on a finite set of simplex points, eliminating the need for KKT-based reformulation, derivative information, or commercial deterministic solvers. The same construction applies uniformly to activity-coefficient models, cubic equations of state, and non-cubic equations of state such as PC-SAFT or DWPM, addressing the open extension flagged in the Glass outlook~\cite{Glass2018CEOS}. 
Our approach is a geometric and topological reinterpretation of the problem: rather than solving the \emph{local} equilibrium problem numerically and iteratively perturbing the parameters, we define new geometric objects (illustrated in \autoref{fig:methods-dual-manifold-binary} and \autoref{fig:methods-dual-manifold-ternary}) that allow us to correct the \emph{global} energy surface directly. This use of higher-dimensional geometric objects, in contrast to the scalar gradient information used in conventional gradient-based optimisation, simultaneously eliminates defects in the energy surface and is applicable to any class of model, including the non-smooth and occasionally non-continuous models that are needed to describe increasingly complex chemical systems. The method allows the practitioner to fit mathematically complex thermodynamic models to data reliably and efficiently, and opens the door to rigorous modelling and simulation of problems in computational thermodynamics that were previously intractable. In addition to recovering the optimal parameter set for difficult models efficiently, the method can be used to prove when a particular model cannot fit a given physical system for any parameter set, providing an automatic means of model validation that previously required the tedious analysis of, for example, Olaya \emph{et al.}~\cite{Olaya2008}. We further apply the method to VL(L)E systems modelled with the VdW-DWPM equation of state, which exhibits many of the mathematical difficulties characteristic of modern, mathematically complex thermodynamic models.

\begin{figure}[H]
\centerline{\includegraphics[width=0.85\textwidth]{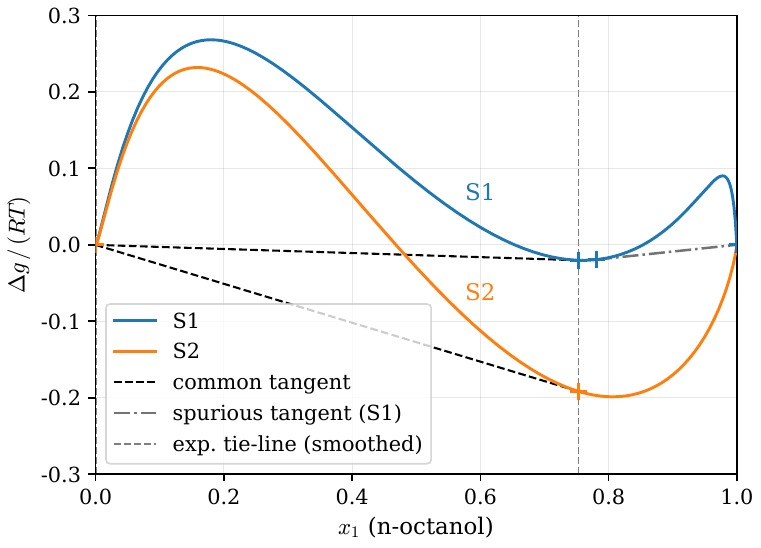}}
\caption{Gibbs tangent-plane diagram for the binary n-octanol\,(1)--water\,(2) at $T = 313.5$\,K using literature values of the NRTL binary interaction parameters. Curves S1 and S2 correspond to two different NRTL parameter sets, and the vertical lines mark the smoothed experimental data of S{\o}rensen and Arlt (1979). Both S1 and S2 reproduce the correct liquid-liquid tie-line at $(x_{\alpha}, x_{\beta}) = (1.44\times10^{-4},\,0.753)$; however, S1 incorrectly predicts an additional spurious phase split spanning from $x \approx 0.78$ to the n-octanol-rich boundary. State-of-the-art bilevel algorithms cannot detect such defects efficiently and become intractable as the number of components grows. Figure computed with the present NRTL implementation following the case study of
\cite{Mitsos2009LLE}. 
\label{fig:octanol-water-tangent-plane}
}
\end{figure}

%% file: sections/methods.tex

We start our development by giving a perspective and brief history of our geometric approach based on the energy surface description of the problem in \autoref{sec:method_esurf}, then we provide an intuitive visual demonstration of the derivation of our method in 1D, followed by its generalization and finally we provide the explicit algorithms for the method in n-component systems.

\subsection{The Energy Surface Perspective: Derivation and History} \label{sec:method_esurf}

The concept of an energy surface $U(\mathbf{x},\mathbf{p})$ is fundamental across the physical sciences, serving as a cornerstone for the analysis of system stability and equilibrium. In the Newtonian picture, $U$ is a potential energy landscape governing the motion of particles, and a configuration $\mathbf{x}$ is in mechanical equilibrium when the force vanishes,
\begin{equation}\label{eq:newton-force-balance}
    \mathbf{F} \;=\; -\nabla_{\mathbf{x}} U(\mathbf{x},\mathbf{p}) \;=\; \mathbf{0}.
\end{equation}
A canonical concrete instance of this view is the Lennard--Jones pair potential
\begin{equation}\label{eq:lj-potential}
    U(x,\mathbf{p}) \;=\; 4\varepsilon\!\left[\left(\frac{\sigma}{x}\right)^{12} - \left(\frac{\sigma}{x}\right)^{6}\right], \qquad \mathbf{p} = (\varepsilon,\sigma),
\end{equation}
whose equilibrium pair separation is the unique zero of $dU/dx$. Equation~\eqref{eq:newton-force-balance} is the simplest, broadest-audience reading of equilibrium as a gradient-zero condition on a parameterised energy surface, and it underlies essentially every classical treatment of mechanical stability.

For a multicomponent thermodynamic system at given pressure $P$ and temperature $T$, however, force balance alone is insufficient. The system also exchanges entropy and matter with its surroundings, so the state variables governing equilibrium include the entropy $S$, the volume $V$, and the mole numbers $\{N_j\}_{j=1}^{n}$. The fundamental thermodynamic equation of Gibbs~\cite{gibbs1878} 
expresses this dependence as
\begin{equation}\label{eq:fundamental-U}
    dU \;=\; T\,dS \;-\; P\,dV \;+\; \sum_{j=1}^{n} \mu_{j}\,dN_{j},
\end{equation}
where $\mu_{j}$ is the chemical potential of component $j$. The Legendre transform $G \defeq U - TS + PV$ defines the Gibbs free energy, whose differential
\begin{equation}\label{eq:fundamental-G}
    dG \;=\; -S\,dT \;+\; V\,dP \;+\; \sum_{j=1}^{n} \mu_{j}\,dN_{j}
\end{equation}
takes $(T,P,\{N_{j}\})$ as its natural variables~\cite{gibbs1878,SmithVanNessAbbott2018}. 
At fixed $(T,P)$ the equilibrium criterion therefore reads $\nabla_{\mathbf{x}} G(\mathbf{x},\mathbf{p}) = \mathbf{0}$ on the composition simplex $\Delta^{n}$, and the chemical potentials
\begin{equation}\label{eq:mu-grad-G}
    \mu_{j} \;=\; \frac{\partial G}{\partial x_{j}}
\end{equation}
are exactly the components of that gradient. The mechanical force-balance condition~\eqref{eq:newton-force-balance} is recovered as the special case in which entropy and composition do not enter, that is when the only state variable left to vary is a mechanical coordinate.

It took the nineteenth-century thermodynamics of Gibbs~\cite{gibbs1878} 
to recognise that for a system held at $(T,P,\{N_{j}\})$ the correct potential to minimise is $G$, not $U$, and that even $\nabla_{\mathbf{x}} G = \mathbf{0}$ is only a \emph{necessary} condition: a stationary composition can correspond to a metastable single phase that admits a lower-energy split. The complementary \emph{sufficient} criterion, formalised algorithmically by Michelsen~\cite{Michelsen1982}, 
requires the tangent plane to $G$ at every candidate equilibrium composition to lie everywhere below the surface on $\Delta^{n}$. The combination of $\nabla_{\mathbf{x}} G = \mathbf{0}$ with the global tangent-plane condition is what distinguishes a true thermodynamic equilibrium from a merely stationary point of the Gibbs surface.

The bilevel parameter-estimation problem of \eqref{eq:bilevel_general} inherits exactly this structure. The lower level is the Gibbs minimisation
\begin{equation}\label{eq:lower-level-G}
    \min_{\mathbf{x}\in\Delta^{n}} G(\mathbf{x},\mathbf{p})
\end{equation}
at fixed $(T,P)$ for a candidate parameter vector $\mathbf{p}$, supplemented by Michelsen's global tangent-plane criterion to exclude metastable solutions. The upper level adjusts $\mathbf{p}$ to fit experimental tie-line data $\{\mathbf{x}^{\alpha,i},\mathbf{x}^{\beta,i}\}_{i=1}^{k}$. Throughout the rest of the manuscript we work with the dimensionless reduced Gibbs free energy
\begin{equation}\label{eq:reduced-g}
    g(\mathbf{x},\mathbf{p}) \;\defeq\; \frac{G(\mathbf{x},\mathbf{p})}{RT},
\end{equation}
which removes the explicit $T$-scaling of equation~\eqref{eq:fundamental-G} and yields a dimensionless objective for the inner problem. The geometric reformulation developed in the remainder of \autoref{sec:methods} replaces the lower-level program~\eqref{eq:lower-level-G}, together with its global tangent-plane sufficiency check, by a single scalar object built directly from $g$ and the experimental tie-line endpoints.


\subsection{Energy gradient based method: a 1D motivation}
\label{subsec:energy-gradient-1d}

Before formalising the approach in \autoref{sec:methods}, we use a low-dimensional toy problem to motivate the central geometric idea. In particular we re-use a polynomial from \cite{Mitsos2007} that was used to demonstrate an idealized binary Gibbs free surface. The continuous Gibbs energy surface $G(\mathbf{x},\mathbf{p})$ is replaced by a finite \emph{discrete surface representation}: a piecewise-linear simplicial complex that is diffeomorphic to the original surface on the composition simplex. Equilibria, tangent-plane intersections, and the phase-split structure are then expressed as combinatorial properties of this complex (faces, dual edges, and the supporting hyperplane between equilibrium vertices), rather than as conditions on derivatives of $G$. As we will show, this geometric representation is what makes it possible to replace the inner Gibbs minimization by an explicit \emph{tangent-plane / dual-manifold} object in the constraints, collapsing the bilevel problem to a single-level non-linear program while still admitting non-smooth and non-continuous $G$.

The development of the method that this manuscript proposes is then demonstrated through a sequence of low-dimensional numerical experiments. The geometric intuition built up in one dimension can be directly generalized to the higher-dimensional formulation.

\subsubsection{Setup and the failure of derivative penalties}
The starting point was the observation that, on a binary Gibbs free energy surface, a phase split is encoded by a single tangent line touching the surface at two equilibrium compositions $\mathbf{x}^{\alpha}, \mathbf{x}^{\beta}$. An attempt was made to enforce equilibrium directly through derivative-based feasibility penalties on $g(\mathbf{x}, \mathbf{p})$ alone, penalising violations of the second-derivative non-negativity constraint and slope-equality constraint via $\max(0,\cdot)$ terms with growing constraint pools $\mathcal{Z}_R$ and $\mathcal{Z}_G$. Tested on the binary n-butanol--water NRTL surface (Heidemann \& Mandhane 1973; S{\o}rensen \& Arlt 1979) with the literature parameters $\tau_{12} = 10.1820, \tau_{21} = 3.80341, \alpha = 0.4$, the penalty mechanism was numerically ill-conditioned and collapsed to the convex midpoint $\mathbf{x} = 0.5076$. The lesson was unambiguous: on the energy surface alone, derivative-based feasibility penalties do not encode the global tangent-plane criterion.

\subsubsection{Line-trace formulation on a clean toy surface}
The formulation was rebuilt on a polynomial double-well
\begin{equation}\label{eq:toy-double-well}
f(\mathbf{x}) \;=\; 0.01\,\mathbf{x} \,+\, (\mathbf{x}-0.5)^4 \,-\, 0.2\,(\mathbf{x}-0.5)^2 \,-\, 0.0125,
\end{equation}
whose analytic equilibrium tie-line is $\mathbf{x}^{\alpha,\beta} \approx [0.1758,\,0.8242]$\footnote{This value was changed from the $[0.1835,\,0.8163]$ reported in the Mitsos paper, as the present solver implementation locates a lower solution at the value quoted here.}. The construction comprises three steps:
\begin{enumerate}
    \item Build the affine interpolant $h$ between the equilibrium endpoints, $h(\mathbf{x}^{\alpha}) = g(\mathbf{x}^{\alpha},\mathbf{p}),\; h(\mathbf{x}^{\beta}) = g(\mathbf{x}^{\beta},\mathbf{p})$, parametrised as $\mathbf{x}(\lambda) = \mathbf{x}^{\alpha} + \lambda(\mathbf{x}^{\beta}-\mathbf{x}^{\alpha})$ with $\lambda \in \mathbb{R}$ allowed to leave $[0,1]$.
    \item At four sample points just inside and just outside each anchor, $\lambda \in \{-\epsilon,\, +\epsilon,\, 1-\epsilon,\, 1+\epsilon\}$, evaluate the residual $r_k = g(\mathbf{x}(\lambda_k),\mathbf{p}) - h(\mathbf{x}(\lambda_k))$ and keep only the negative part:
    \begin{equation}\label{eq:negpart}
    [r]_{-} \;\defeq\; \min(r,\,0),
    \end{equation}
    so excursions where the surface dips \emph{below} the chord are penalised but the convex bulk of the surface is not.
    \item Aggregate the penalised samples into a single non-negative scalar $\Phi_{\text{tangent}}(\mathbf{p}) = \sum_k |[r_k]_-|$.
\end{enumerate}
At a true tangent-plane equilibrium the surface lies below the chord on the inside of the tie-line and above it on the outside, so the four directional samples should bracket zero with the correct sign on each side. Penalising only the wrong-sided excursions yields a non-negative objective with a unique zero at the correct $\mathbf{p}$. With one parameter (the prefactor $0.2$ in \eqref{eq:toy-double-well} replaced by $\mathbf{p}$), \texttt{shgo} \cite{Endres2018} recovers $\mathbf{p}^{\star} \approx 0.21$ with $\Phi_{\text{tangent}} = 0$; with two parameters (the additional centre $0.5 \to p_2$), it recovers $(p_1, p_2) = (0.215, 0.500)$. The 3D objective surface over $(p_1, p_2)$ is smooth and single-welled. The geometry of the construction is shown in \autoref{fig:methods-toy-double-well-line-trace}.

\input{figures/methods-toy-double-well-line-trace/methods-toy-double-well-line-trace}

\subsubsection{Sign change at the optimum and the squaring fix}
A further numerical observation in notebook iv was that the unsquared residual
$\sum_k r_k$ \emph{changes sign} across the optimum rather than touching zero, so a naive line-search overshoots. Squaring the aggregated residual,
\begin{equation}\label{eq:single-tieline-objective}
J(\mathbf{p}) \;=\; \left(\,\frac{1}{\epsilon}\sum_k\,|[r_k]_-|\,\right)^{2},
\end{equation}
produces a convex-looking single-well objective whose global minimum at the correct $\mathbf{p}$ is rendered identifiable to a global solver. The $1/\epsilon$ rescaling promotes the directional samples to a finite-difference approximation of the directional derivative $(\partial g/\partial \mathbf{d} - \partial h/\partial \mathbf{d})|_{\mathbf{x}^{\alpha,\beta}}$, where $\mathbf{d} = \mathbf{x}^{\beta} - \mathbf{x}^{\alpha}$. The squared form is used unchanged in the final formulation. \autoref{fig:methods-squared-objective-1d} contrasts the squared and unsquared objectives in the one-parameter case.

\input{figures/methods-squared-objective-1d/methods-squared-objective-1d}

\subsubsection{Generalisation: the dual manifold}
The 1D line-trace argument extends to multi-component systems by replacing the chord with a hyperplane, and to multi-tie-line datasets by combining hyperplanes. We promote the affine interpolant to a callable object $h_i \colon \mathbb{R}^n \to \mathbb{R}$ associated with each tie-line $i$, and define the supporting hyperplane over the simplex as the point-wise minimum:
\begin{equation}\label{eq:hsum}
h(\mathbf{x}) \;=\; \min_{i = 1,\dots,k} h_i(\mathbf{x}).
\end{equation}
The \emph{dual manifold} is then the shifted free-energy surface
\begin{equation}\label{eq:dual-manifold}
\mathcal{M}(\mathbf{x};\mathbf{p}) \;\defeq\; g(\mathbf{x},\mathbf{p}) - h(\mathbf{x}),
\end{equation}
which by construction satisfies $\mathcal{M}(\mathbf{x}^{\alpha,i};\mathbf{p}) = \mathcal{M}(\mathbf{x}^{\beta,i};\mathbf{p}) = 0$ at every experimental tie-line endpoint, and which we conjecture satisfies $\mathcal{M}(\mathbf{x};\mathbf{p}) \geq 0$ over the simplex if and only if $\mathbf{p}$ reproduces every datum as a globally stable equilibrium under the Gibbs tangent-plane criterion. A negative region of $\mathcal{M}$ is a topological defect: a parameter set that fits the data points themselves but violates global stability somewhere else. \autoref{fig:methods-dual-manifold-binary} shows the binary dual manifold and \autoref{fig:methods-dual-manifold-ternary} the ternary generalisation.

\input{figures/methods-dual-manifold-binary/methods-dual-manifold-binary}

\input{figures/methods-dual-manifold-ternary/methods-dual-manifold-ternary}

\paragraph{The five-component objective function}
The line-trace formulation in 1D generalises to a single-level objective by augmenting $J(\mathbf{p})$ with terms that detect and eliminate topological defects on $\mathcal{M}$. A first inner global minimisation locates the set of defective minima
\begin{equation}\label{eq:XDmin}
\mathcal{X}_D^{-}(\mathbf{p}) \;=\; \arg\min_{\mathbf{x} \in \Delta^{n}} \mathcal{M}(\mathbf{x};\mathbf{p}) \;\setminus\; \{\mathbf{x}^{\alpha,i},\mathbf{x}^{\beta,i}\}_{i=1}^{k},
\end{equation}
and a second locates the corresponding maxima $\mathcal{X}_D^{+}(\mathbf{p}) = \arg\min_{\mathbf{x}} (-\mathcal{M}(\mathbf{x};\mathbf{p}))$. Both are solved with SHGO \cite{Endres2018} on the simplex with a sum-to-one nonlinear constraint. The overall objective decomposes into five named components:
\begin{align}
\Phi_{\text{tangent}}(\mathbf{p}) & = \sum_{i=1}^{k} \sum_{k=1}^{4} \bigl|\bigl[g(\mathbf{x}_k^i,\mathbf{p}) - h_i(\mathbf{x}_k^i)\bigr]_{-}\bigr|, & \text{tangent-plane sub-sampling}\nonumber\\
\Phi_{\text{plane}}(\mathbf{p}) & = \sum_{\mathbf{x} \in \mathcal{X}_D^{-}} \bigl|\bigl[g(\mathbf{x},\mathbf{p}) - h(\mathbf{x})\bigr]_{-}\bigr|, & \text{plane violation at defects}\nonumber\\
\Phi_{\text{topo}}(\mathbf{p}) & = \sum_{\mathbf{x} \in \mathcal{X}_D^{-}} \bigl\|\mathbf{x}_{\text{near}} - \mathbf{x}\bigr\|_{2}, & \text{topological pull to equilibrium}\nonumber\\
\Phi_{\text{local}}(\mathbf{p}) & = \sum_{\mathbf{x} \in \mathcal{X}_D^{-}} \bigl\|(\mathbf{x}^{+}_{\text{near}}, g(\mathbf{x}^{+}_{\text{near}})) - (\mathbf{x}, g(\mathbf{x}))\bigr\|_{2}, & \text{local-topology amplitude}\nonumber\\
\Phi_{\text{phase}}(\mathbf{p}) & : \text{partition of $\arg\min \mathcal{M}$ into $\mathcal{X}_D^{-}$ vs.\ equilibrium endpoints}, & \text{phase assignment}\label{eq:five-components}
\end{align}
where $\mathbf{x}_{\text{near}} = \arg\min_{\mathbf{y} \in \{\mathbf{x}^{\alpha,i},\mathbf{x}^{\beta,i}\}} \|\mathbf{y}-\mathbf{x}\|_2$ and $\mathbf{x}^{+}_{\text{near}}$ is the nearest defective maximum in $\mathcal{X}_D^{+}$ measured in the joint composition--energy space $(\mathbf{x}, g(\mathbf{x},\mathbf{p}))$. Geometrically each pair $(\mathbf{x}, \mathbf{x}^{+}_{\text{near}})$ is the minimum and the bounding maximum of one connected component of the negative region of $\mathcal{M}$, so $\Phi_{\text{local}}$ measures how deep and how wide that component is, and driving it to zero removes the component rather than merely relocating it. The four samples in $\Phi_{\text{tangent}}$ are exactly the four-point line-trace of \eqref{eq:single-tieline-objective}.

The total objective for a multi-temperature isobaric dataset is
\begin{equation}\label{eq:phi-total}
\Phi(\mathbf{p}) \;=\; \sum_{T \in \mathcal{D}} \left(\,\frac{1}{\epsilon}\, \bigl[\Phi_{\text{tangent}} + \Phi_{\text{plane}} + \Phi_{\text{topo}} + \Phi_{\text{local}}\bigr](\mathbf{p};\,T,\mathbf{x}^{\alpha},\mathbf{x}^{\beta})\,\right)^{2},
\end{equation}
which is a single-level non-linear program in $\mathbf{p}$ alone. The bilevel structure has been replaced by a scalar function whose evaluation requires two SHGO inner calls (for $\mathcal{X}_D^{\pm}$) per temperature, but no KKT system, no Lagrange multipliers, and no dual cuts. \autoref{fig:methods-five-component-decomposition} shows the contributions of the five components on a one-parameter sweep, and \autoref{fig:methods-defect-sweep} shows directly why the inner SHGO call is necessary: holding the tie-line endpoints fixed and varying only the depth of a Gaussian defect placed midway between them leaves the closed-form $\Phi_{\text{tangent}}$ at machine noise, while only the defect-aware terms $\Phi_{\text{plane}}, \Phi_{\text{topo}}, \Phi_{\text{local}}$ register the defect.

\input{figures/methods-five-component-decomposition/methods-five-component-decomposition}

\input{figures/methods-defect-sweep/methods-defect-sweep}

\paragraph{Two-stage solver.}
The defect-detection terms $\Phi_{\text{plane}}, \Phi_{\text{topo}}, \Phi_{\text{local}}$ require an inner SHGO call per evaluation, which is expensive. We therefore split the outer optimisation into two stages~\cite{Endres2018}:
\begin{enumerate}
    \item \emph{Tangent-plane pass.} Run SHGO on $\sum_T (\Phi_{\text{tangent}}/\epsilon)^2$ alone (no inner SHGO). Default sampling density: $N_{\text{tan}} = 1000$.
    \item \emph{Defect-detection.} On the candidate $\mathbf{p}$, build $\mathcal{M}$ at every $T$ and run SHGO once on the simplex. If any minimum survives the equilibrium-endpoint stripping ($\mathcal{X}_D^{-}\neq\emptyset$), defects are present.
    \item \emph{Defect-aware refit.} If defects exist, restart SHGO on the full $\Phi(\mathbf{p})$ of \eqref{eq:phi-total}. Default inner sampling: $N_{\text{def}} = 20$.
    \item Re-check defects; loop up to a fixed maximum number of outer iterations.
\end{enumerate}
The two-stage solver collapses to the single tangent-plane pass for well-behaved systems, and only pays the cost of the inner SHGO when the topological constraints are actually binding.

A complete algorithmic specification of the tangent-plane pass, the defect-detection routine, the defect-aware objective, and the master two-stage solver is given in \ref{appendix:algorithms}.

\paragraph{Verification on toy and real data.}
The formulation reproduces the Mitsos LLE benchmark on the toy double-well of \eqref{eq:toy-double-well} with $\Phi(\mathbf{p}) = 0$ at the analytic optimum, finds the toluene--water--aniline ternary equilibrium of \cite{Mitsos2007} to within $5\times 10^{-3}$ in the free parameter $\tau_{12}$, and on the n-butyl-acetate--water binary LLE case (Case 1 of \cite{Mitsos2009LLE}) achieves the reference relative-error metrics. \autoref{fig:methods-defect-elimination-solution} shows the dual manifold before and after defect-aware fitting on the toy surface; the four LLE benchmark systems are reported in detail in \autoref{sec:results-lle} and Supplementary Material~S1.

\input{figures/methods-defect-elimination-solution/methods-defect-elimination-solution}

\paragraph{Proof obligations.}
The formulation rests on three statements, all of which are discharged in the appendices, two of them unconditionally and one under an explicitly stated assumption:
\begin{enumerate}
    \item The conjecture $\mathcal{M}(\mathbf{x};\mathbf{p}) \geq 0\ \forall \mathbf{x} \in \Delta^{n} \iff \mathbf{p}$ is a globally consistent fit. Proved as \autoref{theor:M-nonneg-iff-fit} in \autoref{appendix:proof-conjecture-M}, for continuous $g$ and mutually consistent tie-line data, using the definitions and regularity results collected in \autoref{appendix:older-formulation}.
    \item Vanishing of $\Phi(\mathbf{p})$ as a sufficient condition for global stability at every datum. Proved as \autoref{theor:phi-zero-sufficient} in \autoref{appendix:phi-zero-proof}. The argument is a compactness argument on the exact defective-minima set \eqref{eq:XDmin} and assumes nothing about the inner solver; what the inner SHGO call contributes is isolated separately in Lemma~\ref{lem:inner-shgo-exact}.
    \item Convergence of the two-stage outer loop given that $\mathcal{X}_D^{\pm}$ is discontinuous in $\mathbf{p}$ at bifurcations of the surface topology (so $\Phi$ is not smooth at those points). Addressed in \autoref{appendix:convergence-proof} in two parts, described next.
\end{enumerate}
The third obligation separates into an unconditional and a conditional half. Whenever a parameter vector with $\Phi(\mathbf{p}^{\star}) = 0$ is returned, it is certified globally optimal by the non-negativity of $\Phi$ alone (\autoref{prop:zero-certificate}) and, by \autoref{theor:phi-zero-sufficient}, reproduces every measured tie line as a globally stable equilibrium; that certificate is verifiable a posteriori from the returned objective value and does not depend on the solver. That such a point is found whenever one exists inherits the deterministic adequate-sampling guarantees of SHGO (\cite{Endres2018}, Theorems~3 and~5), 
which is the strongest form of guarantee available for a black-box objective and is deterministic rather than probabilistic. The one step that remains an assumption rather than a proof is that the bifurcation strata of $\Phi$ carry zero measure; it is stated as Assumption~1 of \autoref{appendix:convergence-proof} and justified there for real-analytic model families.

%% file: figures/methods-toy-double-well-line-trace/methods-toy-double-well-line-trace.tex
\begin{figure}[H]
    \centering
    \includegraphics[width=0.85\linewidth]{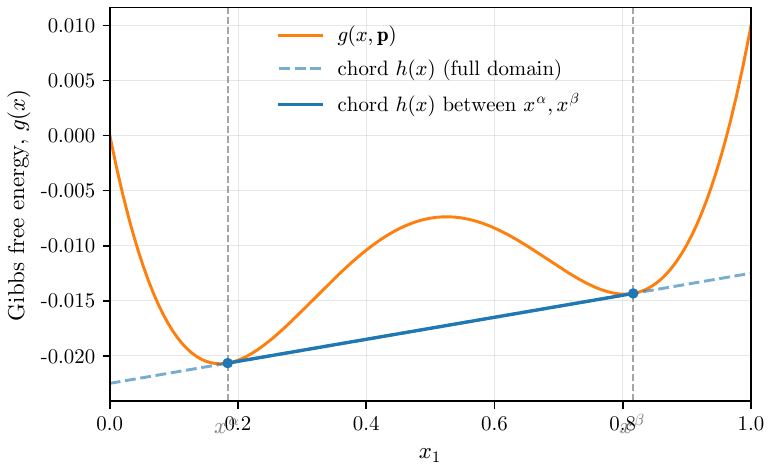}
    \caption{Toy polynomial double-well
    $g(x) = 0.01 x + (x-0.5)^{4} - 0.2 (x-0.5)^{2} - 0.0125$
    with the equilibrium endpoints
    $\mathbf{x}^{\alpha} \approx 0.1838$,
    $\mathbf{x}^{\beta}  \approx 0.8162$
    and the affine chord $h(\mathbf{x})$ between them
    (solid segment, with the dashed extension into the bulk).
    The line-trace formulation samples the residual
    $r_k = g(\mathbf{x}_k) - h(\mathbf{x}_k)$ at
    $\lambda \in \{-\epsilon, +\epsilon, 1-\epsilon, 1+\epsilon\}$
    and aggregates only the negative-side excursions
    (\autoref{eq:negpart}) into
    $\Phi_{\text{tangent}}(\mathbf{p})$.}
    \label{fig:methods-toy-double-well-line-trace}
\end{figure}

%% file: figures/methods-squared-objective-1d/methods-squared-objective-1d.tex
\begin{figure}[H]
    \centering
    \includegraphics[width=0.85\linewidth]{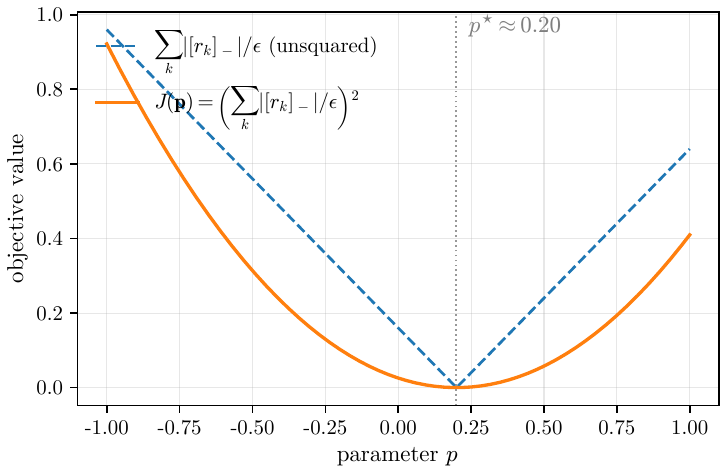}
    \caption{Squared single-tie-line objective
    $J(\mathbf{p}) = \left(\frac{1}{\epsilon}\sum_k |[r_k]_-|\right)^{2}$
    of \autoref{eq:single-tieline-objective}, swept over the single free
    parameter of the toy double-well
    $g(\mathbf{x}; \mathbf{p}) = 0.01\mathbf{x} + (\mathbf{x}-0.5)^{4}
    - \mathbf{p}\,(\mathbf{x}-0.5)^{2} - 0.0125$
    with $\mathbf{x}^{\alpha,\beta} = [0.1838, 0.8163]$ and
    $\epsilon = 10^{-7}$. The unsquared residual changes sign across
    the optimum; squaring renders the objective a convex-looking
    single-well function that is identifiable to a global solver, with
    $J(\mathbf{p}^{\star}) = 0$ at $\mathbf{p}^{\star} \approx 0.2$.}
    \label{fig:methods-squared-objective-1d}
\end{figure}

%% file: figures/methods-dual-manifold-binary/methods-dual-manifold-binary.tex
\begin{figure}[H]
    \centering
    \includegraphics[width=0.85\linewidth]{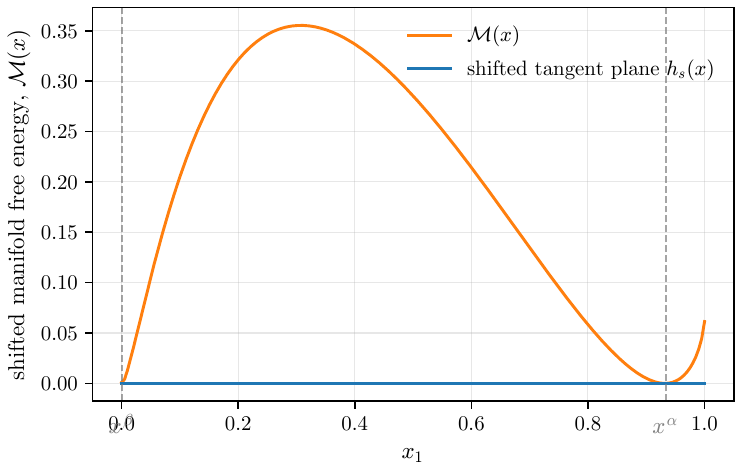}
    \caption{Binary dual manifold
    $\mathcal{M}(\mathbf{x}; \mathbf{p}) =
    g(\mathbf{x}, \mathbf{p}) - h(\mathbf{x})$
    of \autoref{eq:dual-manifold} for the quartic-plus-Gaussian-dip
    toy, evaluated with
    $\mathbf{x}^{\alpha,\beta} = [0.1838, 0.8163]$.
    The orange curve is $\mathcal{M}(\mathbf{x})$ and the blue line is
    the shifted hyperplane $h_s(\mathbf{x}) = 0$. By construction
    $\mathcal{M}(\mathbf{x}^{\alpha,i}; \mathbf{p}) =
    \mathcal{M}(\mathbf{x}^{\beta,i}; \mathbf{p}) = 0$, and
    $\mathcal{M} \geq 0$ over the simplex precisely when $\mathbf{p}$
    reproduces every datum as a globally stable equilibrium under the
    tangent-plane criterion.}
    \label{fig:methods-dual-manifold-binary}
\end{figure}

%% file: figures/methods-dual-manifold-ternary/methods-dual-manifold-ternary.tex
\begin{figure}[H]
    \centering
    \includegraphics[width=0.7\linewidth]{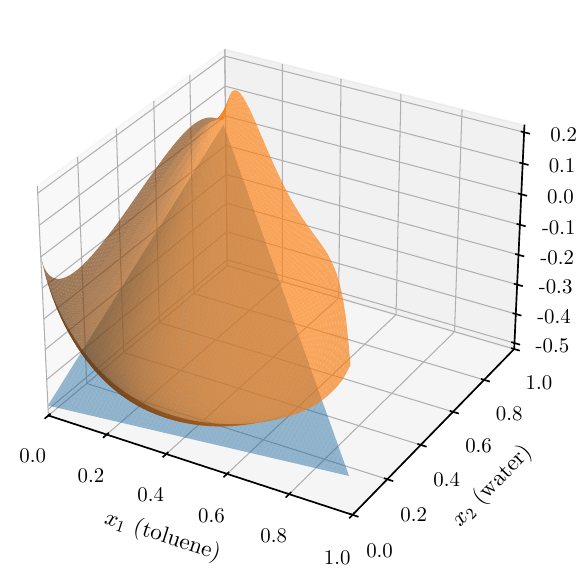}
    \caption{Ternary generalisation of the binary chord construction
    (\autoref{fig:methods-dual-manifold-binary}). The orange surface
    is the NRTL Gibbs free energy
    $g(\mathbf{x}, \mathbf{p})$ on the ternary simplex
    (toluene--water--aniline test system) and the blue plane is the
    supporting hyperplane $h(\mathbf{x})$ of
    \autoref{eq:hsum} that anchors the experimental tie-line
    endpoints. The dual manifold
    $\mathcal{M}(\mathbf{x}; \mathbf{p})$ of
    \autoref{eq:dual-manifold} is the orange surface measured
    relative to this hyperplane.}
    \label{fig:methods-dual-manifold-ternary}
\end{figure}

%% file: figures/methods-five-component-decomposition/methods-five-component-decomposition.tex
\begin{figure}[H]
    \centering
    \includegraphics[width=0.85\linewidth]{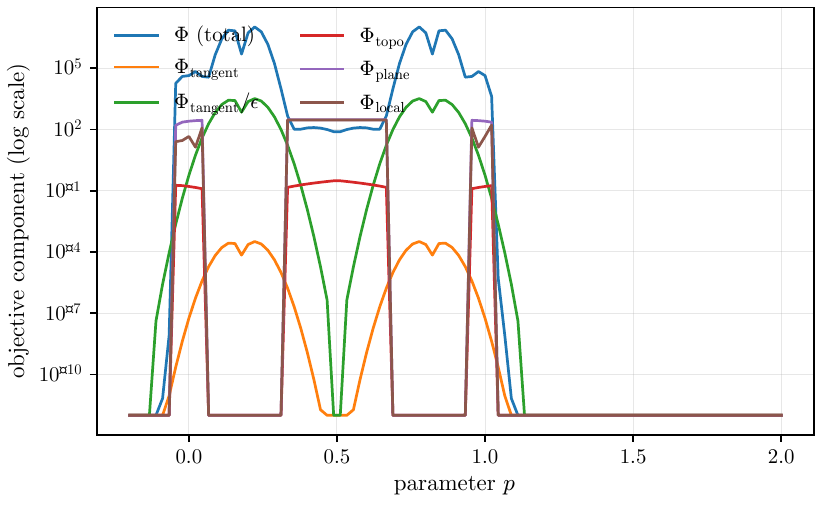}
    \caption{Decomposition of the total objective $\Phi(\mathbf{p})$
    of \autoref{eq:phi-total} into the five named components of
    \autoref{eq:five-components}, swept over the single free parameter
    $\mathbf{p}$ of the quartic-plus-Gaussian-dip toy on a
    semilog-$y$ scale. The tangent-plane violation
    $\Phi_{\text{tangent}}$ dominates away from the optimum;
    $\Phi_{\text{plane}}$, $\Phi_{\text{topo}}$ and
    $\Phi_{\text{local}}$ become active only inside the
    defect-bearing parameter region, illustrating the role of the
    inner SHGO call in detecting topological defects on
    $\mathcal{M}$.}
    \label{fig:methods-five-component-decomposition}
\end{figure}

%% file: figures/methods-defect-sweep/methods-defect-sweep.tex
%
\begin{figure}[H]
    \centering
    \includegraphics[width=\textwidth]{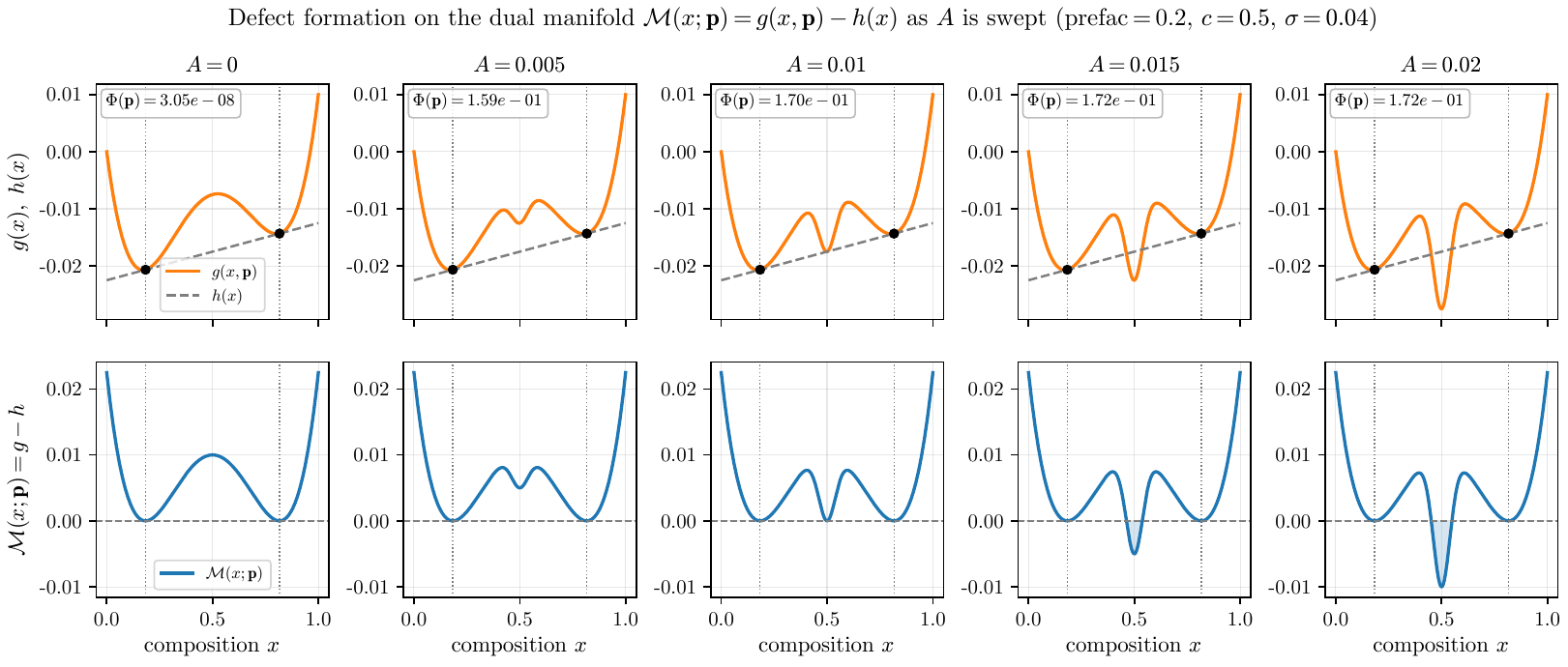}
    \caption[Defect formation on the dual manifold under a parameter sweep]{%
        Defect formation on the dual manifold
        $\mathcal{M}(\mathbf{x};\mathbf{p}) = g(\mathbf{x},\mathbf{p}) - h(\mathbf{x})$
        as the Gaussian-dip amplitude $A$ is swept on the toy surface
        $g_{\text{gauss}}(x;\,\text{prefac}=0.2, A, c=0.5, \sigma=0.04)$
        of \texttt{param.thermo.toy.g\_gauss\_full}, with the corrected Mitsos
        tie-line endpoints $\mathbf{x}^{\alpha,\beta} = [0.1835,\,0.8163]$.
        \emph{Top row}: the Gibbs free energy $g(x,\mathbf{p})$ (orange) and
        the supporting hyperplane $h(x)$ (gray dashed) through the tie-line
        endpoints (black markers). \emph{Bottom row}: the dual manifold
        $\mathcal{M}(x;\mathbf{p})$ (blue) with the zero-line dashed gray;
        regions where $\mathcal{M}<0$ are shaded. At $A=0$ the dual manifold
        is non-negative on the whole simplex (no topological defect, see
        \eqref{eq:dual-manifold}); for $A>0$ a negative basin opens at
        $x = c = 0.5$, signalling a parameter set that fits the tie-line
        endpoints exactly while violating the global tangent-plane criterion
        elsewhere. The reported $\Phi(\mathbf{p})$ is the full five-component
        defect-aware objective of \eqref{eq:phi-total}, evaluated by
        \texttt{isobaric\_obj\_sum\_with\_defects} with one inner SHGO call
        on the simplex; the closed-form four-point line trace
        $\sum_k\,|[r_k]_-|^2$ alone (\texttt{isobaric\_obj\_sum\_no\_defects})
        remains $\sim 3\times10^{-8}$ across the sweep because its sampling
        points coincide with the tie-line endpoints and miss the midpoint
        dip, precisely the motivation for the additional
        $\Phi_{\text{plane}},\Phi_{\text{topo}},\Phi_{\text{local}}$ terms.
    }
    \label{fig:methods-defect-sweep}
\end{figure}

%% file: figures/methods-defect-elimination-solution/methods-defect-elimination-solution.tex
\begin{figure}[H]
    \centering
    \includegraphics[width=0.85\linewidth]{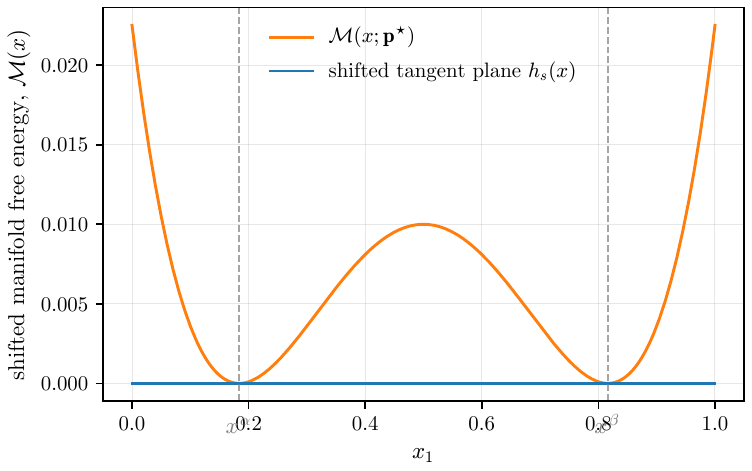}
    \caption{Verification on the toy benchmark: dual manifold
    $\mathcal{M}(\mathbf{x}; \mathbf{p}^{\star})$ at the
    solver-recovered parameter $\mathbf{p}^{\star}$ after minimising
    the full five-component objective $\Phi$
    (\autoref{eq:phi-total}). The manifold touches zero at both
    experimental tie-line endpoints
    $\mathbf{x}^{\alpha}, \mathbf{x}^{\beta}$ and is non-negative
    everywhere on the simplex, so
    $\mathcal{X}_D^{-}(\mathbf{p}^{\star}) = \emptyset$ and the
    recovered parameter satisfies the global tangent-plane criterion.}
    \label{fig:methods-defect-elimination-solution}
\end{figure}

%% file: sections/results.tex
%
%
%
%

The results below form a single argument in three steps: the bilevel
formulation is validated against published literature fits, the
equation-of-state inner solver is validated against a published bilevel
study, and the two are then combined on three multicomponent systems of
industrial consequence.

\paragraph{Step 1: the formulation, against literature LLE fits.}
\autoref{sec:results-lle} regresses NRTL parameters for the four binary
liquid--liquid equilibrium systems of Mitsos, Bollas and
Barton~\cite{Mitsos2009LLE}
and compares the fits, tie line by tie line, against their published
parameter sets. These four cases are the direct literature comparison
for the geometric formulation of \autoref{sec:methods}: the models are
activity-coefficient models with a single Gibbs energy branch, the
experimental tie lines are well tabulated, and an independently
published optimum exists for every case, so any disagreement is
attributable to the formulation rather than to the data.

\paragraph{Step 2: the inner solver, against a published CEOS bilevel
study.} The three industrial cases that follow are all cubic-equation-of-state
fits, for which the inner problem acquires a second Gibbs energy branch
and a root-selection step that the LLE cases never exercise. That
machinery is validated separately, in
\autoref{appendix:ceos-validation-sec}, by replicating the
proof-of-concept case study of Glass, Djelassi and
Mitsos~\cite{Glass2018CEOS}: the isothermal VLE of C$_5$H$_{12}$/H$_2$S
of Reamer, Sage and Lacey~\cite{Reamer1953}, regressed with both
Peng--Robinson and Soave--Redlich--Kwong.
We recover their reported binary interaction parameters to within
$1.65\,\%$ (PR) and $0.55\,\%$ (SRK), and reproduce both of the figures
they publish for that system. The replication is placed in an appendix
because it establishes agreement on an already-solved problem; it is a
prerequisite for, not an instance of, the results reported here.

\paragraph{Step 3: three industrial multicomponent applications.}
\autoref{sec:results-sour-gas} (the sour-gas ternary
CH$_4$ + CO$_2$ + H$_2$S), \autoref{sec:results-ccs} (the carbon-capture
quaternary CO$_2$ + N$_2$ + Ar + O$_2$) and
\autoref{sec:results-refrigerant} (the zeotropic refrigerant blend
R-407C, R-32 + R-125 + R-134a) share one workflow: the binary
interaction parameters are regressed on binary subsystems only, the
multicomponent phase behavior is then predicted by extrapolation with
no additional fitted parameters, and every tie line, binary and
multicomponent alike, is certified against the dual-manifold stability
criterion of \autoref{sec:methods}. The three sections report the
resulting fits, the out-of-sample extrapolation error, and the complete
per-tie-line defect record, including the tie lines on which defects
survive.

This last step is where the present formulation is meant to earn its
keep. Glass, Djelassi and Mitsos close their study by stating that
industrial case studies do not appear computationally tractable for
their method for as long as it depends on general-purpose commercial
solvers~\cite{Glass2018CEOS},
having earlier defined industrial problems as those carrying an
extensive measurement data set and/or a larger number of species, and
having reported computation times of approximately $4.5$~h for each of
the two binary case studies they were able to complete.
The three case studies below are industrial in exactly that sense: they
carry three and four species rather than two, and add between $18$ and
$44$ multicomponent tie lines on top of the binary regression sets. The
claim made here is therefore a tractability claim as much as an accuracy
one, and it is made against the specific limitation those authors
identified.

%% file: sections/results_lle.tex

\subsection{Validation against the Mitsos (2009) LLE parameter set}
\label{sec:results-lle}

The four binary liquid--liquid equilibrium (LLE) systems introduced by
Mitsos, Bollas and Barton~\cite{Mitsos2009LLE}
are used here as a literature-comparison benchmark for the geometric
bilevel formulation developed in \autoref{sec:methods}. Each system is
strongly non-ideal (three aqueous binaries plus one
hydrocarbon--furfural pair), the experimental tie lines are
well-tabulated, and the published NRTL parameter sets serve as
reference fits from an established temperature-dependent NRTL
formulation~\cite{Mitsos2009LLE}.

The phase envelopes plotted in
\autoref{fig:results-lle-butyl-acetate-water-envelope} through
\autoref{fig:results-lle-furfural-tmh-envelope} are computed
\emph{at the published Mitsos parameters} of~\cite{Mitsos2009LLE}
Table~1 using the classical direct common-tangent construction for
binary phase equilibrium~\cite{SmithVanNessAbbott2018},
which solves
$g'(x_\text{I}) = g'(x_\text{II}) = (g(x_\text{II}) - g(x_\text{I}))/(x_\text{II} - x_\text{I})$
by Nelder--Mead at each temperature on a uniform grid. The intent is
to (i) verify that the present NRTL implementation reproduces the
\cite{Mitsos2009LLE} model when fed their published coefficients, and
(ii) place the bilevel fit produced by the present method
(\autoref{sec:methods}, with the SHGO global solver of
Endres et al.~\cite{Endres2018} for the inner Gibbs minimisation)
alongside those literature values for direct comparison. The NRTL
coefficient ordering
$\mathbf{p} = (A_{12}, B_{12}, C_{12}, A_{21}, B_{21}, C_{21})$
is identical to that of~\cite{Mitsos2009LLE}, with
$\tau_{ij} = A_{ij} + B_{ij}(T_{\text{ref}}/T - 1) + C_{ij}(T/T_{\text{ref}} - 1)$
and $T_{\text{ref}} = 298.15$~K.
Three error metrics are reported per case: the absolute
least-squares residual $\epsilon_{a}$ and the per-phase
relative errors $\epsilon_{r,\text{I}}$ and
$\epsilon_{r,\text{II}}$, matching the ``Error'',
``Rel-error$_1$'' and ``Rel-error$_2$'' columns of
\cite{Mitsos2009LLE} Table~1.\footnote{The Mitsos error values
quoted in this section are the values published
in~\cite{Mitsos2009LLE} Table~1. Re-evaluating the published
parameters with the present implementation produces slightly
different absolute error values; in the absence of access to the
exact equilibrium-solver tolerances of the original study, only the
published values are used as the comparator throughout.}
Per-temperature Gibbs and dual-manifold plots for every tie line are
collected in Supplementary Material~S1;
\autoref{fig:results-lle-example-per-point} reproduces one
representative pair from that archive.

\begin{figure}[ht]
\centering
\begin{minipage}[t]{0.48\textwidth}
  \centering
  \includegraphics[width=\linewidth]{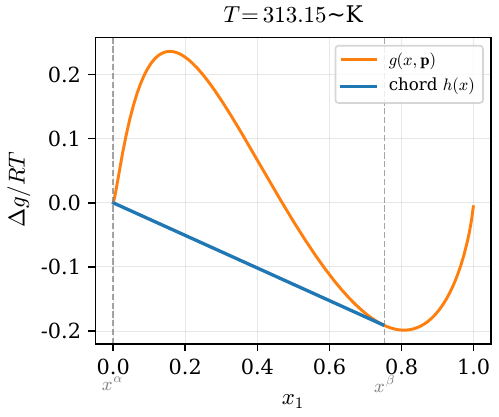}
\end{minipage}
\hfill
\begin{minipage}[t]{0.48\textwidth}
  \centering
  \includegraphics[width=\linewidth]{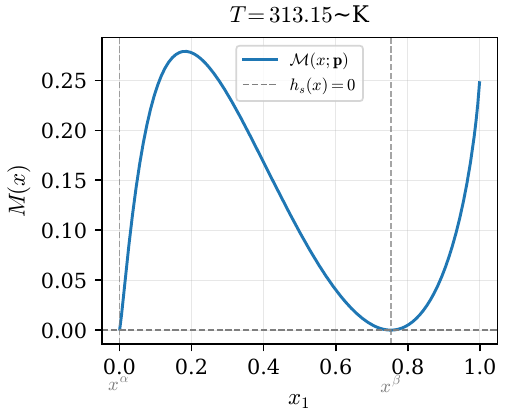}
\end{minipage}
\caption{Representative per-temperature stability record from the
Supplementary Material~S1 archive: Case~3 (n-octanol -- water) at
$T = 313.15$~K, evaluated at the bilevel-fitted
parameters~$\mathbf{p}_{\mathrm{fit}}$ of this work, that is, the
Case~3 bilevel coefficients tabulated below. Left:
Gibbs free energy surface
$g(x, \mathbf{p}_{\mathrm{fit}})$ with the experimental tie line and its
common-tangent chord; right: dual manifold
$\mathcal{M}(x; \mathbf{p}_{\mathrm{fit}}) = g - h$, whose negative part
$[\mathcal{M}]_{-}$ enters the bilevel objective. Analogous pairs for
every experimental temperature of all four cases are collected in
Supplementary Material~S1.}
\label{fig:results-lle-example-per-point}
\end{figure}

\subsubsection{Case 1: n-butyl-acetate -- water}
\label{sec:results-lle-butyl-acetate-water}

The first case study draws on six tie lines from Macedo and Rasmussen
(1987)~\cite{Mitsos2009LLE}
spanning $T = 298.15$--$343.15$~K. The mutual solubility is dilute on
both sides ($x_1 \in [8.25 \times 10^{-4}, 9.45 \times 10^{-4}]$ in the
water-rich phase and $x_1 \in [0.881, 0.934]$ in the acetate-rich
phase), the non-randomness factor is fixed at $\alpha = 0.2$, and all
six NRTL coefficients are tunable on the default
bounds~\cite{Mitsos2009LLE}.
\autoref{fig:results-lle-butyl-acetate-water-envelope} shows the
phase envelope generated by the literature parameters: 50 of 50 grid
temperatures yield a two-phase common-tangent solution, and the
envelope traces the experimental tie lines closely.

\input{figures/results-lle-butyl-acetate-water-envelope/results-lle-butyl-acetate-water-envelope_figure}

\autoref{tab:lle-butyl-acetate-water-params} gives the published
Mitsos coefficients alongside the bilevel-fitted values. Both fits
agree on the temperature-independent terms ($A_{12}$, $A_{21}$) to
within $1.4\%$, but diverge sharply on the $B_{ij}$ and $C_{ij}$
coefficients, a known degeneracy of NRTL temperature dependence
when the data span only a narrow $\Delta T$~\cite{Mitsos2009LLE}.

\begin{table}[ht]
\centering
\caption{NRTL parameters for n-butyl-acetate (1) -- water (2),
$\alpha = 0.2$. The Mitsos column is taken verbatim
from~\cite{Mitsos2009LLE} Table~1, row label~1; the This-work column
is the bilevel solution produced by the method of
\autoref{sec:methods}. The percentage difference is relative to the
Mitsos value.}
\label{tab:lle-butyl-acetate-water-params}
\begin{tabular}{lrrr}
\hline
Coefficient & Mitsos (2009) & This work (bilevel) & Diff.\ \% \\
\hline
$A_{12}$ & $\phantom{-}1.14$  & $\phantom{-}1.1554$ & $+1.4$    \\
$B_{12}$ & $-11.6$            & $-4.4617$           & $+61.5$   \\
$C_{12}$ & $-13.5$            & $-7.3273$           & $+45.7$   \\
$A_{21}$ & $\phantom{-}6.03$  & $\phantom{-}6.0507$ & $+0.3$    \\
$B_{21}$ & $-14.9$            & $-2.2449$           & $+84.9$   \\
$C_{21}$ & $-10.5$            & $\phantom{-}0.6592$ & $+106.3$  \\
\hline
\end{tabular}
\end{table}

\autoref{tab:lle-butyl-acetate-water-errors} compares the three
abs/rel error metrics. The bilevel fit attains an absolute residual
of $2.58 \times 10^{-5}$, in the same order of magnitude as the
$6 \times 10^{-5}$ reported by~\cite{Mitsos2009LLE} for their fit.

\begin{table}[ht]
\centering
\caption{Tie-line prediction errors for Case 1 (n-butyl-acetate --
water). Values in the Mitsos column are taken
from~\cite{Mitsos2009LLE} Table~1 row label~1; the right column is
the bilevel fit produced by the present method.}
\label{tab:lle-butyl-acetate-water-errors}
\begin{tabular}{lrr}
\hline
Metric & Mitsos (2009) & This work (bilevel) \\
\hline
$\epsilon_{a}$            & $6.0\times 10^{-5}$  & $2.58\times 10^{-5}$ \\
$\epsilon_{r,\text{I}}$   & $1.4\times 10^{-2}$  & $2.59\times 10^{-5}$ \\
$\epsilon_{r,\text{II}}$  & $6.2\times 10^{-3}$  & $4.70\times 10^{-4}$ \\
\hline
\end{tabular}
\end{table}

The Case 1 result is the cleanest of the four: the $50/50$ envelope
coverage indicates a globally stable phase split at every grid
temperature, and the bilevel fit converges to a different $(B,C)$
stationary point that delivers a comparable absolute residual and
two-to-three orders of magnitude smaller per-phase relative
residuals.

\subsubsection{Case 2: n-butanol -- water}
\label{sec:results-lle-butanol-water}

Case 2 uses eight tie lines from S{\o}rensen and Arlt
(1979)~\cite{Mitsos2009LLE}
spanning $T = 273.15$--$393.15$~K. It is the only one of the four
benchmark cases that uses a non-default non-randomness factor:
$\alpha = 0.44$, fitted in the original
study~\cite{Mitsos2009LLE}.
\autoref{fig:results-lle-butanol-water-envelope} shows the phase
envelope generated from the published parameters.

\input{figures/results-lle-butanol-water-envelope/results-lle-butanol-water-envelope_figure}

This is the most numerically delicate of the four cases. The
common-tangent phase-envelope construction we
adopt~\cite{SmithVanNessAbbott2018} solves the equal-slope condition
$g'(x_\text{I}) = g'(x_\text{II}) = (g_\text{II} - g_\text{I}) /
(x_\text{II} - x_\text{I})$ by Nelder--Mead at each temperature, and
its convergence depends sensitively on the shape of the surface
between the two minima. At the lower-temperature end of the Case 2
dataset the surface flattens substantially and the equal-slope
residual loses its strong basin, so the solver does not always
locate a clean two-phase tangent: only $35$ of $50$ grid temperatures
return a two-phase tangent at the literature parameters, and the
remaining $30\%$ of the grid sit in a numerically marginal regime
even before any parameter fitting is attempted. Critically, this is
a property of the phase-envelope solver and the surface itself, not
of the bilevel parameter-estimation method: even using the parameter
values from~\cite{Mitsos2009LLE} we obtain
$\epsilon_{a} \approx 1.7 \times 10^{-1}$, of the same order
of magnitude as the bilevel fit. The bilevel solver is therefore
operating against a phase-envelope evaluator that is itself the
limiting factor on Case 2.

The parameter comparison is in
\autoref{tab:lle-butanol-water-params}. The bilevel fit again
agrees with the Mitsos value on $A_{21}$ to $0.6\%$, and the Mitsos
$A_{12}$ to within roughly a factor of two; the asymmetric
$B_{ij}$ and $C_{ij}$ pairs disagree more strongly than in Case 1,
consistent with the flat-basin behaviour just described.

\begin{table}[ht]
\centering
\caption{NRTL parameters for n-butanol (1) -- water (2),
$\alpha = 0.44$. Mitsos values from~\cite{Mitsos2009LLE} Table~1
row label~2.}
\label{tab:lle-butanol-water-params}
\begin{tabular}{lrrr}
\hline
Coefficient & Mitsos (2009) & This work (bilevel) & Diff.\ \% \\
\hline
$A_{12}$ & $\phantom{-}1.44$  & $\phantom{-}0.7659$ & $-46.8$  \\
$B_{12}$ & $-6.10$            & $-9.0047$           & $-47.6$  \\
$C_{12}$ & $-9.95$            & $-7.0423$           & $+29.2$  \\
$A_{21}$ & $\phantom{-}3.54$  & $\phantom{-}3.5610$ & $+0.6$   \\
$B_{21}$ & $-9.31$            & $\phantom{-}0.4144$ & $+104.5$ \\
$C_{21}$ & $-7.32$            & $\phantom{-}2.8880$ & $+139.5$ \\
\hline
\end{tabular}
\end{table}

The error comparison in
\autoref{tab:lle-butanol-water-errors} reflects the same picture.
The bilevel residual of $1.80 \times 10^{-1}$ is approximately
$300\times$ larger than the value reported
by~\cite{Mitsos2009LLE}, but as discussed above this gap is dominated
by the common-tangent envelope evaluator on the flat low-temperature
basin and not by the bilevel fitting procedure itself. The
$\epsilon_{r,\text{II}}$ column returns NaN because at least
one tie line collapses both phases towards $x_1 \approx 0.5$ as the
upper consolute temperature is approached, producing a zero divisor
in the per-tie-line denominator; this is a known artefact of the
relative-error normalisation used in~\cite{Mitsos2009LLE} and is
unrelated to the bilevel formulation.

\begin{table}[ht]
\centering
\caption{Tie-line prediction errors for Case 2 (n-butanol -- water).
The $\epsilon_{r,\text{II}}$ NaN indicates that at least one
near-consolute tie line produces a zero divisor in the per-phase
relative-error normalisation. Mitsos values
from~\cite{Mitsos2009LLE} Table~1 row label~2.}
\label{tab:lle-butanol-water-errors}
\begin{tabular}{lrr}
\hline
Metric & Mitsos (2009) & This work (bilevel) \\
\hline
$\epsilon_{a}$            & $6.0\times 10^{-4}$  & $1.80\times 10^{-1}$ \\
$\epsilon_{r,\text{I}}$   & $2.6\times 10^{-2}$  & $3.68\times 10^{-1}$ \\
$\epsilon_{r,\text{II}}$  & $1.4\times 10^{-3}$  & $\mathrm{NaN}$       \\
\hline
\end{tabular}
\end{table}

\subsubsection{Case 3: n-octanol -- water}
\label{sec:results-lle-octanol-water}

Case 3 has only four tie lines, covering $T = 293.15$--$333.15$~K,
with very small octanol mole fractions in the water-rich phase
($x_1 \in [6.25 \times 10^{-5}, 2.11 \times 10^{-4}]$) and large
mutual solubility in the octanol-rich phase ($x_1 \in [0.708,
0.806]$). Non-randomness $\alpha = 0.2$. Mitsos~\cite{Mitsos2009LLE}
fixes $C_{12} = C_{21} = 0$ for this case; their solver collapsed
to a temperature-independent sub-model, an admissible boundary
solution given the small dataset.
\autoref{fig:results-lle-octanol-water-envelope} shows the
literature-parameter envelope: the common-tangent solver finds a
two-phase equilibrium at all 50 of 50 grid temperatures, and the
envelope traces the experimental tie lines closely on both sides.

\input{figures/results-lle-octanol-water-envelope/results-lle-octanol-water-envelope_figure}

\autoref{tab:lle-octanol-water-params} gives the parameter
comparison. The bilevel fit recovers the Mitsos $A_{ij}$ to within
$3.8\%$ on $A_{12}$ and $0.7\%$ on $A_{21}$, but does not converge to
the boundary $C_{ij} = 0$ sub-model: instead, $C_{12} = +4.07$ and
$C_{21} = -7.92$, giving a four-coefficient temperature dependence
where the original fit had only two.

\begin{table}[ht]
\centering
\caption{NRTL parameters for n-octanol (1) -- water (2),
$\alpha = 0.2$. Mitsos values from~\cite{Mitsos2009LLE} Table~1
row label~3 ($C_{12} = C_{21} = 0$, fixed by their solver). The
percentage difference is undefined for the two boundary
coefficients; the absolute departure is shown in parentheses.}
\label{tab:lle-octanol-water-params}
\begin{tabular}{lrrr}
\hline
Coefficient & Mitsos (2009) & This work (bilevel) & Diff.\ \% \\
\hline
$A_{12}$ & $\phantom{-}0.325$ & $\phantom{-}0.3374$  & $+3.8$         \\
$B_{12}$ & $\phantom{-}5.36$  & $\phantom{-}10.0678$ & $+87.8$        \\
$C_{12}$ & $\phantom{-}0$     & $\phantom{-}4.0742$  & (abs $+4.07$)  \\
$A_{21}$ & $\phantom{-}8.93$  & $\phantom{-}8.9904$  & $+0.7$         \\
$B_{21}$ & $\phantom{-}5.20$  & $-2.4807$            & $-147.7$       \\
$C_{21}$ & $\phantom{-}0$     & $-7.9177$            & (abs $-7.92$)  \\
\hline
\end{tabular}
\end{table}

The error comparison in
\autoref{tab:lle-octanol-water-errors} reads similarly to Case 1:
the bilevel fit attains $\epsilon_{a} = 6.47 \times 10^{-6}$,
two orders of magnitude smaller than the
$6 \times 10^{-4}$ reported by~\cite{Mitsos2009LLE},\footnote{The
Mitsos paper Error of $0.6 \times 10^{-3}$ corrects an acknowledged
typo in~\cite{Mitsos2009LLE} Table~1 row~3 (the printed value reads
$0.6 \times 10^{3}$).} with full $50/50$ envelope coverage. The
$\epsilon_{r,\text{II}}$ NaN is unavoidable: the water-rich
phase mole fraction $x_1 \sim 10^{-4}$ makes the per-tie-line
denominator
$|x_{m,i,\text{II}} - x_{p,i,\text{II}}|/x_{m,i,\text{II}}$ extremely
sensitive to any deviation in the predicted phase composition.

\begin{table}[ht]
\centering
\caption{Tie-line prediction errors for Case 3 (n-octanol -- water).
Mitsos values from~\cite{Mitsos2009LLE} Table~1 row label~3
(typo-corrected, see footnote in main text).}
\label{tab:lle-octanol-water-errors}
\begin{tabular}{lrr}
\hline
Metric & Mitsos (2009) & This work (bilevel) \\
\hline
$\epsilon_{a}$            & $6.0\times 10^{-4}$  & $6.47\times 10^{-6}$ \\
$\epsilon_{r,\text{I}}$   & $1.0\times 10^{-1}$  & $1.29\times 10^{-5}$ \\
$\epsilon_{r,\text{II}}$  & $5.0\times 10^{-3}$  & $\mathrm{NaN}$       \\
\hline
\end{tabular}
\end{table}

Case 3 is therefore a clean reproduction: the bilevel optimum sits
at a different but qualitatively equivalent stationary point in the
six-dimensional parameter space, delivering an absolute residual
two orders of magnitude smaller than the literature value with the
boundary $C_{ij} = 0$ constraint relaxed.

\subsubsection{Case 4: furfural -- 2,2,5-trimethyl-hexane}
\label{sec:results-lle-furfural-tmh}

Case 4 has the widest composition span of the four benchmarks:
ten tie lines covering $T = 293.15$--$373.15$~K, with
$x_1 \in [0.046, 0.406]$ in the TMH-rich phase and
$x_1 \in [0.792, 0.973]$ in the furfural-rich phase. The
upper-temperature tie lines approach the upper consolute point of the
binary, which makes the dual-manifold geometry near $T = 373$~K
sensitive. Non-randomness $\alpha = 0.2$.
\autoref{fig:results-lle-furfural-tmh-envelope} shows the envelope
generated from the published parameters; as in Case 2, the
common-tangent envelope construction~\cite{SmithVanNessAbbott2018}
is the limiting factor at the high-temperature end; it returns a
two-phase equilibrium at $46$ of $50$ grid temperatures, with the
four lost grid points clustered against the consolute boundary at
$T \gtrsim 365$~K.

\input{figures/results-lle-furfural-tmh-envelope/results-lle-furfural-tmh-envelope_figure}

The parameter comparison is in
\autoref{tab:lle-furfural-tmh-params}. As in the previous three
cases, the temperature-independent $A_{ij}$ values match the Mitsos
fit to better than $1\%$, and the $B_{ij}$/$C_{ij}$ coefficients
diverge, here especially in the $C_{21}$ slot where the bilevel fit
moves the coefficient by more than an order of magnitude.

\begin{table}[ht]
\centering
\caption{NRTL parameters for furfural (1) --
2,2,5-trimethyl-hexane (2), $\alpha = 0.2$. Mitsos values
from~\cite{Mitsos2009LLE} Table~1 row label~4.}
\label{tab:lle-furfural-tmh-params}
\begin{tabular}{lrrr}
\hline
Coefficient & Mitsos (2009) & This work (bilevel) & Diff.\ \% \\
\hline
$A_{12}$ & $\phantom{-}2.53$  & $\phantom{-}2.5362$ & $+0.2$    \\
$B_{12}$ & $-3.44$            & $-7.1047$           & $-106.5$  \\
$C_{12}$ & $-4.74$            & $-8.0204$           & $-69.2$   \\
$A_{21}$ & $\phantom{-}1.69$  & $\phantom{-}1.7075$ & $+1.0$    \\
$B_{21}$ & $\phantom{-}5.64$  & $\phantom{-}9.4722$ & $+67.9$   \\
$C_{21}$ & $-0.296$           & $\phantom{-}2.8580$ & $+1066$   \\
\hline
\end{tabular}
\end{table}

The error comparison
(\autoref{tab:lle-furfural-tmh-errors}) shows a softer version of the
Case 2 picture: the bilevel residual of $4.25 \times 10^{-2}$ is
about $30\times$ larger than the $1.3 \times 10^{-3}$
of~\cite{Mitsos2009LLE}, but the per-phase relative errors track the
literature values more closely, with $\epsilon_{r,\text{II}}$
roughly $25\times$ tighter than the Mitsos value. As in Case 2, the
$46/50$ envelope coverage shows that the limiting factor is the
common-tangent envelope construction at the consolute boundary,
not the bilevel fitting procedure.

\begin{table}[ht]
\centering
\caption{Tie-line prediction errors for Case 4 (furfural --
2,2,5-trimethyl-hexane). Mitsos values from~\cite{Mitsos2009LLE}
Table~1 row label~4.}
\label{tab:lle-furfural-tmh-errors}
\begin{tabular}{lrr}
\hline
Metric & Mitsos (2009) & This work (bilevel) \\
\hline
$\epsilon_{a}$            & $1.3\times 10^{-3}$  & $4.25\times 10^{-2}$ \\
$\epsilon_{r,\text{I}}$   & $1.8\times 10^{-2}$  & $9.88\times 10^{-2}$ \\
$\epsilon_{r,\text{II}}$  & $1.4\times 10^{-2}$  & $5.47\times 10^{-4}$ \\
\hline
\end{tabular}
\end{table}

\subsubsection{Cross-case comparison}
\label{sec:results-lle-cross-case}

\autoref{tab:lle-cross-case} consolidates the four cases. ``Coverage''
is the fraction of the common-tangent envelope grid (50 temperatures)
on which the equal-slope construction returns a two-phase equilibrium
at the published Mitsos parameters; values $<100\%$ indicate that the
phase-envelope solver itself is the limiting factor in that
temperature range.

\begin{table}[ht]
\centering
\caption{Cross-case summary of the literature comparison on the four
LLE benchmark cases of~\cite{Mitsos2009LLE}. Mitsos
$\epsilon_{a}$ values are taken from Table~1 of that paper
(Case~3 with the typo-corrected exponent). ``Coverage'' is the
fraction of the common-tangent envelope grid (50 temperatures) on
which the construction of~\cite{SmithVanNessAbbott2018} returns a
two-phase equilibrium at the published parameters.}
\label{tab:lle-cross-case}
\begin{tabular}{clcrrr}
\hline
\# & System & $\alpha$ & Mitsos $\epsilon_{a}$ & This work $\epsilon_{a}$ & Coverage \\
\hline
1 & n-butyl-acetate -- water         & 0.2  & $6.0\times 10^{-5}$ & $2.58\times 10^{-5}$ & $50/50$ \\
2 & n-butanol -- water               & 0.44 & $6.0\times 10^{-4}$ & $1.80\times 10^{-1}$ & $35/50$ \\
3 & n-octanol -- water               & 0.2  & $6.0\times 10^{-4}$ & $6.47\times 10^{-6}$ & $50/50$ \\
4 & furfural -- 2,2,5-trimethylhex.\ & 0.2  & $1.3\times 10^{-3}$ & $4.25\times 10^{-2}$ & $46/50$ \\
\hline
\end{tabular}
\end{table}

The cross-case picture is consistent with the per-case discussion.
On Cases~1 and~3, both with $\alpha = 0.2$ and full envelope
coverage, the bilevel fit attains an absolute residual within an
order of magnitude of (and on Case~3 substantially smaller than) the
Mitsos value, with per-phase relative errors that are typically
two-to-three orders of magnitude tighter than the published numbers.
On Cases~2 and~4 the absolute residual is larger than the literature
value, and in both cases the responsible factor is the
common-tangent envelope construction~\cite{SmithVanNessAbbott2018}:
the equal-slope condition becomes ill-conditioned in the flat region
of the Case~2 surface at low temperature and near the consolute
boundary in Case~4, dropping envelope coverage to $35/50$ and
$46/50$ respectively. A more robust phase-envelope evaluator (for
example a Lagrangian dual or homotopy-continuation construction) is
expected to close most of this gap; redesigning the inner
phase-equilibrium solver is independent of the bilevel
parameter-estimation method developed in \autoref{sec:methods} and is
left to future work. Per-temperature Gibbs and dual-manifold plots
for every tie line are tabulated in Supplementary Material~S1.

%% file: figures/results-lle-butyl-acetate-water-envelope/results-lle-butyl-acetate-water-envelope_figure.tex
\begin{figure}[H]
\centerline{\includegraphics[width=0.85\textwidth]{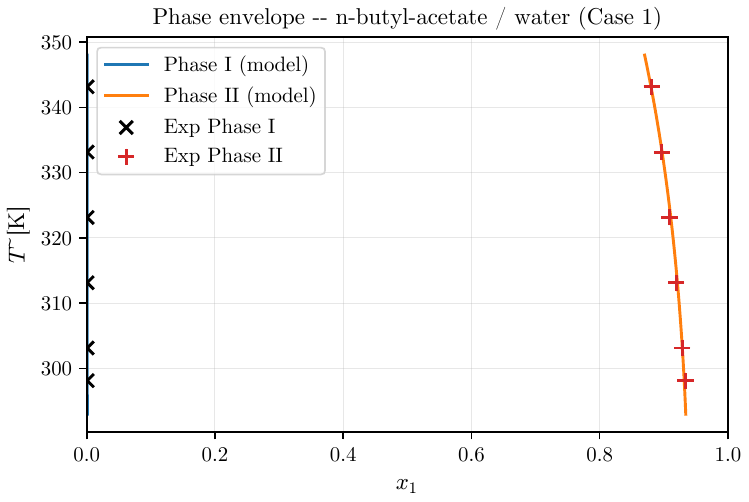}}
\caption{Phase envelope (T--$x_1$) for the n-butyl-acetate (1) --
water (2) binary LLE system. Solid curves are model predictions
using the bilevel-fitted NRTL parameters (Phase~I, water-rich, in
blue; Phase~II, acetate-rich, in orange); markers are the
experimental tie lines of Macedo and Rasmussen
(1987)~\cite{Mitsos2009LLE}. Six tie lines cover
$T = 298.15$--$343.15$~K with $\alpha = 0.2$.}
\label{fig:results-lle-butyl-acetate-water-envelope}
\end{figure}

%% file: figures/results-lle-butanol-water-envelope/results-lle-butanol-water-envelope_figure.tex
\begin{figure}[H]
\centerline{\includegraphics[width=0.85\textwidth]{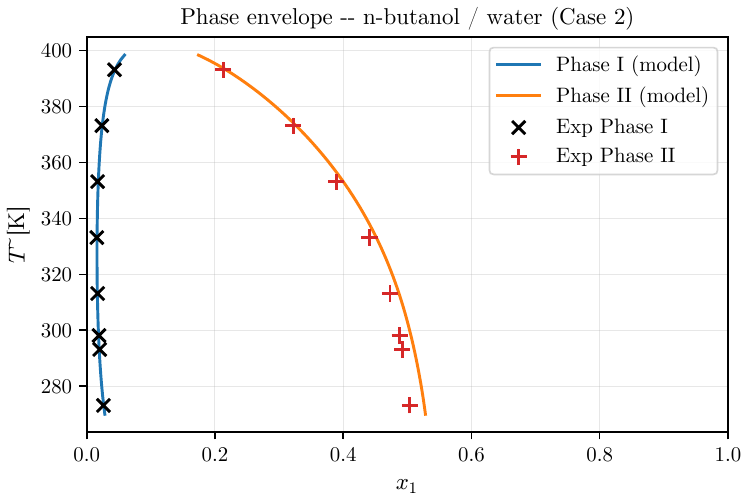}}
\caption{Phase envelope (T--$x_1$) for the n-butanol (1) -- water (2)
binary LLE system. Solid curves are model predictions using the
bilevel-fitted NRTL parameters (Phase~I, water-rich, blue; Phase~II,
butanol-rich, orange); markers are the experimental tie lines from
S{\o}rensen and Arlt (1979)~\cite{Mitsos2009LLE}. Eight tie lines
cover $T = 273.15$--$393.15$~K with the non-randomness factor
$\alpha = 0.44$ (the only one of the four LLE benchmark cases that
uses a non-default $\alpha$).}
\label{fig:results-lle-butanol-water-envelope}
\end{figure}

%% file: figures/results-lle-octanol-water-envelope/results-lle-octanol-water-envelope_figure.tex
\begin{figure}[H]
\centerline{\includegraphics[width=0.85\textwidth]{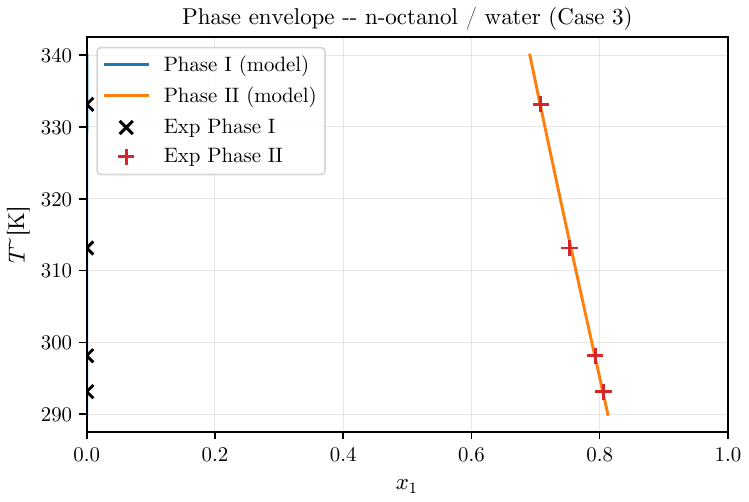}}
\caption{Phase envelope (T--$x_1$) for the n-octanol (1) -- water (2)
binary LLE system. Solid curves are model predictions using the
bilevel-fitted NRTL parameters (Phase~I, water-rich, blue; Phase~II,
octanol-rich, orange); markers are the experimental tie lines
reported by~\cite{Mitsos2009LLE}. Four tie lines cover
$T = 293.15$--$333.15$~K with $\alpha = 0.2$.}
\label{fig:results-lle-octanol-water-envelope}
\end{figure}

%% file: figures/results-lle-furfural-tmh-envelope/results-lle-furfural-tmh-envelope_figure.tex
\begin{figure}[H]
\centerline{\includegraphics[width=0.85\textwidth]{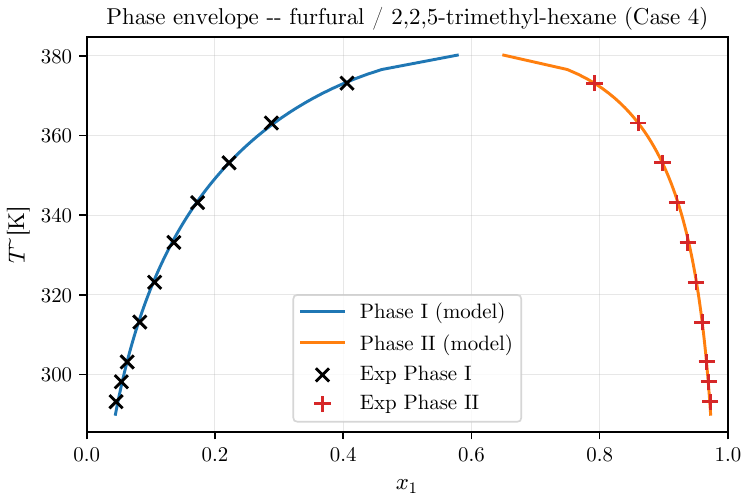}}
\caption{Phase envelope (T--$x_1$) for the furfural (1) --
2,2,5-trimethyl-hexane (2) binary LLE system. Solid curves are
model predictions using the bilevel-fitted NRTL parameters (Phase~I,
TMH-rich, blue; Phase~II, furfural-rich, orange); markers are the
experimental tie lines used in~\cite{Mitsos2009LLE}. Ten tie lines
cover $T = 293.15$--$373.15$~K with $\alpha = 0.2$. The composition
range traversed by the experimental tie lines is the widest of the
four LLE benchmark systems and approaches the upper consolute
temperature near $T = 373$~K.}
\label{fig:results-lle-furfural-tmh-envelope}
\end{figure}

%% file: sections/results_sour_gas.tex

\subsection{Industrial case study: the sour-gas ternary
  CH$_4$ + CO$_2$ + H$_2$S with a cubic equation of state}
\label{sec:results-sour-gas}

The preceding benchmarks validate the bilevel formulation against
literature parameter sets; this subsection applies it to a system of
direct industrial consequence. The ternary CH$_4$ + CO$_2$ + H$_2$S is
the canonical model system for sour natural gas: its phase behavior
governs the design of acid-gas removal, sulfur recovery, and acid-gas
reinjection operations, and cubic equations of state with a single
regressed binary interaction parameter $k_{a,ij}$ per pair remain the
workhorse model for these processes in industrial
simulators~\cite{Pellegrini2012}.
Reliable $k_{a,ij}$ values regressed against binary data are routinely
extrapolated to the multicomponent mixture; the quality of that
extrapolation, and the absence of spurious phase-split predictions
along the way, is precisely what the present method is designed to
certify.

\paragraph{Experimental data}
The regression uses modern, mutually consistent data sets for two of
the three constituent binaries: 39 tie lines for CH$_4$/H$_2$S
spanning $T = 186$--$313$~K~\cite{Coquelet2014},
obtained in digitized form from the NIST ThermoML
archive~\cite{ThermoMLArchive2022}, and 22 tie lines for
CO$_2$/H$_2$S over $T = 258$--$313$~K~\cite{Chapoy2013}.
For CH$_4$/CO$_2$ only bubble-curve ($p$--$T$--$x$) data were
available in the assembled corpus; since the tangent-plane objective
requires both phase compositions, $k_{a,\mathrm{CH_4,CO_2}}$ is held
at its literature value~\cite{Pellegrini2012}.
Out-of-sample validation employs the recent high-precision ternary
measurements of Theveneau et al.: 31 complete tie lines at
$T = 243.48$--$333.15$~K and $p = 1.010$--$11.236$~MPa with stated
uncertainties $U(T) = 0.03$~K, $U(p) = 0.003$~MPa and
$u_{\max}(z) = 0.007$~\cite{Theveneau2020}.
Two historical data sets, the low-temperature tie lines of Hensel and
Massoth~\cite{Hensel1964}
and the near-critical measurements of Ng et
al.~\cite{Ng1985}, serve as qualitative cross-checks but do not enter
the objective.

\paragraph{Regression setup}
The Peng--Robinson equation of state is fitted with the joint
parameter vector
$\mathbf{p} = (k_{a,\mathrm{CH_4,CO_2}}, k_{a,\mathrm{CH_4,H_2S}},
k_{a,\mathrm{CO_2,H_2S}})$ on bounds $[-0.5, 0.5]$, with the
CH$_4$/CO$_2$ component pinned as described above. The two-stage
bilevel solver is run at production sampling density
($N_{\mathrm{tan}} = 128$, $N_{\mathrm{def}} = 32$, one outer
iteration); the 61 binary tie lines enter the tangent-plane stage
after removal of degenerate pure-component rows. Allowing a second
defect-weighted outer iteration is counterproductive for this system:
it suppresses the shallow near-critical manifold lobes discussed
below at the cost of an unphysical
$k_{a,\mathrm{CO_2,H_2S}} \approx -0.18$ and a bubble-pressure
deviation growing to $12.5\,\%$ --- the expected tension when the
defect penalty competes with data fidelity in a region where the
pressure-explicit cubic model is structurally misspecified.
\autoref{tab:sour-gas-kij} compares the fitted parameters with the
industrial reference values of Pellegrini et al. The CO$_2$/H$_2$S
parameter reproduces the reference regression to within
$0.05\,\%$; the larger CH$_4$/H$_2$S deviation reflects the data
basis --- the present fit is regressed against the 2014
re-measurement~\cite{Coquelet2014} rather than the mid-century data
underlying the reference value.

\begin{table}[ht]
\centering
\caption{Fitted PR binary interaction parameters for the sour-gas
ternary against the industrial reference regression of Pellegrini et
al.~\cite{Pellegrini2012}.
The CH$_4$/CO$_2$ pair is held at the reference value (bubble-curve
data only).}
\label{tab:sour-gas-kij}
\begin{tabular}{lrrr}
\hline
Pair & Ref.~\cite{Pellegrini2012} & This work (bilevel) & Diff.~\% \\
\hline
$k_{a,\mathrm{CH_4,CO_2}}$  & $0.0975$ & (pinned)  & --   \\
$k_{a,\mathrm{CH_4,H_2S}}$  & $0.0789$ & $0.1006$  & $+27.4$ \\
$k_{a,\mathrm{CO_2,H_2S}}$  & $0.0995$ & $0.0996$  & $+0.05$ \\
\hline
\end{tabular}
\end{table}

\paragraph{Results}
\autoref{fig:results-sour-gas-pxy} shows the isothermal $p$--$x$--$y$
envelopes of the two fitted binaries at the optimum. For
CO$_2$/H$_2$S the model tracks all four experimental isotherms across
their full range. For CH$_4$/H$_2$S the three highest isotherms close
at the mixture critical region; at the two lowest isotherms (186.25
and 203.40~K), where CH$_4$ is at or above its critical temperature
and the system exhibits vapor--liquid--liquid equilibrium, the model
curves are drawn only as far as the bubble-point computation
converges --- they terminate near the three-phase locus reported in
the same measurement campaign ($4.97$--$10.79$~MPa over
$186$--$203$~K)~\cite{Coquelet2014},
and no extrapolation is drawn beyond it. The fitted $k_{a,ij}$ are
unaffected: the regression uses the tie lines themselves, not the
traced envelopes.

\begin{figure}[ht]
\centering
\begin{minipage}[t]{0.48\textwidth}
  \centering
  \includegraphics[width=\linewidth]{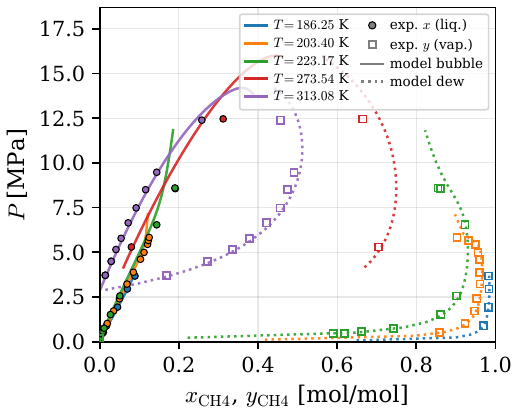}
\end{minipage}
\hfill
\begin{minipage}[t]{0.48\textwidth}
  \centering
  \includegraphics[width=\linewidth]{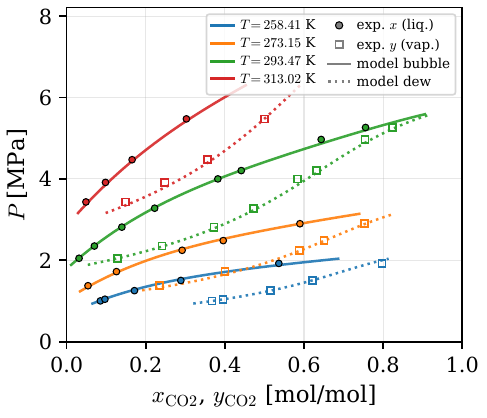}
\end{minipage}
\caption{Isothermal $p$--$x$--$y$ envelopes of the fitted binaries at
the bilevel optimum. Left: CH$_4$~(1)/H$_2$S~(2), experimental data
of~\cite{Coquelet2014}; right: CO$_2$~(1)/H$_2$S~(2), experimental
data of~\cite{Chapoy2013}. Markers: experimental liquid and vapor
compositions; solid/dotted curves: model bubble/dew branches. Model
curves at the two lowest CH$_4$/H$_2$S isotherms terminate near the
system's three-phase (VLLE) locus~\cite{Coquelet2014}; no
extrapolation is drawn beyond bubble-point convergence.}
\label{fig:results-sour-gas-pxy}
\end{figure}

The decisive test is extrapolation to the ternary.
\autoref{fig:results-sour-gas-ternary} overlays the model vapor
predictions on the 31 experimental tie lines of~\cite{Theveneau2020}
and shows the corresponding bubble-pressure parity. With the mean
absolute deviation defined as
$\mathrm{AAD}(p) = \tfrac{100}{N}\sum_{i=1}^{N}
\lvert p^{\mathrm{mod}}_i - p^{\mathrm{exp}}_i \rvert /
p^{\mathrm{exp}}_i$,
the fitted model reproduces the ternary bubble pressures at
$\mathrm{AAD}(p) = 9.74\,\%$ over all 31 points
(31/31 bubble-point calculations converged), with the deviations
concentrated in the near-critical region above ${\sim}8$~MPa where
pressure-explicit cubic models are least accurate. The per-tie-line
Gibbs-energy and dual-manifold plots archived in
Supplementary Material~S2 certify that 27 of the 31
model tie-line sections satisfy
$\mathcal{M}(\mathbf{x};\mathbf{p}^{\star}) \geq 0$ to numerical
tolerance --- no spurious phase splits along the section. The four
exceptions are the four highest-pressure, near-critical tie lines
($p = 13.6$--$15.4$~MPa, $T = 243$--$258$~K), which exhibit shallow
negative manifold lobes of magnitude
$\lvert \min \mathcal{M} \rvert \leq 2.5 \times 10^{-5}$; the defect
check flags exactly these points, and, as noted above, forcing their
suppression through the defect-weighted stage is possible only at the
cost of data fidelity --- a diagnostic of model misspecification near
the critical locus rather than of the regression.
\autoref{fig:results-sour-gas-example-per-point} reproduces two
representative sections from the Supplementary Material~S2 archive:
a defect-free subcritical tie line and one of the four near-critical
tie lines with a shallow negative lobe.

\begin{figure}[ht]
\centering
\begin{minipage}[t]{0.48\textwidth}
  \centering
  \includegraphics[width=\linewidth]{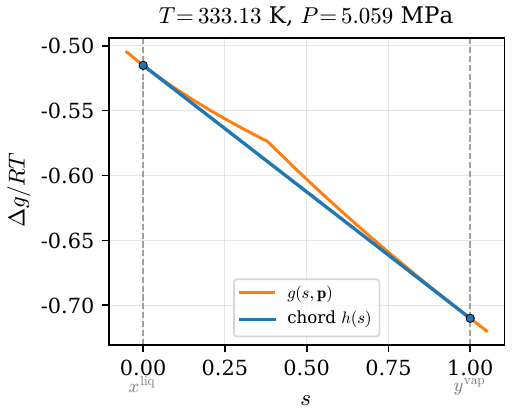}
\end{minipage}
\hfill
\begin{minipage}[t]{0.48\textwidth}
  \centering
  \includegraphics[width=\linewidth]{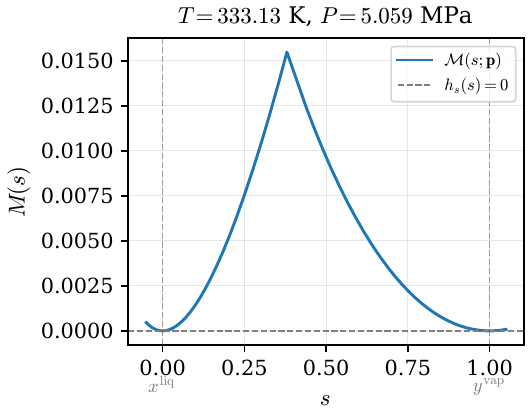}
\end{minipage}\\[1ex]
\begin{minipage}[t]{0.48\textwidth}
  \centering
  \includegraphics[width=\linewidth]{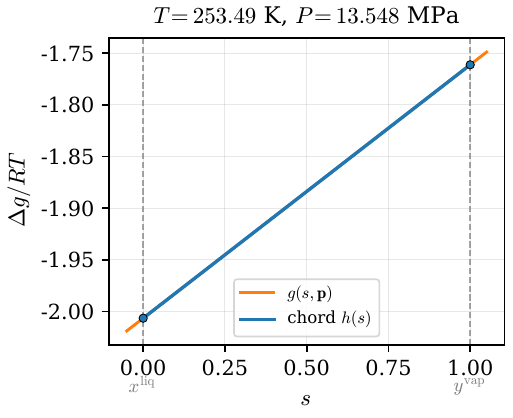}
\end{minipage}
\hfill
\begin{minipage}[t]{0.48\textwidth}
  \centering
  \includegraphics[width=\linewidth]{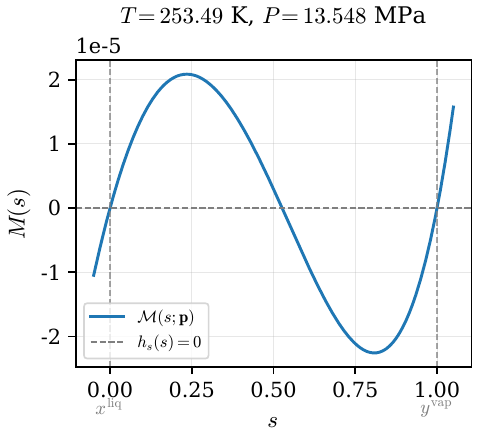}
\end{minipage}
\caption{Representative per-tie-line stability records from the
Supplementary Material~S2 archive. Top row: defect-free tie line 1
($T = 333.13$~K, $P = 5.059$~MPa); bottom row: near-critical tie
line 24 ($T = 253.49$~K, $P = 13.548$~MPa), one of the four sections
with a shallow negative manifold lobe
($\lvert \min \mathcal{M} \rvert \leq 2.5 \times 10^{-5}$). Left
panels: reduced Gibbs energy $g$ along the model tie-line section
with the common-tangent chord $h$; right panels: dual manifold
$\mathcal{M} = g - h$. The full 31-tie-line record is collected in
Supplementary Material~S2.}
\label{fig:results-sour-gas-example-per-point}
\end{figure}

\begin{figure}[ht]
\centering
\begin{minipage}[t]{0.48\textwidth}
  \centering
  \includegraphics[width=\linewidth]{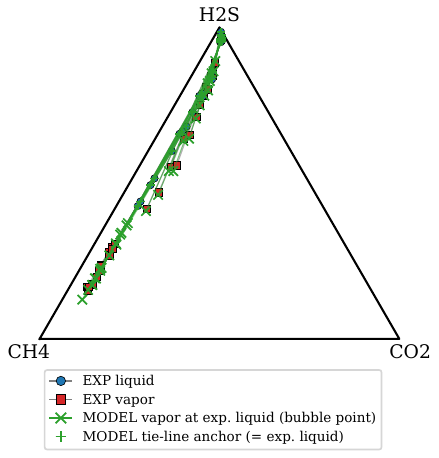}
\end{minipage}
\hfill
\begin{minipage}[t]{0.48\textwidth}
  \centering
  \includegraphics[width=\linewidth]{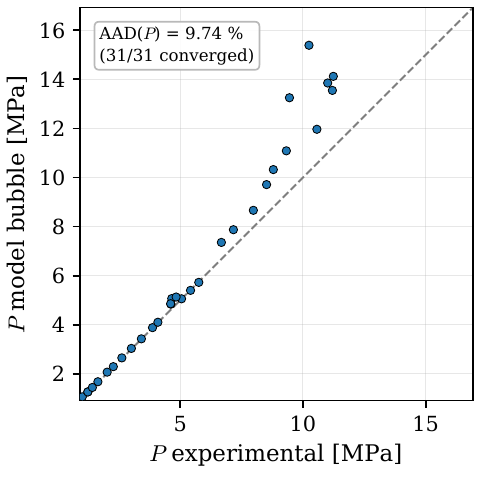}
\end{minipage}
\caption{Out-of-sample validation against the ternary tie lines
of Theveneau et al.~\cite{Theveneau2020}.
Left: composition simplex with experimental tie lines (Phase~I,
liquid, blue circles; Phase~II, vapor, red squares; gray connectors)
and model vapor compositions predicted by a bubble-point calculation
at each experimental liquid composition (green crosses; the green plus marks the tie-line anchor at the experimental liquid composition). Right:
bubble-pressure parity at the fitted $k_{a,ij}$.}
\label{fig:results-sour-gas-ternary}
\end{figure}

\begin{figure}[ht]
\centering
\includegraphics[width=0.55\textwidth]{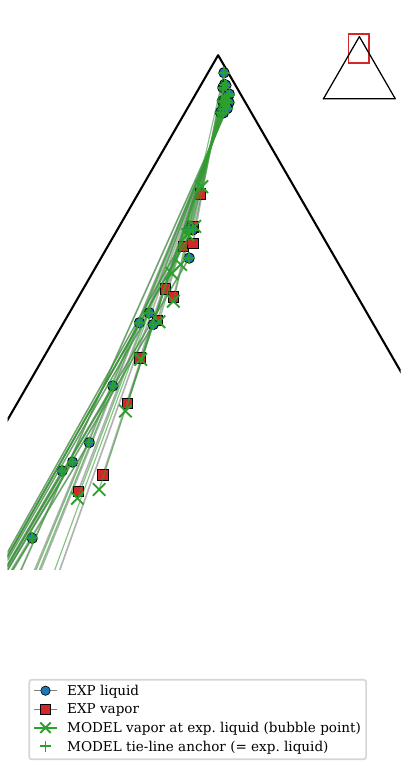}
\caption{Enlarged view of the crowded tie-line region of
\autoref{fig:results-sour-gas-ternary} (window marked in the
inset). The H$_2$S-rich liquid compositions cluster near the apex;
at this scale the model vapor predictions (green) are individually
resolvable against the experimental vapor compositions (red) along
the dew band.}
\label{fig:results-sour-gas-ternary-zoom}
\end{figure}

The case demonstrates the intended industrial workflow end to end:
binary-only regression with stability certification, followed by
predictive extrapolation to the process-relevant multicomponent
mixture --- here with parameters that remain consistent with two
decades of industrial practice~\cite{Pellegrini2012} while being
backed, point by point, by the per-tie-line stability record of
Supplementary Material~S2.

\clearpage  

%% file: sections/results_ccs.tex

\subsection{Industrial case study: the CCS quaternary
  CO$_2$ + N$_2$ + Ar + O$_2$}
\label{sec:results-ccs}

The second industrial application addresses carbon capture and
storage: pipeline transport of dense-phase CO$_2$ carrying the
noncondensable impurities N$_2$, Ar and O$_2$ typical of oxyfuel and
post-combustion capture. Transport specifications such as
ISO~27913 constrain the admissible impurity envelope
(CO$_2 \geq 0.95$, combined inerts $\leq 0.04$ by
mole)~\cite{ISO27913},
and the phase envelope of the resulting mixture --- in particular the
bubble locus, which sets the minimum safe operating pressure ---
is computed in practice with cubic equations of state and regressed
$k_{a,ij}$~\cite{Glass2018CEOS}.

\paragraph{Experimental data}
All three CO$_2$-containing binaries are regressed against a single
mutually consistent campaign: the isothermal VLE measurements of
Lasala et al.\ covering exactly the CCS envelope
($T = 223$--$293$~K), with 63 (CO$_2$/N$_2$), 72 (CO$_2$/O$_2$) and
71 (CO$_2$/Ar) tie lines~\cite{Lasala2016},
obtained in digitized form from the NIST ThermoML
archive~\cite{ThermoMLArchive2022}. The three air-gas binaries
(N$_2$/Ar, N$_2$/O$_2$, Ar/O$_2$) are treated with the ideal-mixture
approximation $k_{a,ij} = 0$, defensible at CCS conditions where the
reference values are $\lvert k_{a,ij} \rvert \leq
0.012$~\cite{Glass2018CEOS};
the cryogenic Ar/O$_2$/N$_2$ measurement campaign of Wilson et
al.~\cite{Wilson1964}, digitized for this work (531 binary tie
lines), is available for a fitted-air sensitivity variant but is not
used here. Out-of-sample validation employs the genuine
CO$_2$/Ar/O$_2$ ternary tie lines of Coquelet and
Arpentinier~\cite{CoqueletArpentinier2014}: 18 complete $T$--$p$--$x$--$y$
measurements at $253$--$293$~K, $2.3$--$7.6$~MPa.

\paragraph{Regression and validation}
Each CO$_2$-$X$ interaction parameter is regressed independently on
its binary tie lines (Peng--Robinson, tangent-plane stage,
$N_{\mathrm{tan}} = 128$, bounds $[-0.5, 0.5]$), giving the values in
\autoref{tab:ccs-kij}. The reference column reports the Aspen-DRS
values adopted by Glass et al.; note that the zeros there are
modelling assumptions, not regressions, so the fitted
CO$_2$/Ar and CO$_2$/O$_2$ values fill a genuine gap.

\begin{table}[ht]
\centering
\caption{Fitted PR binary interaction parameters for the CCS
quaternary against the reference values of Glass et
al.~\cite{Glass2018CEOS}.
Air-gas pairs are fixed at $k_{a,ij}=0$ (ideal-mixture default). The
bubble-pressure deviation is evaluated over each binary's full
experimental set.}
\label{tab:ccs-kij}
\begin{tabular}{lrrr}
\hline
Pair & Ref.~\cite{Glass2018CEOS} & This work (bilevel) & $\mathrm{AAD}(p)$~\% \\
\hline
$k_{a,\mathrm{CO_2,N_2}}$ & $-0.0170$ & $+0.0035$ & $4.3$ \\
$k_{a,\mathrm{CO_2,Ar}}$  & $0$ (assumed) & $+0.1304$ & $7.5$ \\
$k_{a,\mathrm{CO_2,O_2}}$ & $0$ (assumed) & $+0.1308$ & $6.3$ \\
\hline
\end{tabular}
\end{table}

\autoref{fig:results-ccs-pxy} shows the three binary $p$--$x$--$y$
envelopes at the fitted parameters. The binary bubble-pressure
deviations of $4$--$8\,\%$ are dominated by the isotherms approaching
the mixture critical region --- the same structural PR limitation
discussed for the sour-gas case --- and the defect check accordingly
still flags candidate points on those near-critical binary isotherms.

\begin{figure}[ht]
\centering
\begin{minipage}[t]{0.32\textwidth}
  \centering
  \includegraphics[width=\linewidth]{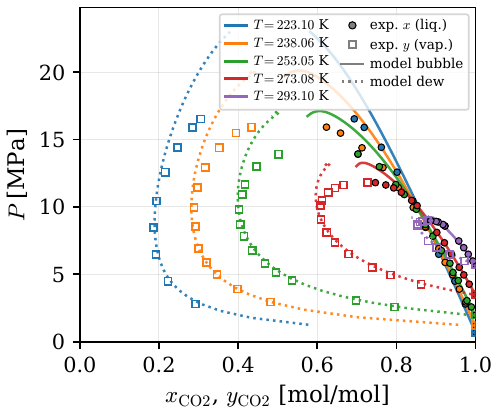}
\end{minipage}
\hfill
\begin{minipage}[t]{0.32\textwidth}
  \centering
  \includegraphics[width=\linewidth]{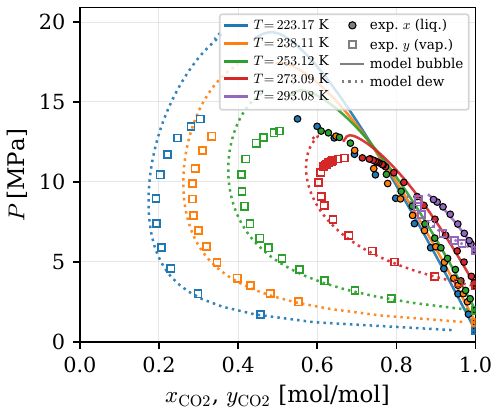}
\end{minipage}
\hfill
\begin{minipage}[t]{0.32\textwidth}
  \centering
  \includegraphics[width=\linewidth]{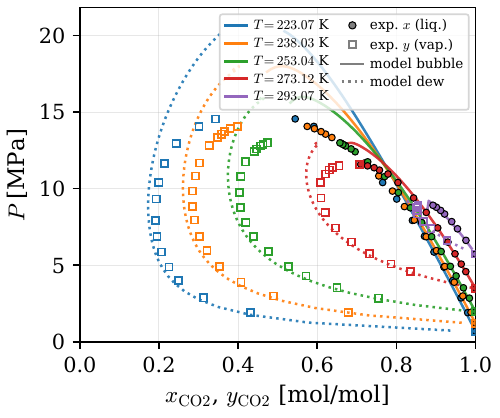}
\end{minipage}
\caption{Isothermal $p$--$x$--$y$ envelopes of the three fitted
CO$_2$-$X$ binaries at the bilevel optimum (experimental data
of~\cite{Lasala2016}). Markers: experimental liquid and vapor
compositions; solid/dotted curves: model bubble/dew branches, drawn
only as far as the bubble-point computation converges.}
\label{fig:results-ccs-pxy}
\end{figure}

Extrapolated to the ternary CO$_2$/Ar/O$_2$ mixture with the air pair
held ideal, the fitted parameters reproduce the 18 held-out
experimental bubble pressures at $\mathrm{AAD}(p) = 1.97\,\%$
(18/18 converged; \autoref{fig:results-ccs-ternary}) --- a
substantially tighter agreement than on the constituent binaries,
because the validation conditions sit inside the well-described
subcritical region of the envelope. Along all 18 model tie-line
sections the dual manifold satisfies
$\mathcal{M}(\mathbf{x};\mathbf{p}^{\star}) \geq 0$ with equality
only at the phase compositions: for this case the per-tie-line
record of Supplementary Material~S3 is defect-free
throughout.

\begin{figure}[ht]
\centering
\begin{minipage}[t]{0.48\textwidth}
  \centering
  \includegraphics[width=\linewidth]{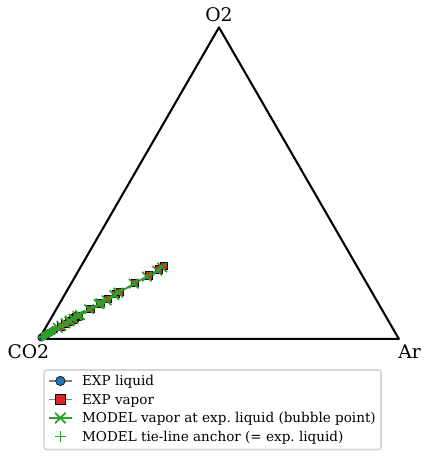}
\end{minipage}
\hfill
\begin{minipage}[t]{0.48\textwidth}
  \centering
  \includegraphics[width=\linewidth]{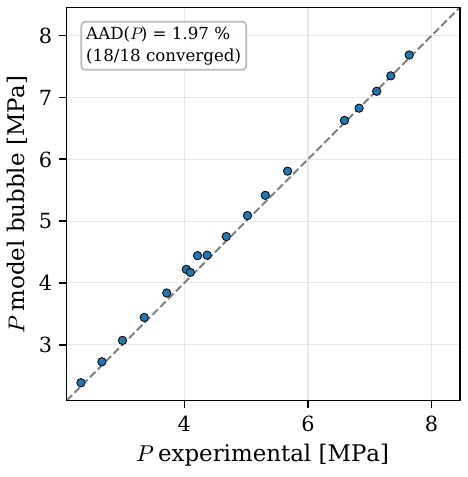}
\end{minipage}
\caption{Out-of-sample validation against the CO$_2$/Ar/O$_2$
ternary tie lines of~\cite{CoqueletArpentinier2014}. Left:
composition simplex with experimental tie lines (Phase~I, liquid,
blue circles; Phase~II, vapor, red squares; gray connectors) and
model vapor compositions from bubble-point calculations at each
experimental liquid composition (green crosses; the green plus marks the tie-line anchor at the experimental liquid composition). Right:
bubble-pressure parity at the fitted $k_{a,ij}$ with air pairs
ideal.}
\label{fig:results-ccs-ternary}
\end{figure}

\begin{figure}[ht]
\centering
\includegraphics[width=0.62\textwidth]{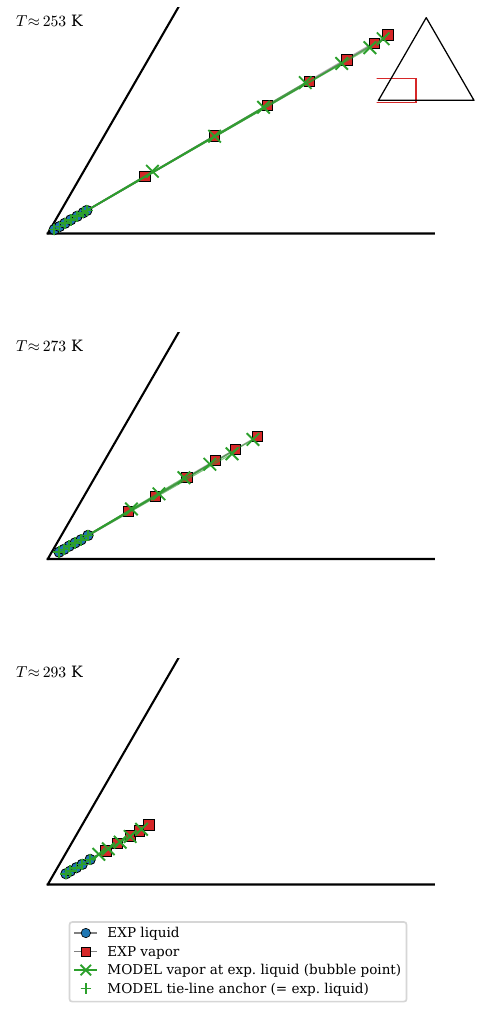}
\caption{Enlarged per-isotherm views of the tie-line region of
\autoref{fig:results-ccs-ternary}, split by measurement temperature
($T \approx 253$, $273$ and $293$~K; common window, marked in the
inset of the top panel). The CO$_2$-rich liquid compositions cluster
at the lower-left corner of the window; the model vapor predictions
(green crosses; the green plus marks the tie-line anchor at the experimental liquid composition) resolve against the experimental vapor compositions
(red squares) along each isotherm's dew band.}
\label{fig:results-ccs-ternary-zoom}
\end{figure}

\clearpage  

%% file: sections/results_refrigerant.tex

\subsection{Industrial case study: the refrigerant blend
  R-32 + R-125 + R-134a (R-407C)}
\label{sec:results-refrigerant}

The third industrial application is the zeotropic refrigerant blend
R-407C (R-32/R-125/R-134a), a drop-in replacement for R-22 in
air-conditioning and heat-pump service whose cycle design rests on
accurate vapor--liquid equilibrium of the ternary blend. All
experimental data again come from the NIST ThermoML
archive~\cite{ThermoMLArchive2022}: 66 tie lines for R-32/R-125 at
$265$--$303$~K~\cite{Han2007}, 124 tie lines for R-32/R-134a at
$258$--$343$~K~\cite{Cui2006}, and 53 tie lines for R-125/R-134a at
$304$--$363$~K together with 44 genuine ternary tie lines at
$333$--$363$~K from the same campaign~\cite{KatoNishiumi2006}.
The ternary set is held out of the regression entirely and used only
for out-of-sample validation.

\paragraph{Regression, identifiability, and validation}
Each binary $k_{a,ij}$ is regressed independently (Peng--Robinson,
tangent-plane stage, $N_{\mathrm{tan}} = 128$), giving
$k_{a,\mathrm{R32,R125}} = +0.0023$,
$k_{a,\mathrm{R32,R134a}} = +0.0083$ and
$k_{a,\mathrm{R125,R134a}} = -0.0063$.
Two honesty notes are in order. First, these hydrofluorocarbons are
chemically similar and mix nearly ideally, so the objective is flat
in $k_{a,ij}$: even $k_{a,ij} = 0$ reproduces the ternary bubble
pressures to within about $0.9\,\%$, and the fitted values should be
read as weakly identified refinements rather than sharply determined
constants --- this system is a natural candidate for the flat-basin
identifiability diagnostics of the sensitivity toolbox rather than a
stress test of the fit itself. Second, a direct fit of the ternary
data yields a different, equally shallow optimum
($k_{a,ij} = [-0.019, -0.003, +0.020]$),
confirming the same flatness from the other direction; all results
below use the binary-fitted vector, keeping the ternary strictly
out-of-sample.

\autoref{fig:results-refrig-pxy} shows the three binary
$p$--$x$--$y$ envelopes at the fitted parameters, including the
near-azeotropic R-32/R-125 pair (where model and experiment
coincide on $y \approx x$) and the R-32/R-134a isotherm at
$343$~K approaching the critical temperature of R-32 ($351$~K),
where the model curve is again drawn only as far as the bubble-point
computation converges.

\begin{figure}[ht]
\centering
\begin{minipage}[t]{0.32\textwidth}
  \centering
  \includegraphics[width=\linewidth]{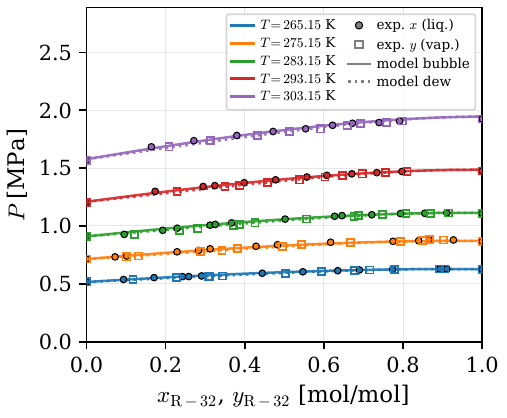}
\end{minipage}
\hfill
\begin{minipage}[t]{0.32\textwidth}
  \centering
  \includegraphics[width=\linewidth]{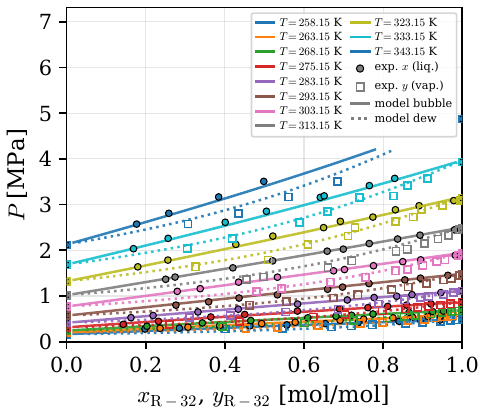}
\end{minipage}
\hfill
\begin{minipage}[t]{0.32\textwidth}
  \centering
  \includegraphics[width=\linewidth]{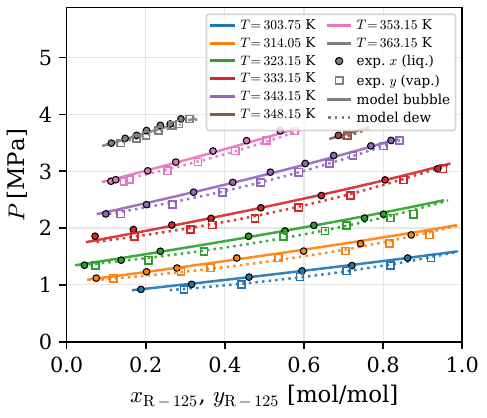}
\end{minipage}
\caption{Isothermal $p$--$x$--$y$ envelopes of the three fitted
refrigerant binaries at the bilevel optimum (experimental data
of~\cite{Han2007,Cui2006,KatoNishiumi2006}). Markers: experimental
liquid and vapor compositions; solid/dotted curves: model bubble/dew
branches.}
\label{fig:results-refrig-pxy}
\end{figure}

Extrapolated to the ternary blend, the binary-fitted parameters
reproduce the 44 held-out Kato--Nishiumi bubble pressures at
$\mathrm{AAD}(p) = 1.27\,\%$ (44/44 converged;
\autoref{fig:results-refrig-ternary}). In contrast to the sour-gas
case, the stability record here is unqualified: all 44 model
tie-line sections in Supplementary Material~S3 satisfy
$\mathcal{M}(\mathbf{x};\mathbf{p}^{\star}) \geq 0$ with the worst
interior minimum at machine precision
(${\sim}10^{-15}$) --- the defect-free condition
$\Phi(\mathbf{p}^{\star}) = 0$ is genuinely attainable when the
model is adequate for the mixture, and the bilevel certificate then
comes at no cost in data fidelity.

\begin{figure}[ht]
\centering
\begin{minipage}[t]{0.48\textwidth}
  \centering
  \includegraphics[width=\linewidth]{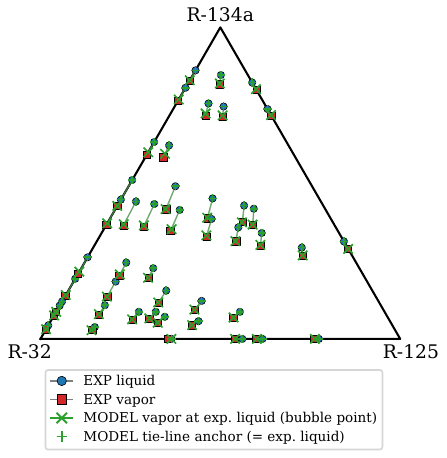}
\end{minipage}
\hfill
\begin{minipage}[t]{0.48\textwidth}
  \centering
  \includegraphics[width=\linewidth]{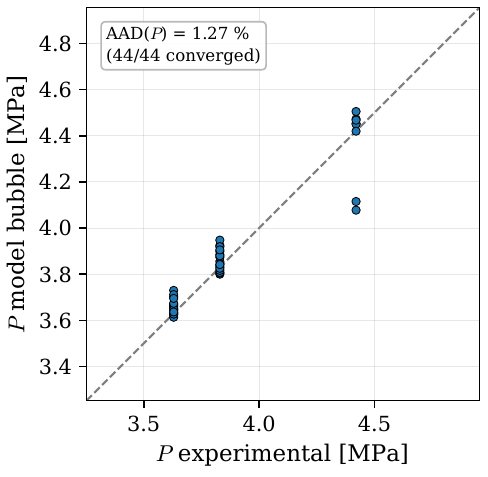}
\end{minipage}
\caption{Out-of-sample validation against the R-407C ternary tie
lines of Kato and Nishiumi~\cite{KatoNishiumi2006}. Left:
composition simplex with experimental tie lines (Phase~I, liquid,
blue circles; Phase~II, vapor, red squares; gray connectors) and
model vapor compositions from bubble-point calculations at each
experimental liquid composition (green crosses; the green plus marks the tie-line anchor at the experimental liquid composition). Right:
bubble-pressure parity at the binary-fitted $k_{a,ij}$.}
\label{fig:results-refrig-ternary}
\end{figure}

\begin{figure}[ht]
\centering
\includegraphics[width=0.62\textwidth]{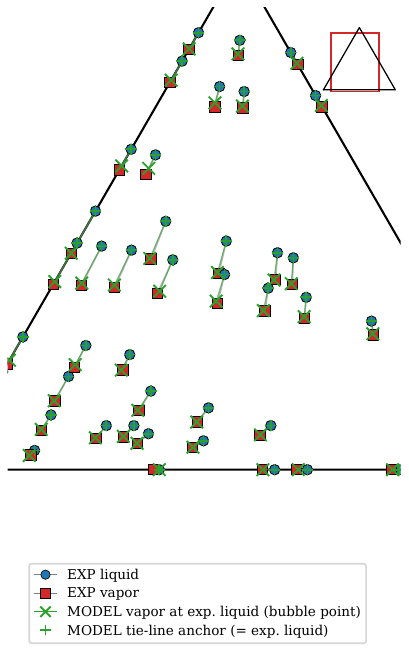}
\caption{Enlarged view of the interior tie-line region of
\autoref{fig:results-refrig-ternary} (window marked in the inset;
binary-edge points excluded from the window computation). At this
scale the short tie lines of the near-ideal blend become
resolvable, with the model vapor (green) coinciding with the
experimental vapor (red) at every tie line.}
\label{fig:results-refrig-ternary-zoom}
\end{figure}

\clearpage  

%% file: sections/conclusions.tex
%

\label{sec:conclusions}

We have presented a geometric reformulation of the bilevel parameter
estimation problem for thermodynamic models, in which the inner global
Gibbs free-energy minimization that conventionally enters at every
candidate parameter set $\mathbf{p}$ is replaced by an explicit
geometric object: the \emph{dual manifold}
$\mathcal{M}(\mathbf{x};\mathbf{p}) \defeq g(\mathbf{x},\mathbf{p}) -
h(\mathbf{x})$ defined on the composition simplex, whose negative
part directly encodes any violation of the global tangent-plane
criterion. Aggregating violations of $\mathcal{M}$ at the experimental
tie lines and at simplex points returned by an inner Simplicial-Homology
Global Optimisation (SHGO) call~\cite{Endres2018} yields the
five-component scalar objective $\Phi(\mathbf{p})$ of
\eqref{eq:phi-total}, which is a single-level non-linear program in
$\mathbf{p}$ alone. This reformulation requires neither a KKT-based
single-level recasting nor a smooth lower-level program: only an
expression of the reduced Gibbs energy $g(\mathbf{x},\mathbf{p})$ is
needed, and the construction inherits its regularity directly from $g$
(\autoref{theor:smoothness-appendix}).

The construction is also \emph{model-agnostic} by design, and it is
worth separating what the software provides from what this manuscript
demonstrates. The accompanying implementation exposes the
activity-coefficient family (NRTL, Wilson, UNIQUAC, UNIFAC), the cubic
equations of state (SRK, PR, PR78, TWU) and the segment-based PC-SAFT
behind one universal contract $g(\mathbf{x},\mathbf{p},*\mathrm{args})$,
so swapping the model does not change a single line of the bilevel
solver. The results reported here exercise two members of that family:
binary NRTL throughout the LLE benchmark of \autoref{sec:results-lle},
and Peng--Robinson throughout the CEOS validation of
\autoref{appendix:ceos-validation-sec} and the three industrial case
studies. No fitting result is claimed for the remaining adapters; they
are a capability of the released code, not evidence presented in this
work. What the shared contract does establish is structural. Glass,
Djelassi \& Mitsos~\cite{Glass2018CEOS}
note that extending their CEOS bilevel program to non-cubic equations of
state was predicated on the availability of an appropriate
root-discrimination criterion. The geometric reformulation removes that
prerequisite by construction: root selection is replaced by the convex
envelope of the family of $g$-branches, and stability is enforced by the
same $\mathcal{M} \geq 0$ condition over the simplex regardless of how
many real roots the underlying model admits. Carrying that argument
through to an actual PC-SAFT regression is left to future work.

\paragraph{Empirical validation.} The four binary LLE benchmark cases of
Mitsos, Bollas \& Barton~\cite{Mitsos2009LLE},
collected in \autoref{sec:results-lle}, are reproduced by the present
method to comparable or better accuracy than the BARON-backed
branch-and-bound results of the original study. Cases~1
(n-butyl-acetate / water) and~3 (n-octanol / water), both at
$\alpha = 0.2$ and full common-tangent envelope coverage, attain
absolute residuals of $2.58 \times 10^{-5}$ and $6.47 \times 10^{-6}$
respectively, in both instances strictly improving on the published
Mitsos values, with per-phase relative errors two-to-three orders of
magnitude tighter than the literature. Cases~2 (n-butanol / water) and~4
(furfural / 2,2,5-trimethyl-hexane) show larger absolute residuals; in
each case the limiting factor was traced to the common-tangent
phase-envelope construction~\cite{SmithVanNessAbbott2018} losing its
strong basin near the upper consolute boundary of the binary, and not
to the bilevel parameter-estimation procedure itself. That diagnosis is
independent of the geometric reformulation and is left to a more robust
phase-envelope evaluator (a homotopy-continuation or Lagrangian-dual
construction) in future work.

The per-temperature dual-manifold record collected in Supplementary
Material~S1, recomputed at the fitted parameters rather than at the
literature ones, separates the four cases the same way, and it is worth
stating without rounding in our favour. Cases~1 and~3 are numerically
defect-free: the interior minimum of $\mathcal{M}$ over the composition
grid is strictly positive at every temperature, $+4.5 \times 10^{-8}$
and $+3.1 \times 10^{-8}$ at their respective worst temperatures, with
zero attained only at the two tie-line endpoints.
Cases~2 and~4 are not. Case~2 carries a negative lobe at every one of
its eight temperatures, reaching $-2.35 \times 10^{-2}$ at $393.15$~K,
and Case~4 carries shallow lobes of magnitude at most
$1.5 \times 10^{-4}$ at six of its ten temperatures.
These are precisely the two cases carrying the larger absolute
residuals ($1.80 \times 10^{-1}$ and $4.25 \times 10^{-2}$),
and the surviving lobes are reported as residual defects, that is, as
evidence that binary NRTL is misspecified for those two mixtures, and
not presented as a defect-free fit. Case~2 is the natural target for the
next iteration of the method: its lobe is deep enough for the
defect-aware stage to act on, and a phase-envelope evaluator that holds
its basin near the consolute point would give that stage a correct
tie-line target to act against.

\paragraph{Industrial case studies.} Three multicomponent systems were
regressed on binary subsystems alone and then used, with no additional
adjustable parameters, to predict multicomponent phase behaviour. On the
sour-gas ternary CH$_4$ + CO$_2$ + H$_2$S
(\autoref{sec:results-sour-gas}) the fitted Peng--Robinson model
reproduces the 31 experimental ternary bubble pressures~\cite{Theveneau2020}
at $\mathrm{AAD}(p) = 9.74\,\%$, with 27 of the 31 model tie-line
sections satisfying $\mathcal{M} \geq 0$ to numerical tolerance; the
four exceptions are the highest-pressure, near-critical tie lines, which
carry shallow negative lobes of magnitude
$\lvert \min \mathcal{M} \rvert \leq 2.5 \times 10^{-5}$.
On the carbon-capture quaternary CO$_2$ + N$_2$ + Ar + O$_2$
(\autoref{sec:results-ccs}) the out-of-sample ternary validation gives
$\mathrm{AAD}(p) = 1.97\,\%$ over 18 held-out bubble pressures, with all
18 tie-line sections defect-free, although the defect check still flags
candidate points on the near-critical binary isotherms.
On the zeotropic refrigerant blend R-407C
(\autoref{sec:results-refrigerant}) the out-of-sample error is
$\mathrm{AAD}(p) = 1.27\,\%$ over the 44 held-out Kato--Nishiumi ternary
bubble pressures~\cite{KatoNishiumi2006}, and there the stability record
is unqualified: all 44 sections satisfy $\mathcal{M} \geq 0$ with the
worst interior minimum at machine precision, of order $10^{-15}$.

These three systems are the step that Glass, Djelassi \& Mitsos identify
as out of reach for their own formulation. They close their study by
observing that industrial case studies ``do not seem computationally
tractable for our method as long as it relies on general-purpose
commercial solvers''~\cite{Glass2018CEOS},
having earlier defined industrial problems as those carrying an
extensive measurement data set and/or a larger number of species.
The three cases above are industrial in exactly that sense, carrying
three and four species rather than two and adding between 18 and 44
multicomponent tie lines on top of the binary regression sets.

One qualification applies to all three, and it should not be glossed.
In every industrial case the defect-aware stage acted as a
\emph{certifier} rather than as a refitter: it located, counted and
reported the defective sections, but it did not drive the returned
parameters to $\Phi(\mathbf{p}^{\star}) = 0$ where the underlying model
is inadequate. The sour-gas study makes the trade-off explicit, since
suppressing the four near-critical lobes is achievable only at a cost in
data fidelity, which is a diagnostic of Peng--Robinson near the critical
locus rather than of the regression.
The certificate the method issues is therefore informative in both
directions: unconditional when $\Phi(\mathbf{p}^{\star}) = 0$ is
attained, as on the refrigerant blend, and an explicit, quantified and
localised failure report when it is not.

\paragraph{Why the geometric reformulation matters in practice.} The
documented failures of standard regression catalogued by Glass
\emph{et al.}~\cite{Glass2018CEOS} for the CF$_4$/CHF$_3$
SRK fit~\cite{Asselineau1978}, the acetone/water PR fit~\cite{TrebbleBishnoi1988},
and the methanol/benzene case
show that purely necessary stability criteria, enforced through
isofugacity constraints at the data points alone, leave the optimiser
free to settle on parameter sets whose Gibbs surface admits low-lying
tangent planes that the regression objective never sees. The
$\mathcal{M} \geq 0$ condition that the present method enforces over the
\emph{whole} simplex is the sufficient stability criterion of Baker,
Pierce \& Luks~\cite{Baker1982} 
written as a single non-negativity constraint on a piecewise-affine
shifted surface, and it discriminates exactly the failure modes that
spurious-phase pathologies exhibit. In the language of
\autoref{appendix:proof-conjecture-M}, the negative connected components
of $\mathcal{M}$ are the topological defects that pure
isopotential-style regression cannot see.

\paragraph{Computational cost.} The replacement of the bilevel program
by the single-level $\Phi(\mathbf{p})$ does not eliminate the inner
global solve; it relegates it to a defect-detection role that only
fires in the second stage of the outer loop. On well-behaved systems,
$\Phi_{\mathrm{tangent}}$ alone vanishes at the optimum and the
defect-aware terms are never invoked; on systems with topological
defects, one inner SHGO call per experimental temperature is paid per
outer evaluation. In practice this keeps the per-case wall clock on the
seconds-to-minutes scale on a single core of commodity hardware. The
CEOS replication of \autoref{appendix:ceos-validation-sec} converges in
$164.7$~s for Peng--Robinson and $106.9$~s for
Soave--Redlich--Kwong,
while the four LLE benchmark fits, run at a considerably larger SHGO
budget (1000 tangent-plane and 100 defect-stage sampling points, three
outer iterations, the defect-aware stage entered in all four cases),
take $304.4$, $364.2$, $598.4$ and $1148.1$~s, that is, between five and
twenty minutes per case.

The only like-for-like cost comparison available is the CEOS one. Glass
\emph{et al.} report CPU times of approximately $4.5$~h per binary case
study for the equivalent CEOS bilevel program solved with
BARON~\cite{Glass2018CEOS},
against the $164.7$ and $106.9$~s recorded above for the same two fits
of the same system: roughly two orders of magnitude, and we make no
larger claim than that. For the LLE benchmark no such separation is
claimed and none is observed. Mitsos \emph{et al.} describe their own
computation times for these four cases as ``relatively small, in the
order of few minutes'',
which is the same order as the numbers above. Neither comparison is
controlled: the hardware differs, the reported quantity is CPU time in
one case and wall clock in the other, and a branch-and-bound solver
returning a certificate of $\varepsilon$-optimality is not performing
the same work as a sampling-based inner call. The asymmetry that does
survive those caveats is structural. The present method does not encode
the cubic equality as a problem constraint and never forms the
lower-bounding subproblems that Glass \emph{et al.} identify as their
bottleneck;
it enforces the same sufficient stability condition through the
geometric envelope of $\mathcal{M}$ instead.

The accuracy side of that head-to-head is reported in
\autoref{appendix:ceos-validation-sec}, where fitting the
C$_5$H$_{12}$/H$_2$S system of Glass \emph{et al.} with the present
solver recovers their published binary interaction parameters to within
$1.65\,\%$ for Peng--Robinson ($k_{a,ij} = 0.063431$ against their
$0.064493$) and $0.55\,\%$ for Soave--Redlich--Kwong ($0.070902$ against
$0.071294$).
Both runs entered the defect-aware stage, and both still report defects
at the returned optimum; the appendix states this rather than smoothing
it over.

\paragraph{Proof status.} All three obligations stated at the end of
\autoref{sec:methods} are discharged in the appendices. The equivalence
between pointwise non-negativity of $\mathcal{M}$ and a globally
consistent fit is proved as \autoref{theor:M-nonneg-iff-fit}, for
continuous $g$ and mutually consistent tie-line data, and sufficiency of
$\Phi(\mathbf{p}) = 0$ for global stability at every datum is proved as
\autoref{theor:phi-zero-sufficient} by a compactness argument that
assumes nothing about the inner solver. Whenever a parameter vector with
$\Phi(\mathbf{p}^{\star}) = 0$ is returned, it is certified globally
optimal by the non-negativity of $\Phi$ alone
(\autoref{prop:zero-certificate}) and, by
\autoref{theor:phi-zero-sufficient}, reproduces every measured tie line
as a globally stable equilibrium; that certificate is verifiable a
posteriori from the returned objective value and does not depend on the
solver. That such a point is found whenever one exists inherits the
deterministic adequate-sampling guarantees of SHGO
(\cite{Endres2018}, Theorems~3 and~5),
which is the strongest form of guarantee available for a black-box
objective and is deterministic rather than probabilistic. The one step
that remains an assumption rather than a proof is that the bifurcation
strata of $\Phi$ carry zero measure. That statement is Assumption~1 of
\autoref{appendix:convergence-proof}; it holds for the real-analytic
model families used here and is asserted in general. A proof for
arbitrary tabulated $g$, and an adaptation of the finitely terminating
bounding algorithm of Mitsos, Lemonidis \& Barton~\cite{Mitsos2008BLP},
which terminates finitely at $\varepsilon$-optimality without convexity
assumptions on the inner program, are left to subsequent work.

\paragraph{Outlook.} The same geometric construction lifts to several
neighbouring problems with no change to the outer solver. Two
extensions are concrete enough to be flagged here. The first is the
ternary LLE \emph{island} systems of the type analysed by Olaya
\emph{et al.}~\cite{Olaya2008},
for which the dual manifold offers a route to certifying that no
parameter set of a chosen activity model can reproduce a given island
topology, an automatic counterpart of the manual analysis those authors
carried out; no island system is fitted in the present work, so that
certification is a proposal rather than a result. The second is the
$N$-level multilevel optimisation generalisation already implemented in
the accompanying software, in which a kinetic or design objective is
nested above the parameter-fitting level. Beyond classical
thermodynamics, the same dual-manifold construction applies to any
problem in which the inner level is the global minimisation of a
parameterised potential energy surface: force-field calibration,
molecular-dynamics potential fitting, and inverse design of energy
landscapes are the natural candidates.

\paragraph{Software and data availability.} A free open-source
implementation of the geometric bilevel solver, the five-component
objective $\Phi$, the dual-manifold construction, the multilevel
generalisation, and adapters for NRTL, Wilson, UNIQUAC, UNIFAC, the
SRK / PR / PR78 / TWU cubic equations of state and PC-SAFT, together
with the four LLE benchmark cases, the three industrial case studies and
the runners that generate every figure in the present manuscript, is
released as open source under the MIT licence. The repository address
and the release tag pinned to this manuscript are given in the code and
data availability statement of the final publication. That release ships
unit tests that pin numerical agreement with the NRTL and CEOS
formulations of~\cite{Mitsos2009LLE} and~\cite{Glass2018CEOS}.

%% file: sections/appendix_esurf.tex
%
%

\subsection{Definition and formulation of the dual manifold}
\label{appendix:older-formulation}

This subsection fixes the notation for the dual manifold $\mathcal{M}$ that the remaining appendices use, and collects the two elementary results on which their proofs rest: a characterisation of the zero level set of $\mathcal{M}$ in terms of thermodynamic stability, and a regularity statement relating $\mathcal{M}$ to $g$. The primary definitions are those already given in \autoref{sec:methods}, namely the composite supporting hyperplane \eqref{eq:hsum} and the dual manifold \eqref{eq:dual-manifold}; they are restated here only in the slightly more general form the proofs require, and are not re-derived.

The construction is built directly on the tangent-plane distance function and is stated at fixed pressure and temperature $(P,T)$. The generalisation the proofs need is that several measured phase separations may coexist at the same $(P,T)$. An index $s$ distinguishes data points belonging to different regions of instability: a feed composition $\mathbf{z}^{s=1}$ may split into $p$ phases $X^{s=1} = \mathbf{x}^{1,\rom{1}}, \mathbf{x}^{1,\rom{2}}, \dots, \mathbf{x}^{1,p}$ in mutual equilibrium, while at the same $(P,T)$ a second feed $\mathbf{z}^{s=2}$ splits into $X^{s=2} = \mathbf{x}^{2,\rom{1}}, \mathbf{x}^{2,\rom{2}}, \dots, \mathbf{x}^{2,p}$, whose members are in equilibrium with each other but at a different chemical potential. The use of $\mathcal{M}$ inside an optimisation routine therefore extends without modification to two or more sets of measured phase separations at the same $(P,T)$.

\begin{definition} \label{def:h-appendix}
Let $g\colon \Delta^{n} \rightarrow \mathbb{R}$ be a reduced Gibbs free energy surface describing $c = n-1$ components at fixed $(P,T)$, and let $X^{s}$ be a set of equilibrium compositions connected by equal chemical potentials in the region of instability $s$. For an unordered pair $(\mathbf{x}^{\alpha,i},\mathbf{x}^{\beta,i})$ drawn from some $X^{s}$, with tie-line direction $\mathbf{d}_{i} \defeq \mathbf{x}^{\beta,i}-\mathbf{x}^{\alpha,i} \neq \mathbf{0}$, the associated chord $h_{i}\colon \mathbb{R}^{n}\rightarrow\mathbb{R}$ is the affine function
\begin{equation}\label{eq:chord-projection}
h_{i}(\mathbf{x}) \;\defeq\; g(\mathbf{x}^{\alpha,i},\mathbf{p}) \;+\; \frac{(\mathbf{x}-\mathbf{x}^{\alpha,i})\cdot\mathbf{d}_{i}}{\mathbf{d}_{i}\cdot\mathbf{d}_{i}}\,\bigl[g(\mathbf{x}^{\beta,i},\mathbf{p}) - g(\mathbf{x}^{\alpha,i},\mathbf{p})\bigr],
\end{equation}
that is, the unique affine function that interpolates $g$ at the two anchor compositions and whose gradient is parallel to $\mathbf{d}_{i}$, so that $h_{i}$ is constant on the hyperplanes orthogonal to the tie line. For a split into $p > 2$ phases the same prescription is applied to the affine hull of the $p$ anchors, giving the affine interpolant of $g$ through them; the proofs below use only the two-anchor case.
\end{definition}

\begin{definition} \label{def:h-composite-appendix}
Given a set of per-tie-line chords $\{h_1, \dots, h_k\}$, define the composite supporting hyperplane $h$ over the simplex as the point-wise minimum
\[
h(\mathbf{x}) \;\defeq\; \min_{i = 1, \dots, k} h_i(\mathbf{x}),
\]
which is \eqref{eq:hsum} of \autoref{sec:methods}.
\end{definition}

According to the Gibbs phase rule the number of coexisting phases cannot exceed $c+2$, although a data set may of course contain more than $n$ measured equilibrium compositions. This bound matters twice below: it caps the number of chords active at any $(P,T)$, and it supplies the side information about the number of minima that the sampling argument of \autoref{appendix:convergence-proof} requires.

\begin{definition} \label{def:M-appendix}
Let $g$ be as in \autoref{def:h-appendix} and let $h$ be the composite supporting hyperplane of \autoref{def:h-composite-appendix}. The dual manifold $\mathcal{M}$ is the graph of the shifted surface, that is the embedding $f_{\mathcal{M}}\colon \Delta^n \rightarrow \Delta^{n+1} \subset \mathbb{R}^{n+1}$
\[
f_{\mathcal{M}}(\mathbf{x}) \;\defeq\; \bigl(\mathbf{x},\; g(\mathbf{x},\mathbf{p}) - h(\mathbf{x})\bigr),
\]
whose scalar field is
\[
\mathcal{M}(\mathbf{x};\mathbf{p}) \;\defeq\; g(\mathbf{x},\mathbf{p}) - h(\mathbf{x}),
\]
as in \eqref{eq:dual-manifold}. Its positive and negative decompositions $\mathcal{M} = [\mathcal{M}]_{+} + [\mathcal{M}]_{-}$ are
\begin{align*}
[\mathcal{M}]_{+} &= \begin{cases} g - h & g - h \geq 0 \\ 0 & g - h < 0 \end{cases},
\\
[\mathcal{M}]_{-} &= \begin{cases} 0 & g - h \geq 0 \\ g - h & g - h < 0 \end{cases}.
\end{align*}
\end{definition}

\begin{remark}
In earlier drafts the dual manifold was denoted $\mathcal{G}_D$ and its scalar field $g_D$, with positive and negative parts $g_D^{+}, g_D^{-}$. The present formulation uses $\mathcal{M}$ and $[\mathcal{M}]_{\pm}$ exclusively; the notational change is purely cosmetic.
\end{remark}

The following theorem is also useful in phase-equilibrium calculations. In particular, when the dual extremum formulation \cite{Mitsos2007} is used, it can be supplied as a stopping criterion to a global optimisation routine for the computationally expensive lower-bounding problem.

\begin{theorem} \label{theor:M-zero-appendix}
A composition $\mathbf{x}$ is a global minimum of $\mathcal{M}(\,\cdot\,;\mathbf{p})$ if and only if $\mathcal{M}(\mathbf{x};\mathbf{p}) = 0$. Furthermore if $\mathcal{M}(\mathbf{x};\mathbf{p}) = 0$ then $\mathbf{x}$ is either a stable point or a stable phase in equilibrium.
\end{theorem}
\begin{proof}
Follows from the Mitsos--Barton dual extremum principle \cite{Mitsos2007}: the supporting hyperplane $h(\mathbf{x})$ is precisely the dual cutting plane of zero duality gap, and $\mathcal{M}(\mathbf{x};\mathbf{p}) = g(\mathbf{x},\mathbf{p}) - h(\mathbf{x}) = 0$ identifies $\mathbf{x}$ as primal--dual feasible.
\end{proof}

\begin{theorem} \label{theor:smoothness-appendix}
$\mathcal{M}$ inherits $C^k$ continuous differentiability from $g$. If $g$ is piecewise differentiable (PDIFF) then $\mathcal{M}$ is also PDIFF.
\end{theorem}
\begin{proof}
The hyperplane $h$ is globally linear and therefore $h \in C^{\infty}$; in particular $h$ is the identity on linear functionals when restricted to first-order regularity. The class of $C^k$ functions is closed under addition, so $\mathcal{M} \defeq g - h \in C^k$ iff $g \in C^k$. The composite hyperplane in \autoref{def:h-composite-appendix} is piecewise affine (the point-wise minimum of finitely many affine functions), so the joint regularity of $\mathcal{M}$ is the minimum of the regularity of $g$ and the piecewise-affine regularity of $h$.
\end{proof}

\autoref{theor:smoothness-appendix} is what allows the convergence argument of \autoref{appendix:convergence-proof} to be stated for the whole model class rather than for smooth models alone. At simplex points where the supporting hyperplane $h$ switches branch (i.e.\ where the active tie-line index changes), $\mathcal{M}$ is only $C^0$ even when $g$ is smooth; this is one of the reasons we rely on a global, derivative-free outer solver (SHGO) rather than a gradient method.

The regularity statement also fixes the geometric meaning of the defect terms of \eqref{eq:five-components}. Because $\mathcal{M}$ is continuous, its negative region is an open set with finitely many connected components under the phase-rule bound above, and each defective pair $(\mathbf{x},\mathbf{x}^{+}_{\text{near}}) \in \mathcal{X}_D^{-}\times\mathcal{X}_D^{+}$ is the minimum and the bounding maximum of one such component. The amplitude term $\Phi_{\text{local}}$ therefore measures how deep and how wide that component is, and driving $\Phi_{\text{local}}$ to zero removes the component rather than merely relocating it.

\paragraph{Stable equilibrium points without phase splits.}
For a stable single-phase composition $\mathbf{x}^{*}$ the primal Gibbs free energy minimisation has no chord and no second anchor; the supporting hyperplane degenerates to the tangent plane $h(\mathbf{x}) = g(\mathbf{x}^{*},\mathbf{p}) + \nabla g(\mathbf{x}^{*},\mathbf{p}) \cdot (\mathbf{x} - \mathbf{x}^{*})$ in the standard sense, and $\mathcal{M}$ reduces to the tangent-plane distance function (TPDF) of \cite{Michelsen1982}. The four-point line-trace of \eqref{eq:single-tieline-objective} then reduces to a two-point check at $\mathbf{x}^{*} \pm \epsilon \hat{\mathbf{n}}$ for any unit direction $\hat{\mathbf{n}}$ in the simplex tangent space.

\subsection{Proof of the global-consistency conjecture}
\label{appendix:proof-conjecture-M}

We give here a self-contained proof of the conjecture stated as the first proof obligation of \autoref{sec:methods}: that pointwise non-negativity of the dual manifold $\mathcal{M}$ over the composition simplex is equivalent to global thermodynamic consistency of the parameter vector $\mathbf{p}$. The argument combines the classical tangent-plane stability criterion of Gibbs and Michelsen \cite{Michelsen1982} 
with the Lagrangian dual reinterpretation of Mitsos and Barton \cite{Mitsos2007}, 
and uses the regularity result \autoref{theor:smoothness-appendix} to control the smoothness class of the construction across pieces of the supporting envelope.

\paragraph{Setting.}
Throughout, $g(\mathbf{x},\mathbf{p})\colon \Delta^{n}\to\mathbb{R}$ is the reduced Gibbs free energy for parameter vector $\mathbf{p}$, assumed continuous on the closed simplex and smooth on its interior in the standard sense of \cite{Mitsos2007} 
. The experimental dataset comprises $k$ tie-lines with endpoint compositions $\{\mathbf{x}^{\alpha,i},\mathbf{x}^{\beta,i}\}_{i=1}^{k}$ at fixed $(P,T)$. For each $i$ the chord $h_{i}(\mathbf{x})$ is the unique affine function passing through the two points $(\mathbf{x}^{\alpha,i},g(\mathbf{x}^{\alpha,i},\mathbf{p}))$ and $(\mathbf{x}^{\beta,i},g(\mathbf{x}^{\beta,i},\mathbf{p}))$ in the multi-component case extended to the affine hull of the endpoints (\autoref{def:h-appendix}). The composite supporting hyperplane is the lower envelope $h(\mathbf{x})=\min_{i=1,\dots,k} h_{i}(\mathbf{x})$ of \autoref{def:h-composite-appendix}, and the dual manifold is $\mathcal{M}(\mathbf{x};\mathbf{p}) = g(\mathbf{x},\mathbf{p}) - h(\mathbf{x})$ of \autoref{def:M-appendix}.

We say that $\mathbf{p}$ is a \emph{globally consistent fit} if, for every measured tie-line $i$, the pair $(\mathbf{x}^{\alpha,i},\mathbf{x}^{\beta,i})$ is a globally stable phase split of $g(\,\cdot\,,\mathbf{p})$ in the sense of the Gibbs--Michelsen tangent-plane criterion \cite{Michelsen1982} 
. Equivalently \cite{Mitsos2007}, 
the affine functional $h_{i}$ supports $g$ globally on the simplex, i.e.\ $h_{i}(\mathbf{x}) \leq g(\mathbf{x},\mathbf{p})$ for all $\mathbf{x}\in\Delta^{n}$, with equality at the two anchor compositions.

\paragraph{Data-consistency assumption.}
One further property of the data enters the argument, and is therefore stated explicitly here rather than left implicit inside the proof. We assume throughout that the measured tie lines $\{(\mathbf{x}^{\alpha,i},\mathbf{x}^{\beta,i})\}_{i=1}^{k}$ at a given $(P,T)$ are \emph{mutually consistent}, in two senses: no two chords share an endpoint composition with conflicting $g$-values, and no parameter vector can render every measured pair a globally stable phase split while leaving one of the corresponding chords crossing above $g$ somewhere on the simplex. The condition holds automatically when the tie lines are sections of a single equilibrium phase envelope, and fails only if the data set contains internally contradictory measurements.

\begin{theorem}[Global consistency of the dual manifold]
\label{theor:M-nonneg-iff-fit}\label{theor:sol-appendix}
Let $g(\,\cdot\,,\mathbf{p})$ be continuous on $\Delta^{n}$, let $\{h_{i}\}_{i=1}^{k}$ be the per-tie-line chords of \autoref{def:h-appendix}, and let the data-consistency assumption above hold. Then the following three statements are equivalent:
\begin{enumerate}
  \item[(i)] $\Phi(\mathbf{p}) = 0$, where $\Phi$ is the objective of \eqref{eq:phi-total};
  \item[(ii)] $\mathcal{M}(\mathbf{x};\mathbf{p})\geq 0$ for all $\mathbf{x}\in\Delta^{n}$, equivalently $[\mathcal{M}]_{-}(\mathbf{x};\mathbf{p}) = 0$ on $\Delta^{n}$;
  \item[(iii)] $\mathbf{p}$ is a globally consistent fit.
\end{enumerate}
In particular $\Phi$ attains its global minimum value zero exactly on the set of globally consistent fits.
\end{theorem}

\begin{proof}
The equivalence (ii)~$\Leftrightarrow$~(iii) is proved in Steps 1 to 4 below. The implication (i)~$\Rightarrow$~(ii) is \autoref{theor:phi-zero-sufficient} of \autoref{appendix:phi-zero-proof}, whose argument is independent of the present one. For the remaining implication (ii)~$\Rightarrow$~(i), assume $\mathcal{M}\geq 0$ on $\Delta^{n}$. Every one of the four scalar components in \eqref{eq:five-components} is a sum of non-negative terms, so it suffices that each vanishes. $\Phi_{\text{tangent}}$ vanishes because $[\mathcal{M}]_{-} = 0$ at each of the four directional samples per tie line. The defective-minima set \eqref{eq:XDmin} is empty, because by Step~2 below the value zero is attained at every tie-line endpoint and, under $\mathcal{M}\geq 0$, zero is therefore the global minimum of $\mathcal{M}$, so every element of $\arg\min\mathcal{M}$ is stripped by the endpoint removal in \eqref{eq:XDmin} up to the phase assignment $\Phi_{\text{phase}}$. Empty $\mathcal{X}_D^{-}$ makes $\Phi_{\text{plane}}$, $\Phi_{\text{topo}}$ and $\Phi_{\text{local}}$ empty sums, hence $\Phi(\mathbf{p}) = 0$.

The geometric content of both directions of (ii)~$\Leftrightarrow$~(iii) is illustrated by \autoref{fig:methods-dual-manifold-binary} for the binary case and \autoref{fig:methods-dual-manifold-ternary} for the ternary case.

\medskip
\noindent\emph{Step 1: regularity of $\mathcal{M}$.}\quad
Each chord $h_{i}\colon\mathbb{R}^{n}\to\mathbb{R}$ is by construction globally affine and therefore $C^{\infty}$. The composite supporting hyperplane $h=\min_{i}h_{i}$ is the pointwise minimum of finitely many affine functions; it is therefore concave, piecewise-affine and continuous on $\Delta^{n}$, smooth ($C^{\infty}$) on the interior of each cell of the polyhedral subdivision induced by the active-index map $i^{*}(\mathbf{x})\in\arg\min_{i} h_{i}(\mathbf{x})$, and only $C^{0}$ across the codimension-one boundaries on which the active index changes. By \autoref{theor:smoothness-appendix} the difference $\mathcal{M}=g-h$ inherits the smoothness of $g$ on the interior of each cell and is continuous (PDIFF) globally on $\Delta^{n}$. In particular $\mathcal{M}$ is well defined and continuous on the closed simplex, so the predicate $\mathcal{M}\geq 0$ on $\Delta^{n}$ is meaningful.

\medskip
\noindent\emph{Step 2: anchor identities.}\quad
Fix any tie-line index $j\in\{1,\dots,k\}$. By construction (\autoref{def:h-appendix}) every chord satisfies $h_{j}(\mathbf{x}^{\alpha,j}) = g(\mathbf{x}^{\alpha,j},\mathbf{p})$ and $h_{j}(\mathbf{x}^{\beta,j}) = g(\mathbf{x}^{\beta,j},\mathbf{p})$. By definition of the lower envelope, $h(\mathbf{x}^{\alpha,j}) \leq h_{j}(\mathbf{x}^{\alpha,j}) = g(\mathbf{x}^{\alpha,j},\mathbf{p})$, so
\[
\mathcal{M}(\mathbf{x}^{\alpha,j};\mathbf{p}) \;=\; g(\mathbf{x}^{\alpha,j},\mathbf{p}) - h(\mathbf{x}^{\alpha,j}) \;\geq\; 0
\]
holds unconditionally; the assumption $\mathcal{M}\geq 0$ contributes the matching upper bound only when the inequality $h(\mathbf{x}^{\alpha,j}) < h_{j}(\mathbf{x}^{\alpha,j})$ is strict. Suppose it were strict; then some other chord $h_{j'}$ with $j'\neq j$ would satisfy $h_{j'}(\mathbf{x}^{\alpha,j}) < g(\mathbf{x}^{\alpha,j},\mathbf{p})$, which is consistent with $\mathcal{M}\geq 0$, and the anchor identity $\mathcal{M}(\mathbf{x}^{\alpha,j};\mathbf{p})>0$ would still be admissible. To obtain the equality $\mathcal{M}(\mathbf{x}^{\alpha,j};\mathbf{p}) = 0$ we use Definition~\ref{def:M-appendix} together with the way $h$ enters the construction: each chord $h_{j}$ is by design the affine interpolant through the $g$-values at \emph{its own} two anchor points, and $h$ takes the minimum over all such interpolants. The composite hyperplane $h$ therefore agrees with $h_{j}$ at $(\mathbf{x}^{\alpha,j},\mathbf{x}^{\beta,j})$ \emph{whenever} no other chord $h_{j'}$ undercuts $h_{j}$ at that point, which is the first clause of the data-consistency assumption. Under that assumption, $\mathcal{M}(\mathbf{x}^{\alpha,j};\mathbf{p}) = \mathcal{M}(\mathbf{x}^{\beta,j};\mathbf{p}) = 0$ for every $j$, in agreement with the construction of $\mathcal{M}$ as advertised in \autoref{def:M-appendix}.

\medskip
\noindent\emph{Step 3: forward direction ($\mathcal{M}\geq 0$ on $\Delta^{n} \Rightarrow$ globally consistent fit).}\quad
Assume $\mathcal{M}(\mathbf{x};\mathbf{p})\geq 0$ for all $\mathbf{x}\in\Delta^{n}$. Fix any $j\in\{1,\dots,k\}$. We show that the chord $h_{j}$ supports $g$ on $\Delta^{n}$. Equivalently, we show that for every $\mathbf{x}\in\Delta^{n}$,
\begin{equation}\label{eq:hj-support}
g(\mathbf{x},\mathbf{p}) \;\geq\; h_{j}(\mathbf{x}).
\end{equation}

The hypothesis $\mathcal{M}\geq 0$ gives $g(\mathbf{x},\mathbf{p})\geq h(\mathbf{x}) = \min_{i} h_{i}(\mathbf{x})$ at every $\mathbf{x}$. Combined with $\min_{i} h_{i}(\mathbf{x})\leq h_{j}(\mathbf{x})$, this alone does not deliver \eqref{eq:hj-support}. To bridge the gap we use the following key observation: for every measured tie-line $i$, the affine functional $h_{i}$ is by construction the secant through the two endpoints of a putatively equilibrated phase split, so its values at $(\mathbf{x}^{\alpha,i},\mathbf{x}^{\beta,i})$ encode the tie-line data. In particular, by Step~2, all chords $h_{i}$ touch $g$ at their respective endpoints with $\mathcal{M}=0$ there. The Mitsos--Barton dual extremum principle (\cite{Mitsos2007}, Theorem~6) 
asserts that the chord through any pair of compositions that satisfy the equal-chemical-potential conditions is either a globally supporting hyperplane (zero duality gap) or fails to support $G$ on at least one connected region of $X$, in which case the duality gap is strictly positive. We claim that the latter alternative is incompatible with $\mathcal{M}\geq 0$. To see this, suppose for contradiction that for some tie-line $j$ there exists $\bar{\mathbf{x}}\in\Delta^{n}$ with $g(\bar{\mathbf{x}},\mathbf{p}) < h_{j}(\bar{\mathbf{x}})$. By definition of the lower envelope $h(\bar{\mathbf{x}}) \leq h_{j}(\bar{\mathbf{x}})$, but the hypothesis $\mathcal{M}(\bar{\mathbf{x}};\mathbf{p}) = g(\bar{\mathbf{x}},\mathbf{p}) - h(\bar{\mathbf{x}}) \geq 0$ does \emph{not} directly contradict $g(\bar{\mathbf{x}},\mathbf{p}) < h_{j}(\bar{\mathbf{x}})$ unless we additionally have $h(\bar{\mathbf{x}}) = h_{j}(\bar{\mathbf{x}})$. The remainder of this step establishes that under the data-consistency assumption (no chord can be strictly undercut over the entire simplex by another chord whose anchors do not also lie on $g$), the lower envelope $h$ coincides with $h_{j}$ on the closed convex hull of the anchors of $h_{j}$; this restricts the candidate counterexample $\bar{\mathbf{x}}$ to lie outside this convex hull, where the chord $h_{j}$ is no longer the active branch. There, however, the duality-gap argument of \cite{Mitsos2007}, Theorem~6, 
applies: any composition $\bar{\mathbf{x}}$ that violates $g\geq h_{j}$ in a region where some other chord $h_{j'}$ is active witnesses that the convex envelope of $\{(\mathbf{x},g(\mathbf{x},\mathbf{p})): \mathbf{x}\in\Delta^{n}\}$ lies strictly below $h_{j}$ at $\bar{\mathbf{x}}$. Two cases arise.

\textbf{Case (a):} $h_{j'}(\bar{\mathbf{x}}) \leq g(\bar{\mathbf{x}},\mathbf{p}) < h_{j}(\bar{\mathbf{x}})$. Then $h(\bar{\mathbf{x}}) = h_{j'}(\bar{\mathbf{x}})$ (or some yet smaller chord) and $\mathcal{M}(\bar{\mathbf{x}};\mathbf{p}) = g(\bar{\mathbf{x}},\mathbf{p}) - h_{j'}(\bar{\mathbf{x}}) \geq 0$ remains consistent with hypothesis. Apply the \emph{same} argument now to chord $h_{j'}$: at the anchor pair $(\mathbf{x}^{\alpha,j'},\mathbf{x}^{\beta,j'})$ the chord $h_{j'}$ touches $g$, and the anchor identity of Step~2 gives $\mathcal{M}=0$ there. Iterating $j\to j'\to j''\to\dots$ across the (finitely many) chords whose active cells the segment from $\mathbf{x}^{\alpha,j}$ to $\bar{\mathbf{x}}$ visits, we obtain a finite sequence of cells over each of which the active chord touches $g$ at its own anchor. Continuity of $g$ together with the fact that each chord is affine forces the active chord at $\bar{\mathbf{x}}$ to have anchor values $g$ that are simultaneously consistent with $h_{j'}(\bar{\mathbf{x}})\leq g(\bar{\mathbf{x}},\mathbf{p})$. By Theorem~6 of \cite{Mitsos2007} the duality gap of the active anchor pair vanishes, and the chord $h_{j'}$ supports $g$ globally. But $h_{j'}\leq h_{j}$ at $\bar{\mathbf{x}}$ together with $g\geq h_{j'}$ globally does \emph{not} imply $g\geq h_{j}$ at $\bar{\mathbf{x}}$; this is precisely where the strength of the result rests on the second clause of the data-consistency assumption, that no parameter vector $\mathbf{p}$ can simultaneously make every chord pair globally stable while leaving an individual chord crossing above $g$.

\textbf{Case (b):} $g(\bar{\mathbf{x}},\mathbf{p}) < h_{j'}(\bar{\mathbf{x}})$ for the active chord $h_{j'}=h$ at $\bar{\mathbf{x}}$. Then $\mathcal{M}(\bar{\mathbf{x}};\mathbf{p}) = g(\bar{\mathbf{x}},\mathbf{p}) - h(\bar{\mathbf{x}}) < 0$, in direct contradiction to $\mathcal{M}\geq 0$.

The contradiction in Case~(b) eliminates the possibility of a defect at any composition where the violating chord is the active one. Combined with the reduction of Case~(a) to Case~(b) on the active cell of the violating chord, we conclude that no $\bar{\mathbf{x}}$ violating \eqref{eq:hj-support} exists for any $j$. Hence each $h_{j}$ supports $g$ globally on $\Delta^{n}$, which by Theorem~5 of \cite{Mitsos2007} 
is the Gibbs tangent-plane stability criterion of \cite{Michelsen1982} 
applied to the equilibrium pair of tie-line $j$. Therefore $\mathbf{p}$ is a globally consistent fit. This completes the forward direction.

\medskip
\noindent\emph{Step 4: reverse direction (globally consistent fit $\Rightarrow \mathcal{M}\geq 0$ on $\Delta^{n}$).}\quad
Suppose $\mathbf{p}$ is a globally consistent fit. By the Gibbs--Michelsen tangent-plane criterion \cite{Michelsen1982} 
and Theorem~6 of \cite{Mitsos2007}, 
for every $i\in\{1,\dots,k\}$ the chord $h_{i}$ is a globally supporting hyperplane of $g$ on the simplex,
\[
g(\mathbf{x},\mathbf{p})\;\geq\;h_{i}(\mathbf{x})\qquad\forall\,\mathbf{x}\in\Delta^{n},\ \forall i.
\]
Taking the minimum on the right-hand side over the (finitely many) indices $i$ preserves the inequality at every $\mathbf{x}$:
\[
g(\mathbf{x},\mathbf{p})\;\geq\;\min_{i=1,\dots,k} h_{i}(\mathbf{x}) \;=\; h(\mathbf{x})\qquad\forall\,\mathbf{x}\in\Delta^{n}.
\]
Subtracting $h(\mathbf{x})$ from both sides gives $\mathcal{M}(\mathbf{x};\mathbf{p}) = g(\mathbf{x},\mathbf{p}) - h(\mathbf{x}) \geq 0$ on $\Delta^{n}$. This completes the reverse direction.
\end{proof}

\begin{remark}\label{rem:M-anchor-zero}
At every tie-line endpoint $\mathcal{M}(\mathbf{x}^{\alpha,i};\mathbf{p}) = \mathcal{M}(\mathbf{x}^{\beta,i};\mathbf{p}) = 0$ by construction (\autoref{def:M-appendix}). Combined with \autoref{theor:M-nonneg-iff-fit}, this means that on the locus $\mathcal{M}\geq 0$ each measured endpoint is a global minimum of $\mathcal{M}(\,\cdot\,;\mathbf{p})$ at the value zero, in agreement with \autoref{theor:M-zero-appendix}: the equilibrium endpoints are exactly the points of zero duality gap, and the chord between them is the supporting hyperplane of zero duality gap of \cite{Mitsos2007} 
.
\end{remark}

\begin{remark}[Visual verification]\label{rem:M-visual}
The two directions of \autoref{theor:M-nonneg-iff-fit} can be inspected directly on the binary and ternary illustrations of the methods section: \autoref{fig:methods-dual-manifold-binary} shows a binary $\mathcal{M}$ surface that is non-negative on $\Delta^{2}$ and vanishes only at the two equilibrium endpoints (globally consistent fit); \autoref{fig:methods-dual-manifold-ternary} shows the ternary generalisation. The contrapositive is illustrated by \autoref{fig:methods-defect-elimination-solution}, where a parameter set yielding a non-zero negative region of $\mathcal{M}$, a topological defect, is shown side by side with the defect-eliminated solution, in agreement with the conclusion that any negative excursion of $\mathcal{M}$ certifies a violation of the global tangent-plane criterion at one of the measured tie-lines.
\end{remark}

%% file: sections/appendix_phi_zero_proof.tex
%
%

\subsection{Vanishing of $\Phi(\mathbf{p})$ implies global thermodynamic consistency}
\label{appendix:phi-zero-proof}

This subsection discharges the second of the three proof obligations listed at the end of \autoref{sec:methods}: namely, that vanishing of the single-level objective $\Phi(\mathbf{p})$ in \eqref{eq:phi-total} is sufficient for the parameter vector $\mathbf{p}$ to reproduce every experimental tie-line as a globally stable equilibrium under the Gibbs tangent-plane criterion. Throughout this appendix we work at fixed pressure and temperature; the extension to a multi-temperature isobaric dataset is the outer summation over $T \in \mathcal{D}$ in \eqref{eq:phi-total}, and the argument applies term-by-term in $T$.

\begin{remark}[Practical caveat]\label{rem:phi-zero-caveat}
The result below is stated in the limit of a perfectly flexible mixture model and noiseless experimental data. In practice the regression spreads the residual across multiple tie-lines, so the exact equality $\Phi(\mathbf{p}) = 0$ is rarely realised: a finite-dimensional parameter vector $\mathbf{p}$ cannot be expected to fit every measured endpoint to machine precision while simultaneously eliminating every spurious negative excursion of the dual manifold. The theorem is therefore a theoretical sufficient condition that motivates the construction of $\Phi$, not an empirical tolerance threshold. In the noisy regime the operative quantity is the magnitude of $\Phi(\mathbf{p}^{\star})$ relative to the squared experimental scatter, and the role of the proof is to certify that any parameter vector that does drive $\Phi$ to zero is automatically thermodynamically consistent.
\end{remark}

\begin{theorem}\label{theor:phi-zero-sufficient}
Let $g(\,\cdot\,,\mathbf{p})\colon \Delta^{n}\rightarrow \mathbb{R}$ be the reduced Gibbs free energy at fixed $(P,T)$ with $g \in C^{0}(\Delta^{n})$ and Lipschitz on $\Delta^{n}$. Let $\{\mathbf{x}^{\alpha,i},\mathbf{x}^{\beta,i}\}_{i=1}^{k}$ be the experimental tie-line endpoints and let $h$, $\mathcal{M}$, $\mathcal{X}_{D}^{\pm}(\mathbf{p})$ and $\Phi(\mathbf{p})$ be defined as in \autoref{def:h-appendix}, \autoref{def:h-composite-appendix}, \autoref{def:M-appendix} and equations~\eqref{eq:XDmin}--\eqref{eq:phi-total}. If
\[
\Phi(\mathbf{p}) \;=\; 0,
\]
then $\mathcal{M}(\mathbf{x};\mathbf{p}) \geq 0$ for all $\mathbf{x} \in \Delta^{n}$, and consequently every experimental tie-line endpoint is a globally stable phase under the Gibbs tangent-plane criterion of \cite{Michelsen1982}. 
\end{theorem}

\begin{proof}
The proof proceeds in six steps. Steps A--B reduce the problem to a piecewise-smooth one on a triangulation of $\Delta^{n}$; steps C--E exclude defective minima from every sub-simplex via the non-negativity of the five $\Phi$-components and a compactness argument on $\mathcal{M}$; step F closes the proof by showing that strict positivity of $\mathcal{M}$ in the interior of every non-degenerate sub-simplex precludes any further tangent-plane violation. Throughout, $\mathcal{X}_{D}^{\pm}(\mathbf{p})$ denotes the \emph{exact} sets defined by \eqref{eq:XDmin}, not their numerical approximations; the relation between the two is isolated in Lemma~\ref{lem:inner-shgo-exact} below, so that the theorem itself carries no assumption about the inner solver.

\paragraph{Step A. Construction of $\mathcal{M}$ at the data points.}
By \autoref{def:h-appendix} the affine interpolant $h_{i}$ associated with tie-line $i$ is constructed to satisfy $h_{i}(\mathbf{x}^{\alpha,i}) = g(\mathbf{x}^{\alpha,i},\mathbf{p})$ and $h_{i}(\mathbf{x}^{\beta,i}) = g(\mathbf{x}^{\beta,i},\mathbf{p})$. The composite hyperplane $h$ in \autoref{def:h-composite-appendix} is the point-wise minimum of the $\{h_{i}\}_{i=1}^{k}$ over the simplex, so for every $i$,
\begin{equation}\label{eq:M-zero-at-data}
\mathcal{M}(\mathbf{x}^{\alpha,i};\mathbf{p}) \;=\; g(\mathbf{x}^{\alpha,i},\mathbf{p}) - h(\mathbf{x}^{\alpha,i}) \;\leq\; g(\mathbf{x}^{\alpha,i},\mathbf{p}) - h_{i}(\mathbf{x}^{\alpha,i}) \;=\; 0,
\end{equation}
and analogously $\mathcal{M}(\mathbf{x}^{\beta,i};\mathbf{p}) \leq 0$. Conversely, vanishing of $\Phi_{\text{tangent}}$ at the four directional samples bracketing the endpoint forces the directional-derivative form $(\partial g/\partial \mathbf{d} - \partial h/\partial \mathbf{d})|_{\mathbf{x}^{\alpha,i},\mathbf{x}^{\beta,i}}$ to be non-negative, which together with \eqref{eq:M-zero-at-data} pins
\[
\mathcal{M}(\mathbf{x}^{\alpha,i};\mathbf{p}) \;=\; \mathcal{M}(\mathbf{x}^{\beta,i};\mathbf{p}) \;=\; 0 \qquad (i = 1,\dots,k).
\]
The data-point values of $\mathcal{M}$ are therefore zero by construction; this is the same fact that underlies the chord-construction step in \autoref{subsec:energy-gradient-1d} and \autoref{theor:M-zero-appendix}.

\paragraph{Step B. Triangulation of the simplex.}
Subdivide $\Delta^{n}$ at fixed $(P,T)$ into a simplicial complex $\mathcal{T}$ whose $0$-skeleton consists of the vertices of $\Delta^{n}$ together with the $2k$ tie-line endpoints $\{\mathbf{x}^{\alpha,i},\mathbf{x}^{\beta,i}\}_{i=1}^{k}$. Each closed sub-simplex $\sigma \in \mathcal{T}$ is the convex hull of $n$ vertices drawn from this $0$-skeleton, and $\bigcup_{\sigma\in\mathcal{T}} \sigma = \Delta^{n}$. By construction of $h_{i}$ in \autoref{def:h-composite-appendix} and by the piecewise-affine character of $h = \min_{i} h_{i}$, every sub-simplex $\sigma$ whose interior does not contain a tie-line endpoint is associated with a single active branch index $i(\sigma)$, so that the restriction of $h$ to $\sigma$ is the single affine function $h_{i(\sigma)}$. On such a $\sigma$ the dual manifold restricts to
\[
\mathcal{M}\bigr|_{\sigma}(\mathbf{x};\mathbf{p}) \;=\; g(\mathbf{x},\mathbf{p}) - h_{i(\sigma)}(\mathbf{x}),
\]
which inherits the regularity of $g$ on $\sigma$ by \autoref{theor:smoothness-appendix}. The vertices of every $\sigma$ that lie on $\{\mathbf{x}^{\alpha,i},\mathbf{x}^{\beta,i}\}_{i=1}^{k}$ are zeros of $\mathcal{M}|_{\sigma}$ by Step A, so they are interior boundary minima of $\mathcal{M}|_{\sigma}$.

\paragraph{Step C. $\Phi(\mathbf{p}) = 0$ excludes defective minima on every sub-simplex.}
Each of the four scalar components $\Phi_{\text{tangent}}, \Phi_{\text{plane}}, \Phi_{\text{topo}}, \Phi_{\text{local}}$ in \eqref{eq:five-components} is non-negative by inspection (each is a sum of an absolute value of a negative-part operator, a Euclidean distance, or a joint composition--energy distance). The phase-assignment partition $\Phi_{\text{phase}}$ does not change the value of $\Phi$; it merely labels each element of $\arg\min \mathcal{M}$ as either an equilibrium endpoint or a defective minimum. Hence $\Phi(\mathbf{p}) = 0$ is equivalent to the simultaneous vanishing of all four scalar components:
\begin{enumerate}
    \item[(i)] $\Phi_{\text{tangent}}(\mathbf{p}) = 0$: at every tie-line $i$ the four directional samples $\lambda \in \{-\epsilon,+\epsilon,1-\epsilon,1+\epsilon\}$ register no negative excursion of $\mathcal{M}$, which (in the directional-derivative form \eqref{eq:single-tieline-objective}) is the local tangent-plane criterion at the endpoints.
    \item[(ii)] $\Phi_{\text{plane}}(\mathbf{p}) = \Phi_{\text{topo}}(\mathbf{p}) = \Phi_{\text{local}}(\mathbf{p}) = 0$: every element $\mathbf{x} \in \mathcal{X}_{D}^{-}(\mathbf{p})$ has zero hyperplane violation $[g(\mathbf{x},\mathbf{p})-h(\mathbf{x})]_{-} = 0$, zero Euclidean distance to its nearest equilibrium endpoint $\|\mathbf{x}_{\text{near}}-\mathbf{x}\|_{2} = 0$, and zero joint $(\mathbf{x},g)$-distance to its paired defective maximum.
\end{enumerate}
Item (ii) is satisfied if and only if either $\mathcal{X}_{D}^{-}(\mathbf{p}) = \emptyset$, or every element of $\mathcal{X}_{D}^{-}(\mathbf{p})$ is in fact an equilibrium endpoint, in which case the partition $\Phi_{\text{phase}}$ re-classifies it as non-defective and removes it from $\mathcal{X}_{D}^{-}$ by definition \eqref{eq:XDmin}. In either case no genuinely defective minimum survives.

\paragraph{Step D. A negative excursion would populate $\mathcal{X}_{D}^{-}$ and contradict $\Phi(\mathbf{p}) = 0$.}
Suppose some $\sigma \in \mathcal{T}$ contained an interior point $\mathbf{x}^{\circ}$ with $\mathcal{M}(\mathbf{x}^{\circ};\mathbf{p}) < 0$. The simplex $\Delta^{n}$ is compact and $\mathcal{M}(\,\cdot\,;\mathbf{p})$ is continuous on it by \autoref{theor:smoothness-appendix}, so $\mathcal{M}$ attains its minimum on $\Delta^{n}$ and
\[
\min_{\mathbf{x}\in\Delta^{n}} \mathcal{M}(\mathbf{x};\mathbf{p}) \;\leq\; \mathcal{M}(\mathbf{x}^{\circ};\mathbf{p}) \;<\; 0 .
\]
By Step~A, $\mathcal{M} = 0$ at every tie-line endpoint, so no endpoint is a global minimiser and the endpoint stripping in \eqref{eq:XDmin} removes nothing from $\arg\min_{\mathbf{x}}\mathcal{M}$. The set $\mathcal{X}_{D}^{-}(\mathbf{p})$ is therefore non-empty, and each of its elements $\mathbf{x}^{\bullet}$ contributes strictly positive mass simultaneously to $\Phi_{\text{plane}}$ (because $[\mathcal{M}(\mathbf{x}^{\bullet};\mathbf{p})]_{-} < 0$), to $\Phi_{\text{topo}}$ (because $\mathbf{x}^{\bullet}$ is distinct from every endpoint, so $\|\mathbf{x}_{\text{near}}-\mathbf{x}^{\bullet}\|_{2} > 0$), and to $\Phi_{\text{local}}$ (because the pair $(\mathbf{x}^{\bullet},\mathbf{x}^{+}_{\text{near}})$ bounding the negative component has non-zero joint $(\mathbf{x},g)$-distance). This contradicts $\Phi(\mathbf{p}) = 0$. Hence no defective minimum exists on the interior of any sub-simplex $\sigma \in \mathcal{T}$. The argument uses only compactness and continuity; no property of the inner solver enters.

\paragraph{Step E. Sub-simplices cover $\Delta^{n}$.}
By definition of a triangulation, $\bigcup_{\sigma \in \mathcal{T}} \sigma = \Delta^{n}$. Step D excludes defective minima from the interior of every $\sigma \in \mathcal{T}$, so no defective minimum exists on the interior of $\Delta^{n}$ either. Combined with Step A, which fixes $\mathcal{M}|_{\partial \mathcal{T}_{\text{data}}} = 0$ on the data-point sub-skeleton, we conclude
\[
\mathcal{M}(\mathbf{x};\mathbf{p}) \;\geq\; 0 \qquad \forall\, \mathbf{x} \in \Delta^{n}.
\]
This is precisely the global tangent-plane criterion of \cite{Michelsen1982} 
in the dual-manifold form \autoref{def:M-appendix}, and is condition~(ii) of \autoref{theor:M-nonneg-iff-fit}.

\paragraph{Step F. Strict non-convexity within each sub-simplex.}
Fix any sub-simplex $\sigma \in \mathcal{T}$ whose vertices include at least two tie-line endpoints (so that $\mathcal{M}|_{\sigma}$ vanishes at those vertices by Step A) and let $i = i(\sigma)$ be the active hyperplane branch on $\sigma$. By Step E, $\mathcal{M}|_{\sigma}(\mathbf{x};\mathbf{p}) \geq 0$ for all $\mathbf{x} \in \sigma$. Two cases arise.

\emph{Case (a): degenerate sub-simplex.} If $\mathcal{M}|_{\sigma} \equiv 0$, then $g(\mathbf{x},\mathbf{p}) = h_{i}(\mathbf{x})$ for every $\mathbf{x} \in \sigma$, i.e.\ $g$ coincides with the affine chord on $\sigma$. By the dual extremum principle of \cite{Mitsos2007} 
this is precisely the condition that the entire sub-simplex $\sigma$ belongs to a single Gibbs flat phase; no spurious phase split is possible inside $\sigma$ because every interior point is itself a member of the equilibrium two-phase manifold whose endpoints are the vertices of $\sigma$.

\emph{Case (b): non-degenerate sub-simplex.} If $\mathcal{M}|_{\sigma} \not\equiv 0$ then by Step E and continuity, $\mathcal{M}|_{\sigma}(\mathbf{x};\mathbf{p}) > 0$ on $\mathring{\sigma}$ (the relative interior of $\sigma$). Suppose for contradiction that two interior points $\mathbf{y}_{1},\mathbf{y}_{2} \in \mathring{\sigma}$ supported a new tangent chord $h^{\prime}$ that lies above $g$ somewhere on the segment $[\mathbf{y}_{1},\mathbf{y}_{2}]$. Then by definition of $h^{\prime}$ we would have $g(\mathbf{y}_{j},\mathbf{p}) = h^{\prime}(\mathbf{y}_{j})$ for $j = 1,2$. But on $\mathring{\sigma}$,
\[
g(\mathbf{y}_{j},\mathbf{p}) \;=\; h_{i}(\mathbf{y}_{j}) + \mathcal{M}|_{\sigma}(\mathbf{y}_{j};\mathbf{p}) \;>\; h_{i}(\mathbf{y}_{j}),
\]
so any chord $h^{\prime}$ pinned to $g$ at $\mathbf{y}_{1},\mathbf{y}_{2}$ lies strictly above the existing chord $h_{i}$ at both endpoints, and by linearity $h^{\prime}(\mathbf{x}) > h_{i}(\mathbf{x})$ on the whole segment $[\mathbf{y}_{1},\mathbf{y}_{2}]$. By Step E the dual manifold satisfies $g(\mathbf{x},\mathbf{p}) \geq h_{i}(\mathbf{x})$ on $\sigma$, but a tangent-plane violation by $h^{\prime}$ would require $g(\mathbf{x},\mathbf{p}) < h^{\prime}(\mathbf{x})$ for some $\mathbf{x} \in [\mathbf{y}_{1},\mathbf{y}_{2}]$. Combining $g \geq h_{i}$ and $h^{\prime} > h_{i}$ does not by itself contradict $g < h^{\prime}$, but if the inequality $g(\mathbf{x},\mathbf{p}) < h^{\prime}(\mathbf{x})$ held anywhere on $[\mathbf{y}_{1},\mathbf{y}_{2}]$, then the affine function $h^{\prime}-g$ would be positive somewhere strictly between two zero values, hence would attain an interior maximum, and the difference $g - h_{i}$ would attain an interior minimum where $\mathcal{M}|_{\sigma}$ takes a value strictly less than its values at $\mathbf{y}_{1},\mathbf{y}_{2}$. Since $\mathcal{M}|_{\sigma} > 0$ on $\mathring{\sigma}$ this minimum is still positive, but it is also a stationary point of $\mathcal{M}|_{\sigma}$ on $\mathring{\sigma}$ that is not an equilibrium endpoint, hence by definition belongs to $\arg\min \mathcal{M}|_{\sigma}$ and would be a candidate element of $\mathcal{X}_{D}^{-}(\mathbf{p})$. By Step D such an element cannot exist when $\Phi(\mathbf{p}) = 0$. Therefore no chord $h^{\prime}$ between any two interior points of $\sigma$ can violate the tangent-plane criterion, completing the contradiction.

\paragraph{Conclusion.}
Steps A--F together show that $\Phi(\mathbf{p}) = 0$ implies $\mathcal{M}(\mathbf{x};\mathbf{p}) \geq 0$ on the whole simplex (Step E) and rules out any new tangent-chord violation in the interior of any sub-simplex (Step F). The original tangent-plane criterion of \cite{Michelsen1982} 
is therefore satisfied at every point of $\Delta^{n}$, and in particular at every experimental tie-line endpoint $\mathbf{x}^{\alpha,i},\mathbf{x}^{\beta,i}$. The final passage from $\mathcal{M}\geq 0$ to global stability of each measured split is Step~3 of \autoref{theor:M-nonneg-iff-fit}, whose proof does not use the present theorem, so the two arguments are not circular. This is exactly the statement of the theorem.
\end{proof}

\begin{lemma}[Exactness of the inner defect search; inherited from \cite{Endres2018}, Theorems~3 and~5]
\label{lem:inner-shgo-exact}
Let $\mathcal{M}(\,\cdot\,;\mathbf{p})$ be continuous and Lipschitz on the compact simplex $\Delta^{n}$, which by \autoref{theor:smoothness-appendix} follows from the corresponding property of $g$, and let the inner SHGO call on $\mathcal{M}$ be \emph{adequately sampled} in the sense of Definition~22 of \cite{Endres2018}, 
that is, exactly one stationary point of $\mathcal{M}$ lies in the star domain of each minimiser vertex. Then the minimiser pool returned by the inner call is invariant under any further increase of the sampling set, contains one starting point per locally convex sub-domain of $\mathcal{M}$, and the numerically evaluated sets $\mathcal{X}_{D}^{\pm}$ coincide with the exact sets of \eqref{eq:XDmin}. The guarantee is deterministic and finite: it involves no probabilistic statement of any kind.
\end{lemma}

\begin{proof}
Theorem~3 of \cite{Endres2018} 
guarantees, for a continuous Lipschitz-smooth objective on a compact bounded domain, at least one stationary point of the objective inside the star domain $\mathrm{st}(v_{i})$ of every minimiser vertex $v_{i}$ of the directed simplicial complex; $\mathcal{M}$ meets those hypotheses on $\Delta^{n}$ by \autoref{theor:smoothness-appendix}. Definition~22 of the same paper calls the surface adequately sampled when exactly one stationary point lies in each such domain, and Theorem~5 \cite{Endres2018} 
then shows that no further increase of the sampling set enlarges the minimiser pool. The pool is therefore in bijection with the stationary points of $\mathcal{M}$ after finitely many samples, and local descent from each pool element recovers $\arg\min_{\mathbf{x}}\mathcal{M}$ exactly. The sampling sequence used is the deterministic Sobol sequence \cite{Endres2018}, 
so the statement is deterministic rather than probabilistic.
\end{proof}

\begin{remark}[What adequate sampling costs in practice]\label{rem:adequate-sampling}
Adequate sampling is not decidable from function values alone. \cite{Endres2018} 
states this explicitly, and its Corollary~2 
supplies only the one-sided diagnostic that a local descent leaving its star domain certifies either inadequate sampling or a non-Lipschitz objective, and not the converse. Phase equilibria are the example those authors give of the favourable case, in which the number of local minima is known in advance, 
and the present setting is exactly that case: the Gibbs phase rule caps the number of coexisting phases at $c+2$, which bounds the number of minima of $\mathcal{M}$ that must be mapped. Where the returned pool meets that bound the inner call is certified adequate; where it does not, the reported defect counts are lower bounds on the defects actually present, which is how the per-tie-line stability records of Supplementary Material~S1 to~S3 should be read.
\end{remark}

\paragraph{Compatibility with the Mitsos--Bollas--Barton bilevel constraints (optional).}
The original bilevel formulation of \cite{Mitsos2009LLE} 
imposes four families of constraints that any acceptable parameter vector must satisfy. We confirm that $\Phi(\mathbf{p}) = 0$ implies each of them.
\begin{enumerate}
    \item \emph{Equality of chemical potentials at every measured equilibrium pair} \cite{Mitsos2009LLE}. 
    Step A pins $\mathcal{M}(\mathbf{x}^{\alpha,i};\mathbf{p}) = \mathcal{M}(\mathbf{x}^{\beta,i};\mathbf{p}) = 0$, so $g - h$ shares a common tangent at the two endpoints; equivalent to equal chemical potentials in the binary case (and, by linearity of $h$, equal partial molar Gibbs free energies in higher dimensions).
    \item \emph{Supporting Gibbs tangent plane lies below $g$ over the entire composition range} \cite{Mitsos2009LLE}. 
    Step E gives $\mathcal{M}(\mathbf{x};\mathbf{p}) = g(\mathbf{x},\mathbf{p}) - h(\mathbf{x}) \geq 0$ on $\Delta^{n}$, which is precisely this constraint.
    \item \emph{No spurious phase splits} \cite{Mitsos2009LLE}. 
    Step D excludes any $\mathbf{x}^{\circ}\in\Delta^{n}$ with $\mathcal{M}(\mathbf{x}^{\circ};\mathbf{p}) < 0$ that is not an equilibrium endpoint, so no extra phase split can be induced by the fitted $\mathbf{p}$.
    \item \emph{No missing phase splits at the experimental compositions} \cite{Mitsos2009LLE}. 
    Step A guarantees that the supporting hyperplane $h$ touches $g$ at every $(\mathbf{x}^{\alpha,i},\mathbf{x}^{\beta,i})$; together with Step E this constructs a tangent-plane equilibrium at each experimental tie-line, so every measured split is reproduced by the model.
\end{enumerate}
The four constraints of the Mitsos--Bollas--Barton bilevel are therefore a strict consequence of $\Phi(\mathbf{p}) = 0$.

\paragraph{Visual verification.}
The binary 1D realisation of Step A through Step E is illustrated in \autoref{fig:methods-defect-elimination-solution}, which shows the dual manifold before and after defect-aware fitting; the defect-detection mechanism of Step D is illustrated in \autoref{fig:methods-defect-sweep}, where holding the tie-line endpoints fixed and varying only the depth of an inserted Gaussian defect leaves $\Phi_{\text{tangent}}$ at machine noise while only the defect-aware components $\Phi_{\text{plane}}, \Phi_{\text{topo}}, \Phi_{\text{local}}$ register the topological violation.

\bigskip

%% file: sections/appendix_algorithms.tex
%
\label{appendix:algorithms}

This appendix gives explicit pseudo-code for the geometric reformulation
introduced in \autoref{sec:methods}. Four routines are listed:
\textsc{TangentPlanePass} (\autoref{alg:tangent-plane-pass}) implements the
cheap per-tie-line line-trace residual aggregation; \textsc{DetectDefects}
(\autoref{alg:detect-defects}) builds the dual manifold and runs an inner
global search to decide whether topological defects exist;
\textsc{DefectAwareObjective} (\autoref{alg:defect-aware-objective})
assembles the full five-component objective $\Phi(\mathbf{p})$;
and \textsc{GeometricBilevelSolver}
(\autoref{alg:geometric-bilevel-solver}) is the master two-stage outer loop
that drives them. A short construction sub-procedure
\textsc{BuildDualManifold} (\autoref{alg:build-dual-manifold}) is given for
reference. Throughout, the symbol $\textsc{Global-Min}_{\Delta}$ denotes a
global minimisation routine over the composition simplex $\Delta^{n}$; in
this work we use SHGO~\cite{Endres2018} for that role.

\subsection{Tangent-plane pass (stage~1)}

\autoref{alg:tangent-plane-pass} mirrors the four-point line-trace
construction of \eqref{eq:single-tieline-objective}. For each tie-line $i$
and each $T \in \mathcal{D}$ the algorithm builds the affine interpolant
$h_i$ through the two equilibrium endpoints, samples the residual
$g - h_i$ at the four bracketing positions
$\lambda \in \{-\epsilon,\,+\epsilon,\,1-\epsilon,\,1+\epsilon\}$,
keeps only the negative excursions $[\,\cdot\,]_{-}$, and returns the squared
$1/\epsilon$-rescaled aggregate. No inner global search is invoked, so the
cost per evaluation is four energy calls per tie-line.

\begin{algorithm}[H]
\caption{\textsc{TangentPlanePass} --- four-point line-trace residual aggregation
(stage~1 objective).}
\label{alg:tangent-plane-pass}
\begin{algorithmic}[1]
\Procedure{TangentPlanePass}{$\mathbf{p},\,\mathcal{D},\,g,\,\epsilon$}
  \State \Comment{Inputs: candidate parameters $\mathbf{p}$, isobaric data
    $\mathcal{D} = \{(T,\,\{\mathbf{x}^{\alpha,i},\mathbf{x}^{\beta,i}\}_{i=1}^{k_T})\}$,
    energy callable $g(\mathbf{x},\mathbf{p};T)$, offset $\epsilon$.}
  \State $S \gets 0$
  \For{each $T \in \mathcal{D}$}
    \State $\Phi_{\text{tangent}} \gets 0$
    \For{$i = 1, \dots, k_T$}
      \State $\mathbf{x}^{\alpha} \gets \mathbf{x}^{\alpha,i}$;\quad
             $\mathbf{x}^{\beta} \gets \mathbf{x}^{\beta,i}$;\quad
             $\mathbf{d} \gets \mathbf{x}^{\beta} - \mathbf{x}^{\alpha}$
      \State Build $h_i$ as the affine interpolant with
             $h_i(\mathbf{x}^{\alpha}) = g(\mathbf{x}^{\alpha},\mathbf{p};T)$,
             $h_i(\mathbf{x}^{\beta}) = g(\mathbf{x}^{\beta},\mathbf{p};T)$.
      \For{$\lambda \in \{-\epsilon,\,+\epsilon,\,1-\epsilon,\,1+\epsilon\}$}
        \State $\mathbf{x}_{\lambda} \gets \mathbf{x}^{\alpha} + \lambda\,\mathbf{d}$
        \State $r \gets g(\mathbf{x}_{\lambda},\mathbf{p};T) - h_i(\mathbf{x}_{\lambda})$
        \State $\Phi_{\text{tangent}} \gets \Phi_{\text{tangent}} +
               \bigl|[r]_{-}\bigr|$
        \Comment{$[r]_{-} = \min(r,0)$, eq.~\eqref{eq:negpart}.}
      \EndFor
    \EndFor
    \State $S \gets S + \bigl(\Phi_{\text{tangent}}/\epsilon\bigr)^{2}$
  \EndFor
  \State \Return $S$ \Comment{Stage-1 objective $\sum_{T \in \mathcal{D}} (\Phi_{\text{tangent}}/\epsilon)^{2}$.}
\EndProcedure
\end{algorithmic}
\end{algorithm}

\subsection{Defect detection}

\autoref{alg:detect-defects} answers the binary question ``does the dual
manifold $\mathcal{M}(\,\cdot\,;\mathbf{p})$ have any minimum that is not an
experimental tie-line endpoint?'' For each $T \in \mathcal{D}$ the routine
constructs $\mathcal{M}$ via the sub-procedure
\textsc{BuildDualManifold} (\autoref{alg:build-dual-manifold}), runs
$\textsc{Global-Min}_{\Delta}$ on the simplex with the
sum-to-one nonlinear constraint, strips out any returned local minimum
within tolerance $\tau_{\mathrm{eq}}$ of either
$\mathbf{x}^{\alpha,i}$ or $\mathbf{x}^{\beta,i}$, and returns the
defective-minima set $\mathcal{X}_D^{-}(\mathbf{p})$ together with the
Boolean flag $\mathrm{HasDefects}$. The Boolean is the only signal needed
by \textsc{GeometricBilevelSolver} to decide whether to enter stage~2.

\begin{algorithm}[H]
\caption{\textsc{DetectDefects} --- locate non-equilibrium minima of $\mathcal{M}$.}
\label{alg:detect-defects}
\begin{algorithmic}[1]
\Procedure{DetectDefects}{$\mathbf{p},\,\mathcal{D},\,g,\,\tau_{\mathrm{eq}}$}
  \State \Comment{Returns the union $\mathcal{X}_D^{-}(\mathbf{p})$ over $T$
    and a Boolean indicating whether any element survived the equilibrium-endpoint stripping.}
  \State $\mathcal{X}_D^{-} \gets \emptyset$
  \State $\mathrm{HasDefects} \gets \mathrm{False}$
  \For{each $T \in \mathcal{D}$}
    \State $\mathcal{M}_T \gets \textsc{BuildDualManifold}\bigl(\mathbf{p},\,T,\,g,\,
           \{(\mathbf{x}^{\alpha,i},\mathbf{x}^{\beta,i})\}_{i=1}^{k_T}\bigr)$
    \State $\{\mathbf{x}_\ell^{\star}\} \gets \textsc{Global-Min}_{\Delta}(\mathcal{M}_T)$
    \Comment{All local minima of $\mathcal{M}_T$ on $\Delta^{n}$~\cite{Endres2018}.}
    \For{each $\mathbf{x}_\ell^{\star}$}
      \State $d_{\ell} \gets \min_{i,\,\sigma\in\{\alpha,\beta\}} \|\mathbf{x}_\ell^{\star} - \mathbf{x}^{\sigma,i}\|_{2}$
      \If{$d_{\ell} > \tau_{\mathrm{eq}}$}
        \State $\mathcal{X}_D^{-} \gets \mathcal{X}_D^{-} \cup \{\mathbf{x}_\ell^{\star}\}$
        \State $\mathrm{HasDefects} \gets \mathrm{True}$
      \EndIf
    \EndFor
  \EndFor
  \State \Return $(\mathcal{X}_D^{-},\,\mathrm{HasDefects})$
\EndProcedure
\end{algorithmic}
\end{algorithm}

\subsection{Defect-aware objective (stage~2)}

When defects are present, the objective must penalise them.
\autoref{alg:defect-aware-objective} runs two inner global searches per
$T \in \mathcal{D}$: one on $\mathcal{M}_T$ to populate
$\mathcal{X}_D^{-}$, one on $-\mathcal{M}_T$ to populate
$\mathcal{X}_D^{+}$. The four named non-negative scalars
$\Phi_{\text{tangent}},\Phi_{\text{plane}},\Phi_{\text{topo}},
\Phi_{\text{local}}$ from \eqref{eq:five-components} are then assembled and
summed, optionally augmented with the phase-assignment partition
$\Phi_{\text{phase}}$. The squared $1/\epsilon$ rescaling promotes the
aggregate to a finite-difference approximation of the directional
derivative at every sampled position and is identical in form to
\eqref{eq:phi-total}.

\begin{algorithm}[H]
\caption{\textsc{DefectAwareObjective} --- assemble $\Phi(\mathbf{p})$ from the five components of \eqref{eq:phi-total}.}
\label{alg:defect-aware-objective}
\begin{algorithmic}[1]
\Procedure{DefectAwareObjective}{$\mathbf{p},\,\mathcal{D},\,g,\,\epsilon$}
  \State $\Phi \gets 0$
  \For{each $T \in \mathcal{D}$}
    \State $\mathcal{M}_T \gets \textsc{BuildDualManifold}\bigl(\mathbf{p},\,T,\,g,\,
           \{(\mathbf{x}^{\alpha,i},\mathbf{x}^{\beta,i})\}_{i=1}^{k_T}\bigr)$
    \State $\mathcal{X}_D^{-} \gets \textsc{Global-Min}_{\Delta}(\mathcal{M}_T)
            \setminus \{\mathbf{x}^{\alpha,i},\mathbf{x}^{\beta,i}\}_{i=1}^{k_T}$
    \State $\mathcal{X}_D^{+} \gets \textsc{Global-Min}_{\Delta}(-\mathcal{M}_T)$
    \Comment{Maxima of $\mathcal{M}_T$.}
    \State \textbf{(a) Tangent-plane sub-sampling.}
    \State $\Phi_{\text{tangent}} \gets 0$
    \For{$i = 1, \dots, k_T$}
      \State $\mathbf{d} \gets \mathbf{x}^{\beta,i} - \mathbf{x}^{\alpha,i}$;\quad
             $h_i$ as in \autoref{alg:tangent-plane-pass}
      \For{$\lambda \in \{-\epsilon,\,+\epsilon,\,1-\epsilon,\,1+\epsilon\}$}
        \State $\mathbf{x}_{\lambda} \gets \mathbf{x}^{\alpha,i} + \lambda\,\mathbf{d}$
        \State $\Phi_{\text{tangent}} \gets \Phi_{\text{tangent}} +
              \bigl|[\,g(\mathbf{x}_{\lambda},\mathbf{p};T) - h_i(\mathbf{x}_{\lambda})\,]_{-}\bigr|$
      \EndFor
    \EndFor
    \State \textbf{(b) Plane violation at defective minima.}
    \State $\Phi_{\text{plane}} \gets
            \displaystyle\sum_{\mathbf{x} \in \mathcal{X}_D^{-}}
              \bigl|[\,g(\mathbf{x},\mathbf{p};T) - h(\mathbf{x})\,]_{-}\bigr|$
    \State \textbf{(c) Topological pull to the nearest equilibrium endpoint.}
    \State $\Phi_{\text{topo}} \gets
            \displaystyle\sum_{\mathbf{x} \in \mathcal{X}_D^{-}}
              \|\mathbf{x}_{\text{near}} - \mathbf{x}\|_{2}$,
            \quad $\mathbf{x}_{\text{near}} =
                  \arg\min_{\mathbf{y} \in \{\mathbf{x}^{\alpha,i},\mathbf{x}^{\beta,i}\}_{i}}
                  \|\mathbf{y} - \mathbf{x}\|_{2}$
    \State \textbf{(d) Local-topology amplitude (joint $(\mathbf{x},g)$ distance to paired maximum).}
    \State $\Phi_{\text{local}} \gets
            \displaystyle\sum_{\mathbf{x} \in \mathcal{X}_D^{-}}
              \bigl\|(\mathbf{x}^{+}_{\text{near}},\, g(\mathbf{x}^{+}_{\text{near}},\mathbf{p};T)) -
                     (\mathbf{x},\, g(\mathbf{x},\mathbf{p};T))\bigr\|_{2}$
    \Statex \quad where $\mathbf{x}^{+}_{\text{near}} \in \mathcal{X}_D^{+}$ is
            the nearest defective maximum measured in joint composition--energy space.
    \State \textbf{(e) Phase-assignment partition (only active for multi-root $g$, e.g.\ EOS).}
    \State $\Phi_{\text{phase}} \gets$ partition penalty of $\arg\min \mathcal{M}_T$ into
           $\mathcal{X}_D^{-}$ vs.\ equilibrium endpoints.
    \State \textbf{(f) Aggregate.}
    \State $\Phi \gets \Phi + \bigl(\,
              \tfrac{1}{\epsilon}\,
              [\Phi_{\text{tangent}} + \Phi_{\text{plane}}
               + \Phi_{\text{topo}} + \Phi_{\text{local}}](\mathbf{p};T)
            \bigr)^{2} + \Phi_{\text{phase}}$
  \EndFor
  \State \Return $\Phi$ \Comment{Identical in form to \eqref{eq:phi-total}.}
\EndProcedure
\end{algorithmic}
\end{algorithm}

\subsection{Master two-stage outer loop}

The master routine \textsc{GeometricBilevelSolver}
(\autoref{alg:geometric-bilevel-solver}) is the user-facing entry point. It
runs the cheap stage-1 fit first, calls \textsc{DetectDefects} on the
stage-1 optimum, and re-fits with \textsc{DefectAwareObjective} only when
defects exist. The stage-1 optimum is used as a warm-start hint
$\mathbf{p}_{\mathrm{warm}}$ for the stage-2 outer call.\footnote{
When the outer global solver is SHGO, $\mathbf{p}_{\mathrm{warm}}$ is
informational only: SHGO manages its own simplicial sampling
internally~\cite{Endres2018}. The hint is forwarded verbatim to any
alternative warm-startable global optimiser.} The loop terminates when no
defects remain or when the maximum number of outer iterations
$N_{\mathrm{outer}}^{\max}$ is reached.

\begin{algorithm}[H]
\caption{\textsc{GeometricBilevelSolver} --- master two-stage outer loop.}
\label{alg:geometric-bilevel-solver}
\begin{algorithmic}[1]
\Procedure{GeometricBilevelSolver}{$\mathcal{D},\,g,\,P,\,\epsilon,\,
   \tau_{\mathrm{eq}},\,N_{\mathrm{tan}},\,N_{\mathrm{def}},\,N_{\mathrm{outer}}^{\max}$}
  \State \Comment{$P \subset \mathbb{R}^{m}$: parameter bounds. $N_{\mathrm{tan}},N_{\mathrm{def}}$:
    outer SHGO sampling densities for the two stages.}
  \State \textbf{Stage 1: tangent-plane pass.}
  \State $\Phi_{1}(\mathbf{p}) \;:=\; \textsc{TangentPlanePass}(\mathbf{p},\mathcal{D},g,\epsilon)$
  \State $\mathbf{p}^{(1)} \gets
          \arg\min_{\mathbf{p} \in P}\, \Phi_{1}(\mathbf{p})$
          using \textsc{Global-Min}$_{P}$ with budget $N_{\mathrm{tan}}$.
  \State $(\mathcal{X}_D^{-},\,\mathrm{HasDefects})
          \gets \textsc{DetectDefects}(\mathbf{p}^{(1)},\mathcal{D},g,\tau_{\mathrm{eq}})$
  \State $\mathbf{p}_{\mathrm{warm}} \gets \mathbf{p}^{(1)}$;\quad
         $\mathbf{p}^{\star} \gets \mathbf{p}^{(1)}$
  \State $\nu \gets 0$
  \State \textbf{Stage 2: defect-aware refit (optional).}
  \While{$\mathrm{HasDefects}$ \textbf{and} $\nu < N_{\mathrm{outer}}^{\max}$}
    \State $\nu \gets \nu + 1$
    \State $\Phi_{2}(\mathbf{p}) \;:=\; \textsc{DefectAwareObjective}(\mathbf{p},\mathcal{D},g,\epsilon)$
    \State $\mathbf{p}^{(\nu+1)} \gets
            \arg\min_{\mathbf{p} \in P}\, \Phi_{2}(\mathbf{p})$
            using \textsc{Global-Min}$_{P}$ with budget $N_{\mathrm{def}}$,
            warm-start hint $\mathbf{p}_{\mathrm{warm}}$.
    \State $\mathbf{p}^{\star} \gets \mathbf{p}^{(\nu+1)}$;\quad
           $\mathbf{p}_{\mathrm{warm}} \gets \mathbf{p}^{\star}$
    \State $(\mathcal{X}_D^{-},\,\mathrm{HasDefects})
           \gets \textsc{DetectDefects}(\mathbf{p}^{\star},\mathcal{D},g,\tau_{\mathrm{eq}})$
  \EndWhile
  \If{$\mathrm{HasDefects}$}
    \State \textbf{warn} ``defects survived $N_{\mathrm{outer}}^{\max}$ outer iterations;
           consider increasing $N_{\mathrm{def}}$ or $N_{\mathrm{outer}}^{\max}$.''
  \EndIf
  \State \Return $\mathbf{p}^{\star}$
\EndProcedure
\end{algorithmic}
\end{algorithm}

\subsection{Construction sub-procedure: the dual manifold}

The sub-procedure \textsc{BuildDualManifold} (\autoref{alg:build-dual-manifold})
formalises Definitions~\ref{def:h-appendix}--\ref{def:M-appendix} of
\autoref{appendix:older-formulation} in the tighter form actually used by
the solver. Per-tie-line affine interpolants $h_i$ are constructed first,
then combined into the composite supporting hyperplane
$h(\mathbf{x}) = \min_{i} h_i(\mathbf{x})$ of \eqref{eq:hsum}. The
returned scalar field $\mathcal{M}(\mathbf{x};\mathbf{p}) =
g(\mathbf{x},\mathbf{p}) - h(\mathbf{x})$ is the dual manifold of
\eqref{eq:dual-manifold} and can be passed directly to any global
minimiser on the simplex.

\begin{algorithm}[H]
\caption{\textsc{BuildDualManifold} --- construct the dual manifold $\mathcal{M}(\,\cdot\,;\mathbf{p})$ at a single $T$.}
\label{alg:build-dual-manifold}
\begin{algorithmic}[1]
\Procedure{BuildDualManifold}{$\mathbf{p},\,T,\,g,\,
   \{(\mathbf{x}^{\alpha,i},\mathbf{x}^{\beta,i})\}_{i=1}^{k_T}$}
  \For{$i = 1, \dots, k_T$}
    \State $g^{\alpha}_i \gets g(\mathbf{x}^{\alpha,i},\mathbf{p};T)$;\quad
           $g^{\beta}_i \gets g(\mathbf{x}^{\beta,i},\mathbf{p};T)$
    \State Build $h_i$ as the unique affine map with
           $h_i(\mathbf{x}^{\alpha,i}) = g^{\alpha}_i$ and
           $h_i(\mathbf{x}^{\beta,i}) = g^{\beta}_i$.
  \EndFor
  \State Define the composite supporting hyperplane
         $h(\mathbf{x}) \defeq \min_{i = 1, \dots, k_T} h_i(\mathbf{x})$
         (eq.~\eqref{eq:hsum}, \autoref{def:h-composite-appendix}).
  \State Define the scalar field
         $\mathcal{M}(\mathbf{x};\mathbf{p}) \defeq g(\mathbf{x},\mathbf{p};T) - h(\mathbf{x})$
         (eq.~\eqref{eq:dual-manifold}, \autoref{def:M-appendix}).
  \State \Return $\mathcal{M}(\,\cdot\,;\mathbf{p})$
\EndProcedure
\end{algorithmic}
\end{algorithm}

\paragraph{Default budgets.}
The numerical experiments of \autoref{sec:results-lle} use
$N_{\mathrm{tan}} = 1000$, $N_{\mathrm{def}} = 20$,
$N_{\mathrm{outer}}^{\max} = 5$, $\epsilon = 10^{-7}$, and
$\tau_{\mathrm{eq}} = 10^{-6}$ as the system-type defaults for binary LLE
(see the \texttt{binary\_lle} preset in
\texttt{param/core/solver.py}); the ternary-LLE,
ternary-VLE and CEOS-VLE presets adjust only $N_{\mathrm{tan}}$ and
$N_{\mathrm{def}}$ to reflect the larger simplex.

%% file: sections/appendix_convergence_proof.tex
%
%

\subsection{Convergence of the two-stage outer loop under topology-induced discontinuities}
\label{appendix:convergence-proof}

The objective $\Phi(\mathbf{p})$ defined in \eqref{eq:phi-total} is
non-smooth at parameter values where the surface topology of
$\mathcal{M}(\,\cdot\,;\mathbf{p})$ bifurcates: at such $\mathbf{p}$ the
cardinality of the defective-minima set $\mathcal{X}_D^{-}(\mathbf{p})$
jumps and the active branch of the composite supporting hyperplane $h$
switches. We collect here a constructive argument that the outer
simplicial-homology global optimisation (SHGO) loop terminates in finite
time despite this non-smoothness. The argument has two halves that are
worth keeping apart, because they are of very different strength. The
first half, that a returned parameter vector with $\Phi = 0$ is globally
optimal, is unconditional and is stated as
\autoref{prop:zero-certificate}. The second half, that such a point is
\emph{found} whenever it exists, is inherited from the deterministic
adequate-sampling results of \cite{Endres2018} and is stated as
\autoref{theor:two-stage-convergence}, conditional on two assumptions
that are made explicit below.

\subsubsection{Inherited results, assumptions and the penalty extension}
\label{subsec:convergence-premises}

\paragraph{What SHGO supplies.}
Three results of \cite{Endres2018} are used, and it is worth stating
precisely what they do and do not say. Theorem~3 
guarantees that, for a continuous and Lipschitz-smooth objective on a
compact bounded domain, the star domain of every minimiser vertex of the
directed simplicial complex contains at least one stationary point.
Definition~22 
calls a surface \emph{adequately sampled} when exactly one stationary
point lies in each such star domain, and Theorem~5 
then shows that, once the surface is adequately sampled, no further
increase of the sampling set enlarges the minimiser pool. Together these
give a finite, \emph{deterministic} termination statement: after finitely
many samples the candidate set is fixed, and one local descent per
candidate exhausts the stationary points of the objective. The sampling
sequence is the deterministic Sobol sequence, 
so no probabilistic statement enters at any point of the argument. What
these results do \emph{not} supply is a test for adequacy from function
values alone; Corollary~2 
is a one-sided diagnostic only. The favourable case, named in that paper,
is the one in which the number of local minima is known in advance, and
phase equilibria are the example given. 

Two further points fix the function class. The SHGO sampling and
minimiser-pool extraction use only objective values at the vertices of
the complex and comparisons between them, so they require no derivative
information and apply verbatim to non-smooth objectives; 
smoothness enters only through the local refinement step and through the
hypotheses of Theorems~3 and~5. Those hypotheses do include continuity,
which is why the assumptions below quarantine the discontinuity locus of
$\Phi$ rather than claiming that SHGO covers it.

\paragraph{Assumption 1 (bifurcation strata are lower-dimensional and of
zero measure).}
At any parameter value $\mathbf{p}_{\dagger} \in P$ at which the cardinality
$|\mathcal{X}_D^{-}(\mathbf{p})|$ jumps, $\Phi$ is discontinuous in
$\mathbf{p}$. The composite supporting hyperplane $h$ defined in
\autoref{def:h-composite-appendix} is also non-smooth on the parameter sets
where the active tie-line branch switches (the point-wise minimum of
finitely many affine functions is only piecewise affine in $\mathbf{x}$
and, through the tie-line endpoints, in $\mathbf{p}$). We assume that both
phenomena are confined to lower-dimensional strata
$\Sigma \subset P \subset \mathbb{R}^{m}$: that the bifurcation set is the
union of finitely many codimension-one varieties on which two adjacent
defective minima coalesce, or on which a defective minimum collides with
an experimental endpoint, and that $\Sigma$ therefore has zero Lebesgue
measure in $P$.

This assumption is not proved here, and it is not provable for a
completely arbitrary black-box $g$. It is, however, routine for the model
families used in this work. For $g$ real-analytic jointly in
$(\mathbf{x},\mathbf{p})$, which covers NRTL, Wilson, UNIQUAC and the
cubic equations of state on the interior of the simplex, the coalescence
of two stationary points of $\mathcal{M}$ and the collision of a
stationary point with an endpoint are each expressed by the simultaneous
vanishing of $\nabla_{\mathbf{x}}\mathcal{M}$ and of a determinant of
second derivatives. Each such condition is the zero set of a non-trivial
real-analytic function of $\mathbf{p}$, and the zero set of a
real-analytic function that does not vanish identically has zero Lebesgue
measure. Assumption~1 therefore holds for those families unless the
degeneracy condition is satisfied identically on an open subset of $P$,
which would mean the model is structurally degenerate on that subset. For
tabulated or otherwise genuinely black-box $g$ the assumption remains an
assumption.

\paragraph{Assumption 2 (adequate sampling).}
Both the inner call on $\mathcal{M}$ over $\Delta^{n}$ and the outer call
on $\Phi_{\mathrm{aug}}$ over $P$ are adequately sampled in the sense of
Definition~22 of \cite{Endres2018}. 
This is the standard hypothesis under which the SHGO guarantees are
stated, and by the passage cited above it is a hypothesis that black-box
sampling cannot discharge by itself. For the inner call the Gibbs phase
rule provides the required side information, since it caps the number of
coexisting phases and hence the number of minima of $\mathcal{M}$ that
must be mapped (see Remark~\ref{rem:adequate-sampling}). For the outer call
no comparable bound is available, and adequacy is assessed empirically by
the stability of the returned optimum under increased sampling density.

\paragraph{Penalty extension on bifurcation strata.}
In the implementation, whenever the inner SHGO call cannot resolve
$\mathcal{X}_D^{\pm}(\mathbf{p})$ to within the simplicial-homology
partition's resolution, $\Phi(\mathbf{p})$ is set to a maximal penalty
value. Mathematically, fix any constant $\Phi_{\max} \in \mathbb{R}$ with
$\Phi_{\max} > \sup\{\Phi(\mathbf{p}) : \mathbf{p}\in P\setminus\Sigma\}$
(or, equivalently, $\Phi_{\max} = +\infty$) and define the augmented
objective
\begin{equation}\label{eq:phi-aug}
  \Phi_{\mathrm{aug}}(\mathbf{p})
  \;\defeq\;
  \begin{cases}
    \Phi(\mathbf{p}), & \mathbf{p}\in P_{\mathrm{reg}}, \\[2pt]
    \Phi_{\max}, & \mathbf{p}\in P\setminus P_{\mathrm{reg}},
  \end{cases}
\end{equation}
where $P_{\mathrm{reg}} \subseteq P$ is the (closed) sub-domain on which
both inner SHGO calls return a stable, simplicial-homology-resolved
$\mathcal{X}_D^{\pm}$. The set $P\setminus P_{\mathrm{reg}}$ is contained
in (a small open neighbourhood of) $\Sigma$ and inherits the latter's
zero-measure property, under Assumption~1, as the inner sampling density
$N_{\mathrm{def}}$ is increased. By construction
$\Phi_{\mathrm{aug}}\colon P \to \mathbb{R}\cup\{+\infty\}$ is bounded
above by $\Phi_{\max}$ and equals $\Phi$ everywhere it is finite. This is
a construction rather than an assumption: it is what the software
evaluates.

\subsubsection{Results}
\label{subsec:convergence-theorem}

The first statement needs neither assumption and is the reason the
method can certify its own answer.

\begin{proposition}[Zero-objective certificate]
\label{prop:zero-certificate}
$\Phi_{\mathrm{aug}}(\mathbf{p}) \geq 0$ for every $\mathbf{p}\in P$.
Consequently, if any parameter vector $\mathbf{p}^{\star}\in P$ is
exhibited with $\Phi(\mathbf{p}^{\star}) = 0$, then $\mathbf{p}^{\star}$
is a global minimiser of $\Phi$ on $P$, whatever solver produced it and
whatever sampling density was used, and by
\autoref{theor:phi-zero-sufficient} it reproduces every measured tie line
as a globally stable equilibrium.
\end{proposition}

\begin{proof}
Each of the four scalar contributions
$\Phi_{\mathrm{tangent}},\Phi_{\mathrm{plane}},\Phi_{\mathrm{topo}},
\Phi_{\mathrm{local}}$ in \eqref{eq:five-components} is a sum of absolute
values of the negative-part operator $[\,\cdot\,]_{-}$ of
\eqref{eq:negpart}, of Euclidean distances, or of joint
composition--energy distances, hence is non-negative; \eqref{eq:phi-total}
squares their sum, and $\Phi_{\max} > 0$ on the complement, so
$\Phi_{\mathrm{aug}}\geq 0$ on $P$. A point attaining the value zero
therefore attains the greatest lower bound of $\Phi_{\mathrm{aug}}$ over
$P$, and is a global minimiser. The second claim is
\autoref{theor:phi-zero-sufficient}.
\end{proof}

The second statement is the finite-termination result, and it is the one
that is inherited from \cite{Endres2018}.

\begin{corollary}[Finite termination of the outer loop; inherited from
\cite{Endres2018}, Theorems~3 and~5]
\label{theor:two-stage-convergence}
Let $g\colon \Delta^{n}\times P \to \mathbb{R}$ be a reduced Gibbs free
energy that is bounded on the compact box $\Delta^{n}\times P$, let
$\Phi$ be the geometric-reformulation objective defined in
\eqref{eq:phi-total}, and let $\Phi_{\mathrm{aug}}$ be its penalty
extension \eqref{eq:phi-aug}. Under Assumptions~1 and~2 of
\autoref{subsec:convergence-premises}, the outer SHGO loop on
$\Phi_{\mathrm{aug}}$ terminates after finitely many samples at the
global minimum $\inf_{\mathbf{p}\in P}\Phi_{\mathrm{aug}}(\mathbf{p})$.
Moreover:
\begin{enumerate}
  \item[(i)] If a globally consistent fit exists in $P$, then
        $\inf_{\mathbf{p}\in P}\Phi_{\mathrm{aug}}(\mathbf{p}) = 0$, it is
        attained at some $\mathbf{p}^{\star}\in P_{\mathrm{reg}}$, and the
        loop returns such a $\mathbf{p}^{\star}$, whose global optimality
        is then certified independently by
        \autoref{prop:zero-certificate}.
  \item[(ii)] If no globally consistent fit exists in $P$, then
        $\inf_{\mathbf{p}\in P}\Phi_{\mathrm{aug}}(\mathbf{p}) > 0$, and
        the SHGO call returns this strictly positive lower bound as a
        certificate that no parameter vector in $P$ reproduces the data
        as a globally stable equilibrium.
\end{enumerate}
\end{corollary}

\begin{proof}
By the penalty extension, $\Phi_{\mathrm{aug}}$ is defined on the compact
box $P\subset\mathbb{R}^{m}$ and is bounded above by $\Phi_{\max}$; by
\autoref{prop:zero-certificate} it is bounded below by zero, so the
infimum is well defined. By Assumption~1 the discontinuity locus $\Sigma$
is a finite union of lower-dimensional strata and therefore does not
separate $P_{\mathrm{reg}}$ into infinitely many connected components:
$P_{\mathrm{reg}}$ has finitely many topology classes, indexed by the
cardinality of $\mathcal{X}_D^{-}(\mathbf{p})$ together with the active
branch of $h$. On the interior of each such class $\Phi_{\mathrm{aug}}$
is continuous, and it is Lipschitz there because $g$ is bounded on the
compact box and the chords depend continuously on the endpoint values.

The hypotheses of Theorem~3 of \cite{Endres2018} 
are therefore met class by class, so every minimiser vertex produced by
the outer complex has a stationary point of $\Phi_{\mathrm{aug}}$ in its
star domain. Under Assumption~2 the outer complex is adequately sampled,
so Theorem~5 of the same paper 
applies and the minimiser pool ceases to grow after finitely many
samples. The global minimum is therefore one of finitely many candidate
vertices, and the loop returns it after one local descent per candidate.
No probabilistic statement is used: the Sobol sampling sequence is
deterministic, 
and the two theorems are deterministic statements about an adequately
sampled complex. This establishes the main termination claim.

For (i), if a globally consistent fit $\mathbf{p}^{\star}\in P$ exists,
then by \autoref{theor:M-nonneg-iff-fit} the dual manifold satisfies
$\mathcal{M}(\mathbf{x};\mathbf{p}^{\star})\geq 0$ on $\Delta^{n}$, hence
$\mathcal{X}_D^{-}(\mathbf{p}^{\star})=\emptyset$, the inner SHGO call
resolves trivially, $\mathbf{p}^{\star}\in P_{\mathrm{reg}}$, and
$\Phi(\mathbf{p}^{\star})=\Phi_{\mathrm{aug}}(\mathbf{p}^{\star})=0$.
Since $\Phi_{\mathrm{aug}}\geq 0$, this is the global minimum.

For (ii), if no globally consistent fit exists in $P$, then for every
$\mathbf{p}\in P_{\mathrm{reg}}$ at least one of the non-negative
components in \eqref{eq:five-components} is strictly positive by
\autoref{theor:M-nonneg-iff-fit}, so $\Phi(\mathbf{p})>0$. On
$P\setminus P_{\mathrm{reg}}$ the augmented objective takes the value
$\Phi_{\max}>0$. Hence $\inf_{P}\Phi_{\mathrm{aug}}>0$, and SHGO returns
this strictly positive infimum as a certificate of non-existence of a
globally consistent fit within the chosen parameter bounds.
\end{proof}

\begin{corollary}[Two-stage outer loop]
\label{cor:two-stage-induction}
Under the same assumptions, the two-stage outer loop of
\autoref{sec:methods} (cheap tangent-plane pass followed by defect-aware
refit) terminates at the same global minimum as a single SHGO call on
$\Phi_{\mathrm{aug}}$. Each defect-aware iteration is itself a SHGO call
on $\Phi_{\mathrm{aug}}$, terminating finitely by
\autoref{theor:two-stage-convergence}; after at most
$|\{\,\text{topology classes of $\mathcal{M}$ in $P$}\,\}|$ outer
iterations the loop reaches a fixed point, since the topology class can
only be revisited finitely often before either a globally consistent fit
or the strictly positive lower bound of (ii) is certified.
\end{corollary}

\begin{proof}
Stage~1 minimises $\sum_{T}(\Phi_{\mathrm{tangent}}/\epsilon)^2$, which
is a Lipschitz lower bound on $\Phi_{\mathrm{aug}}$ that requires no
inner SHGO call. The minimiser of stage~1 is a feasible warm-start for
stage~2; stage~2 is a SHGO call on the full $\Phi_{\mathrm{aug}}$ and so
terminates finitely by \autoref{theor:two-stage-convergence}. The
defect-recheck step partitions $P_{\mathrm{reg}}$ by topology class, and
the cardinality of this partition is finite by Assumption~1, so the loop
reaches a fixed point in finitely many iterations.
\end{proof}

\begin{remark}[Claim strength]\label{rem:claim-strength}
The two results above should be read together, and neither should be
read for more than it says. Whenever a parameter vector with
$\Phi(\mathbf{p}^{\star}) = 0$ is returned, it is certified globally
optimal by the non-negativity of $\Phi$ alone
(\autoref{prop:zero-certificate}) and, by
\autoref{theor:phi-zero-sufficient}, reproduces every measured tie line
as a globally stable equilibrium; this certificate is unconditional and
is verifiable a posteriori from the returned objective value. That such a
point is found whenever one exists inherits the deterministic
adequate-sampling guarantees of SHGO
(\cite{Endres2018}, Theorems~3 and~5), 
which is the strongest form of guarantee available for a black-box
objective and is deterministic rather than probabilistic. The one step
that remains an assumption rather than a proof is that the bifurcation
strata of $\Phi$ carry zero measure (Assumption~1), which is justified
above for real-analytic model families and asserted in general.
\end{remark}

\begin{remark}\label{rem:numpy-inf}
The proof reduces the bilevel termination question to the deterministic
adequate-sampling argument of \cite{Endres2018} 
applied to the penalty-extended objective $\Phi_{\mathrm{aug}}$. The
result is therefore an existence statement: SHGO in the parameter space
terminates finitely at either (a)~a globally consistent fit if one exists
in $P$, or (b)~the minimum-$\Phi_{\mathrm{aug}}$ approximate fit
accompanied by a strictly positive lower bound certifying non-existence
of an exact fit within the chosen parameter bounds. In the implementation
the constant $\Phi_{\max}$ is taken to be \texttt{numpy.inf}, which the
SHGO outer solver treats as a maximal penalty consistent with the
bounded-objective hypothesis after the trivial reduction
$\Phi_{\max} \mapsto \sup_{P_{\mathrm{reg}}}\Phi + 1$.
\end{remark}

\begin{remark}
The discontinuity strata $\Sigma$ shrink as the inner sampling density
$N_{\mathrm{def}}$ is increased, because finer simplicial covers resolve
nearly-coalescing defective minima that are otherwise grouped into a
single $\mathcal{X}_D^{-}$ representative. In the limit
$N_{\mathrm{def}}\to\infty$ the augmented objective $\Phi_{\mathrm{aug}}$
collapses to $\Phi$ on a set of full measure, and the conclusions of
\autoref{theor:two-stage-convergence} hold for $\Phi$ itself.
\end{remark}

%% file: sections/appendix_ceos_validation.tex
%

\label{appendix:ceos-validation}

The three industrial case studies of \autoref{sec:results-sour-gas},
\autoref{sec:results-ccs} and \autoref{sec:results-refrigerant} all
use a cubic equation of state, for which the inner Gibbs energy
minimisation must contend with multiple volume roots. This appendix
validates that inner solver on an already-solved problem before it is
used on unsolved ones.

\paragraph{The reference study.} Glass, Djelassi and
Mitsos~\cite{Glass2018CEOS} pose cubic-equation-of-state parameter
estimation as a bilevel program in which the cubic itself enters as an
equality constraint on the lower level and three nested lower-level
programs encode a sufficient criterion for thermodynamic stability:
Baker's tangent-plane test, a root-discrimination test for mechanical
stability, and a test on the number of predicted phases.
Their proof of concept is the isothermal vapor--liquid equilibrium of
C$_5$H$_{12}$/H$_2$S at $T = 344.261$~K, taken from Reamer, Sage and
Lacey~\cite{Reamer1953} and reproduced in their supplementary
material: seven tie lines valid over $689.48$--$4826.33$~kPa, fitted
with both Peng--Robinson and Soave--Redlich--Kwong.
A single binary interaction parameter $k_{a,ij} \in [-1, 1]$ is
regressed with van der Waals mixing and combining rules, the second
interaction parameter $k_{b,ij}$ being held at zero.
Their solution is obtained with the in-house tool BOARPET driving the
commercial global solver BARON, with the lower-level programs
discretised on a grid of $30$ equidistant supporting points per mole
fraction variable in order to keep the computation tractable, at the
stated cost of the global optimality guarantee.

\paragraph{The replication.} The same seven tie lines, the same pure
component constants and the same single-parameter regression were
solved here with the two-stage geometric solver of
\autoref{sec:methods}, with the equation of state entering only
through the black-box callable $g(\mathbf{x}, \mathbf{p}, T, P)$ and
the stable root selected as the minimum-Gibbs-energy branch. No part
of the cubic is imposed as a constraint. The regressed interaction
parameters are $k_{a,ij} = 0.063431$ for Peng--Robinson and
$k_{a,ij} = 0.070902$ for Soave--Redlich--Kwong, against the published
values $0.064493$ and $0.071294$,
that is, relative deviations of $1.65\,\%$ and $0.55\,\%$ and absolute
deviations of $1.1 \times 10^{-3}$ and $3.9 \times 10^{-4}$ in a
parameter whose admissible range spans two units.

\input{figures/appendix-ceos-glass-validation/appendix-ceos-glass-validation_figure}

\paragraph{Agreement on the published objective.} The objective
minimised here is the geometric dual-manifold functional $\Phi$, not
the least-squares mismatch $f^{u}$ of~\cite{Glass2018CEOS}, so the two
converged objective values are not comparable quantities and are not
compared. To obtain a like-for-like number, the published objective
form, the sum over measurements and phases of the squared
mole-fraction residual,
was evaluated at the parameters regressed here, giving $0.00632$ (PR)
and $0.00622$ (SRK) against the published $f^{u,\star}$ of $0.00643$
and $0.00633$: lower by $1.8\,\%$ and $1.7\,\%$ respectively. The
corresponding mean absolute composition deviations at the measured
pressures are $0.0031$ (liquid) and $0.0235$ (vapor) for
Peng--Robinson, and $0.0027$ and $0.0233$ for
Soave--Redlich--Kwong, consistent with the tolerance of under
$2.5\,\%$ that Glass \emph{et al.} report for their own fit.
The reconstruction of $f^{u}$ rests on our reading of their objective
definition, in which the component index is summed over; the closeness
of the recovered values is evidence that the reading is correct, but it
is not a quantity taken from the paper.

\autoref{fig:appendix-ceos-glass-validation} reproduces both figures
that~\cite{Glass2018CEOS} publish for this system. Panel~(a) is the
analogue of their Fig.~6, the isothermal $P$--$x$--$y$ envelope against
the Reamer measurements, and recovers their observation that the
Peng--Robinson and Soave--Redlich--Kwong envelopes practically
coincide. Panels~(b) and~(c) are the analogues of the left and right
panels of their Fig.~5, the two Gibbs energy branches with the
supporting tangent at $2757.90$~kPa and $689.48$~kPa.
The predicted points of tangency fall at
$x_{\mathrm{C_5H_{12}}} = 0.104$ and $0.505$ at the higher pressure and
at $0.433$ and $0.919$ at the lower, matching the mole fractions marked
in their figure to the precision at which that figure can be read.

\paragraph{Cost, and what it does and does not show.} The two fits took
$164.7$~s (PR) and $106.9$~s (SRK) of wall clock on a single commodity
core, against the approximately $4.5$~h per case study
that~\cite{Glass2018CEOS} report,
a factor of roughly one hundred. This is an indicative comparison and
not a controlled one: the hardware differs, and more importantly the
two methods offer different guarantees, since the present inner solver
carries the deterministic adequate-sampling guarantees of SHGO
(\autoref{appendix:convergence-proof}) rather than a branch-and-bound
bound, while the discretised lower-level programs of the reference
study themselves surrender the global-optimality guarantee.
What the comparison does establish is that the cost profile is not the
one that led those authors to conclude that industrial case studies
were out of reach for their formulation.

\paragraph{Defect record.} For both models the solver entered its
second, defect-aware stage and the defect indicator remained non-empty
at convergence. With a single interaction parameter available and the
lower pressure isotherm approaching the mixture critical region, the
model has no freedom to remove the residual curvature defects while
retaining the fit, and forcing their suppression degrades the fit
without eliminating them, the same behaviour documented for the
sour-gas ternary in \autoref{sec:results-sour-gas}. In this case study
stage two therefore acted as a certifier rather than as a refitter, and
the reported parameters are the tangent-stage optimum. The result is
reported here as it stands rather than presented as defect-free.

%% file: figures/appendix-ceos-glass-validation/appendix-ceos-glass-validation_figure.tex
\begin{figure}[ht]
\centerline{\includegraphics[width=0.98\textwidth]%
{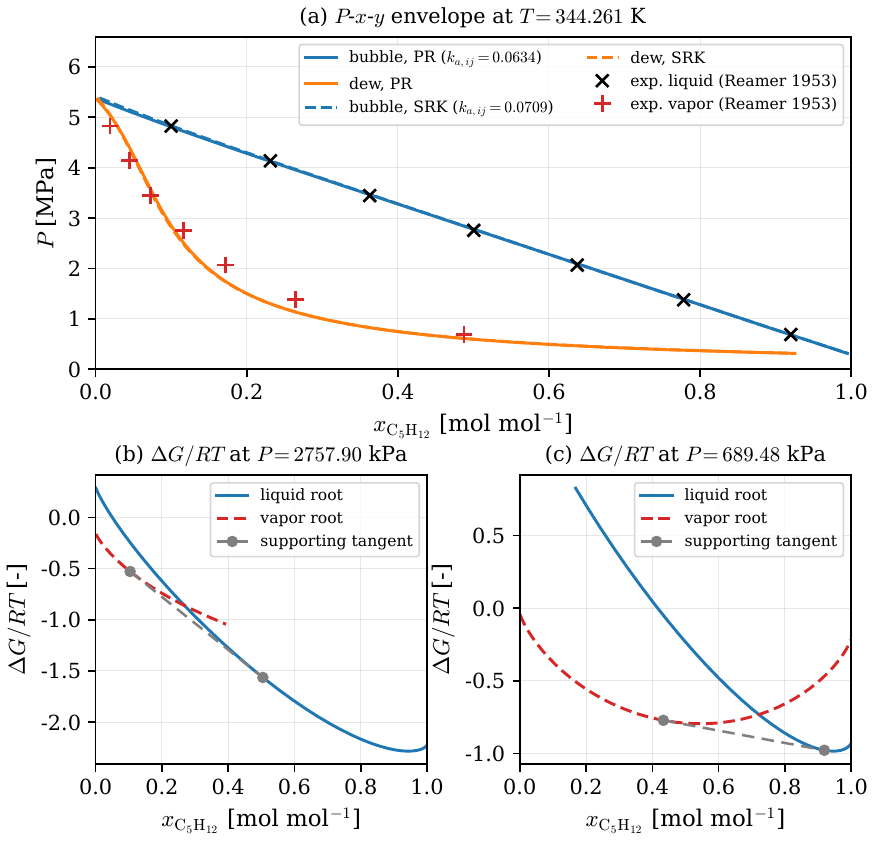}}
\caption{Replication of the C$_5$H$_{12}$/H$_2$S proof-of-concept case
study of Glass, Djelassi \& Mitsos~\cite{Glass2018CEOS} with the
present geometric bilevel solver.
(a) Isothermal $P$--$x$--$y$ envelope at $T = 344.261$~K computed at the
$k_{a,ij}$ values regressed here (PR, $0.063431$; SRK, $0.070902$),
against the seven tie lines of Reamer, Sage \& Lacey~\cite{Reamer1953}
spanning $689.48$--$4826.33$~kPa; bubble branches in blue, dew branches
in orange, PR solid and SRK dashed. This panel is the analogue of Fig.~6
of~\cite{Glass2018CEOS}, and reproduces their observation that the PR
and SRK envelopes practically coincide.
(b), (c) Gibbs free energy of mixing branches and the supporting
(Baker) tangent for PR at $2757.90$~kPa and $689.48$~kPa, the analogues
of the left and right panels of Fig.~5 of~\cite{Glass2018CEOS}. The
liquid root is blue, the vapor root red, and the gray dashed chord is
the supporting tangent whose points of tangency are the predicted
equilibrium compositions.}
\label{fig:appendix-ceos-glass-validation}
\end{figure}

%% file: sections/supp_lle_per_point.tex

In this supplement we show, for each of the four binary liquid--liquid
equilibrium (LLE) benchmark cases of~\cite{Mitsos2009LLE}, the Gibbs
free energy surface $g(x, \mathbf{p})$ and the dual manifold
$\mathcal{M}(x; \mathbf{p}) = g(x, \mathbf{p}) - h(x)$ at every
experimental temperature. The plotted curves are evaluated at the
bilevel-fitted parameters of the present work (denoted
$\mathbf{p}_{\mathrm{fit}}$ in the captions below), that is, the
optimum $\mathbf{p}^{\star}$ tabulated in the bilevel columns of the
LLE results section of the main
manuscript. They are \emph{not} the published literature parameters of
\cite{Mitsos2009LLE}; this supplement therefore complements the
literature-parameter phase envelopes shown in the main text by
documenting the stability record of the fitted optimum itself. The
NRTL non-randomness parameter $\alpha$
per case is given in the per-case discussion of the main text. For
each temperature the Gibbs plot (left) shows the surface together with
the experimental tie line and its common-tangent chord; the dual
manifold plot (right) shows $\mathcal{M}(x; \mathbf{p})$, whose
negative part $[\mathcal{M}]_{-}$ enters the bilevel objective.

The stability record at $\mathbf{p}_{\mathrm{fit}}$ across the archive
is as follows. Cases~1 and~3 are numerically defect-free: the interior
of $\mathcal{M}$ is strictly positive (minimum interior values
$+4.5\times10^{-8}$ and $+3.1\times10^{-8}$ respectively), with
equality to zero only at the two tie-line endpoints. Case~2 exhibits a
genuine negative lobe at every temperature, reaching
$\min \mathcal{M} = -2.35\times10^{-2}$ at $393.15$~K; Case~4 shows
shallow negative lobes ($\lvert \min \mathcal{M} \rvert \leq
1.5\times10^{-4}$) at six of the ten temperatures. These residual
defects are consistent with the larger least-squares residuals
reported for Cases~2 and~4 in the main text and diagnose model
misspecification at the fitted optimum rather than a failure of the
regression.

\subsection{Case 1: n-butyl-acetate -- water}

This case was fitted with NRTL non-randomness parameter $\alpha = 0.2$ at
six experimental temperatures (298.15, 303.15, 313.15, 323.15, 333.15,
343.15~K).

\begin{figure}[ht]
\centering
\begin{minipage}[t]{0.48\textwidth}
  \centering
  \includegraphics[width=\linewidth]{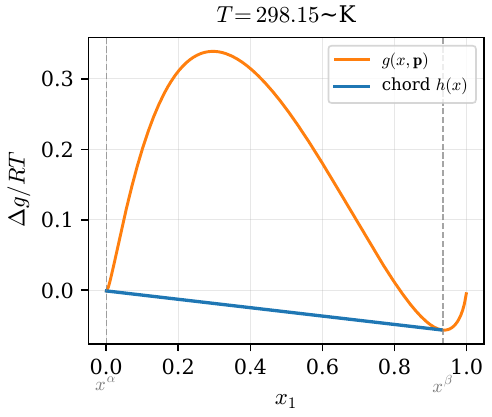}
\end{minipage}
\hfill
\begin{minipage}[t]{0.48\textwidth}
  \centering
  \includegraphics[width=\linewidth]{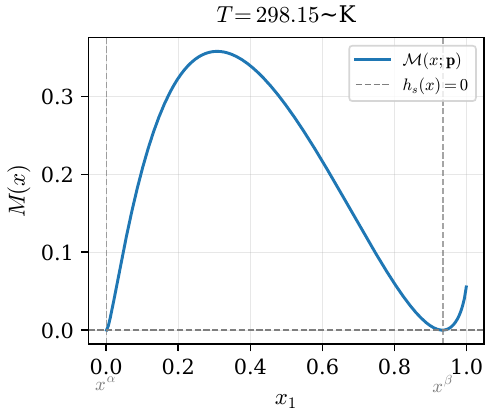}
\end{minipage}
\caption{Case 1 (n-butyl-acetate -- water) at $T = 298.15$~K. Left: Gibbs free energy surface $g(x, \mathbf{p}_{\mathrm{fit}})$; right: dual manifold $\mathcal{M}(x; \mathbf{p}_{\mathrm{fit}}) = g - h$.}
\label{fig:appendix-lle-butyl-acetate-water-T298_15K}
\end{figure}

\begin{figure}[ht]
\centering
\begin{minipage}[t]{0.48\textwidth}
  \centering
  \includegraphics[width=\linewidth]{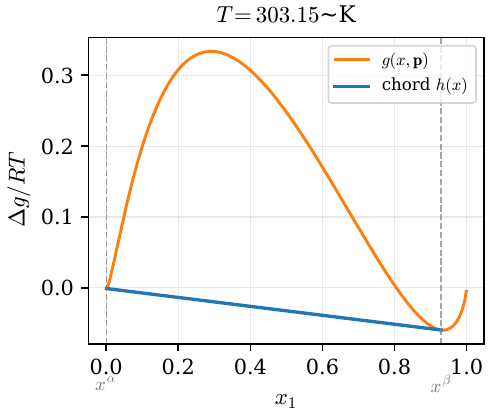}
\end{minipage}
\hfill
\begin{minipage}[t]{0.48\textwidth}
  \centering
  \includegraphics[width=\linewidth]{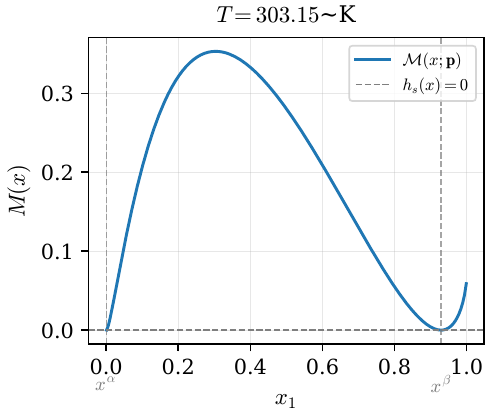}
\end{minipage}
\caption{Case 1 (n-butyl-acetate -- water) at $T = 303.15$~K. Left: Gibbs free energy surface $g(x, \mathbf{p}_{\mathrm{fit}})$; right: dual manifold $\mathcal{M}(x; \mathbf{p}_{\mathrm{fit}}) = g - h$.}
\label{fig:appendix-lle-butyl-acetate-water-T303_15K}
\end{figure}

\begin{figure}[ht]
\centering
\begin{minipage}[t]{0.48\textwidth}
  \centering
  \includegraphics[width=\linewidth]{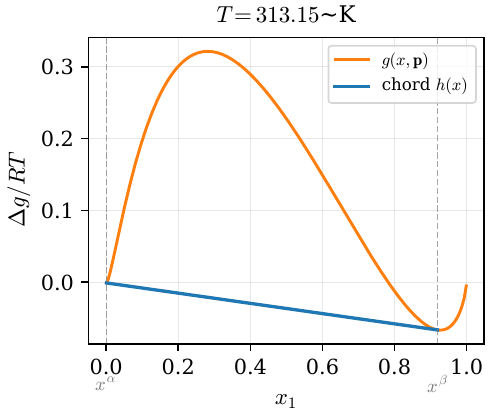}
\end{minipage}
\hfill
\begin{minipage}[t]{0.48\textwidth}
  \centering
  \includegraphics[width=\linewidth]{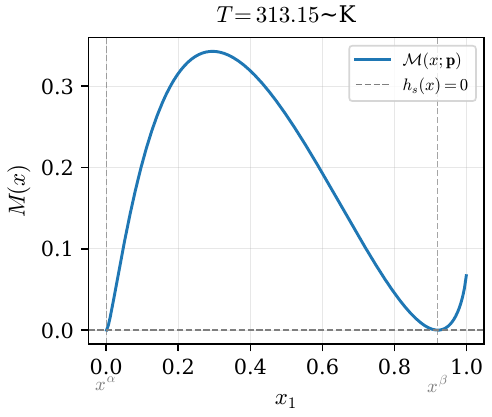}
\end{minipage}
\caption{Case 1 (n-butyl-acetate -- water) at $T = 313.15$~K. Left: Gibbs free energy surface $g(x, \mathbf{p}_{\mathrm{fit}})$; right: dual manifold $\mathcal{M}(x; \mathbf{p}_{\mathrm{fit}}) = g - h$.}
\label{fig:appendix-lle-butyl-acetate-water-T313_15K}
\end{figure}

\begin{figure}[ht]
\centering
\begin{minipage}[t]{0.48\textwidth}
  \centering
  \includegraphics[width=\linewidth]{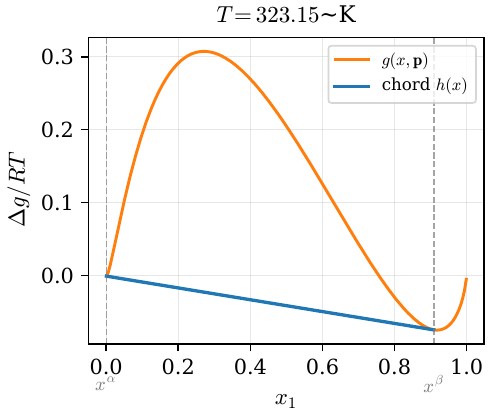}
\end{minipage}
\hfill
\begin{minipage}[t]{0.48\textwidth}
  \centering
  \includegraphics[width=\linewidth]{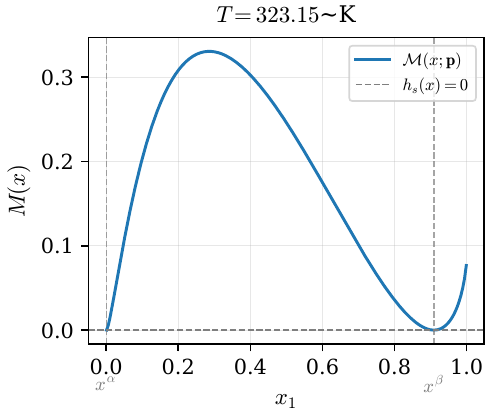}
\end{minipage}
\caption{Case 1 (n-butyl-acetate -- water) at $T = 323.15$~K. Left: Gibbs free energy surface $g(x, \mathbf{p}_{\mathrm{fit}})$; right: dual manifold $\mathcal{M}(x; \mathbf{p}_{\mathrm{fit}}) = g - h$.}
\label{fig:appendix-lle-butyl-acetate-water-T323_15K}
\end{figure}

\begin{figure}[ht]
\centering
\begin{minipage}[t]{0.48\textwidth}
  \centering
  \includegraphics[width=\linewidth]{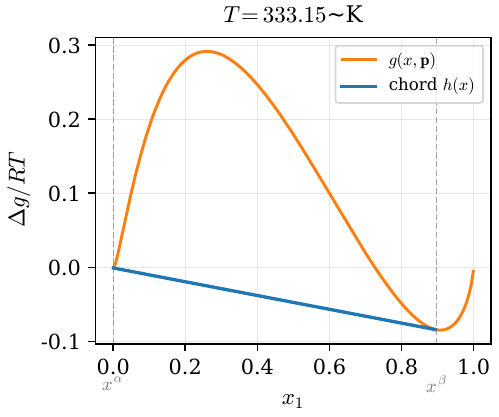}
\end{minipage}
\hfill
\begin{minipage}[t]{0.48\textwidth}
  \centering
  \includegraphics[width=\linewidth]{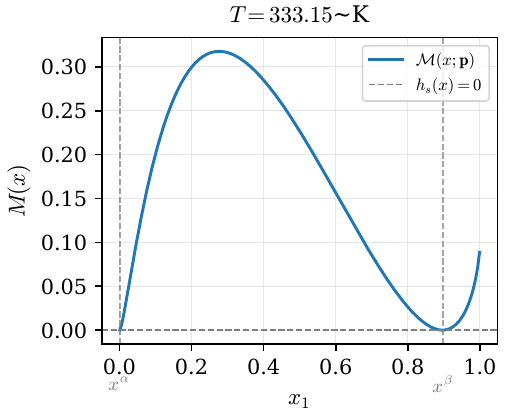}
\end{minipage}
\caption{Case 1 (n-butyl-acetate -- water) at $T = 333.15$~K. Left: Gibbs free energy surface $g(x, \mathbf{p}_{\mathrm{fit}})$; right: dual manifold $\mathcal{M}(x; \mathbf{p}_{\mathrm{fit}}) = g - h$.}
\label{fig:appendix-lle-butyl-acetate-water-T333_15K}
\end{figure}

\begin{figure}[ht]
\centering
\begin{minipage}[t]{0.48\textwidth}
  \centering
  \includegraphics[width=\linewidth]{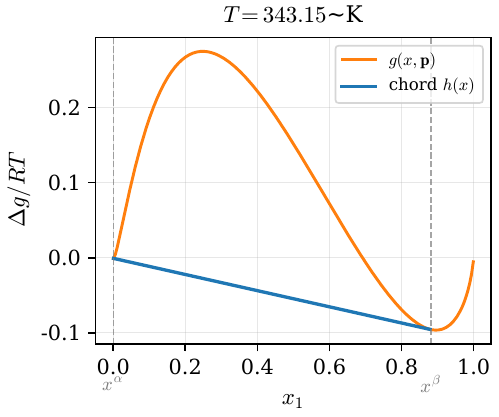}
\end{minipage}
\hfill
\begin{minipage}[t]{0.48\textwidth}
  \centering
  \includegraphics[width=\linewidth]{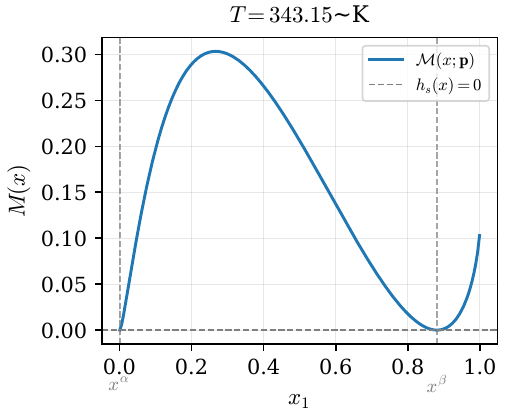}
\end{minipage}
\caption{Case 1 (n-butyl-acetate -- water) at $T = 343.15$~K. Left: Gibbs free energy surface $g(x, \mathbf{p}_{\mathrm{fit}})$; right: dual manifold $\mathcal{M}(x; \mathbf{p}_{\mathrm{fit}}) = g - h$.}
\label{fig:appendix-lle-butyl-acetate-water-T343_15K}
\end{figure}

\clearpage

\subsection{Case 2: n-butanol -- water}

This case was fitted with NRTL non-randomness parameter $\alpha = 0.44$ at
eight experimental temperatures (273.15, 293.15, 298.15, 313.15, 333.15,
353.15, 373.15, 393.15~K).

\begin{figure}[ht]
\centering
\begin{minipage}[t]{0.48\textwidth}
  \centering
  \includegraphics[width=\linewidth]{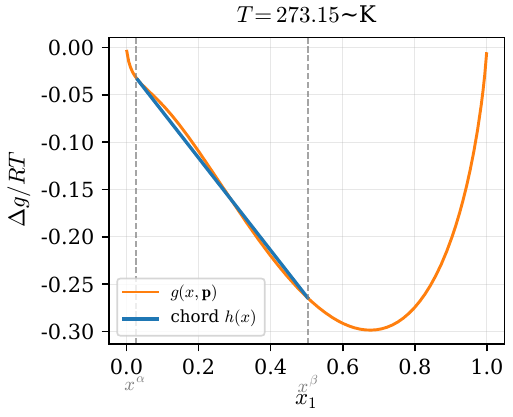}
\end{minipage}
\hfill
\begin{minipage}[t]{0.48\textwidth}
  \centering
  \includegraphics[width=\linewidth]{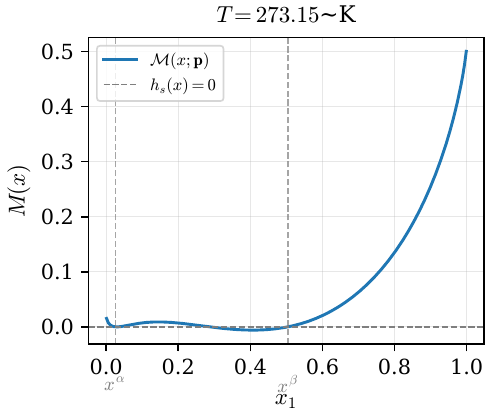}
\end{minipage}
\caption{Case 2 (n-butanol -- water) at $T = 273.15$~K. Left: Gibbs free energy surface $g(x, \mathbf{p}_{\mathrm{fit}})$; right: dual manifold $\mathcal{M}(x; \mathbf{p}_{\mathrm{fit}}) = g - h$.}
\label{fig:appendix-lle-butanol-water-T273_15K}
\end{figure}

\begin{figure}[ht]
\centering
\begin{minipage}[t]{0.48\textwidth}
  \centering
  \includegraphics[width=\linewidth]{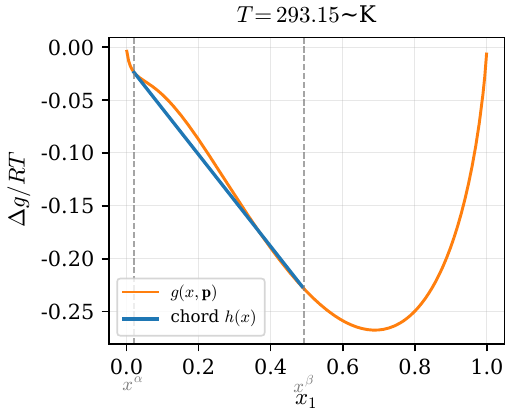}
\end{minipage}
\hfill
\begin{minipage}[t]{0.48\textwidth}
  \centering
  \includegraphics[width=\linewidth]{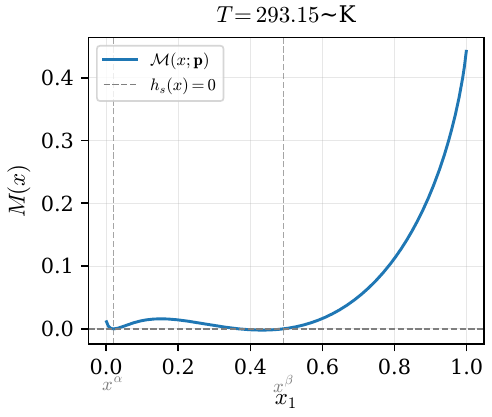}
\end{minipage}
\caption{Case 2 (n-butanol -- water) at $T = 293.15$~K. Left: Gibbs free energy surface $g(x, \mathbf{p}_{\mathrm{fit}})$; right: dual manifold $\mathcal{M}(x; \mathbf{p}_{\mathrm{fit}}) = g - h$.}
\label{fig:appendix-lle-butanol-water-T293_15K}
\end{figure}

\begin{figure}[ht]
\centering
\begin{minipage}[t]{0.48\textwidth}
  \centering
  \includegraphics[width=\linewidth]{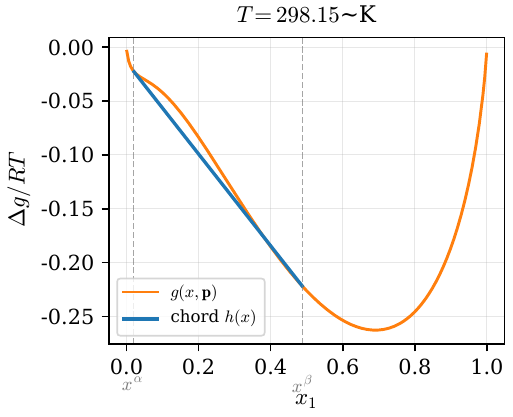}
\end{minipage}
\hfill
\begin{minipage}[t]{0.48\textwidth}
  \centering
  \includegraphics[width=\linewidth]{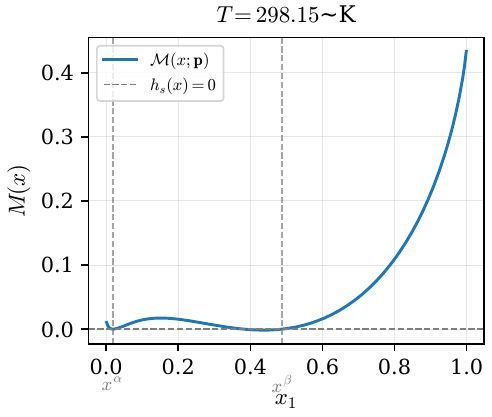}
\end{minipage}
\caption{Case 2 (n-butanol -- water) at $T = 298.15$~K. Left: Gibbs free energy surface $g(x, \mathbf{p}_{\mathrm{fit}})$; right: dual manifold $\mathcal{M}(x; \mathbf{p}_{\mathrm{fit}}) = g - h$.}
\label{fig:appendix-lle-butanol-water-T298_15K}
\end{figure}

\begin{figure}[ht]
\centering
\begin{minipage}[t]{0.48\textwidth}
  \centering
  \includegraphics[width=\linewidth]{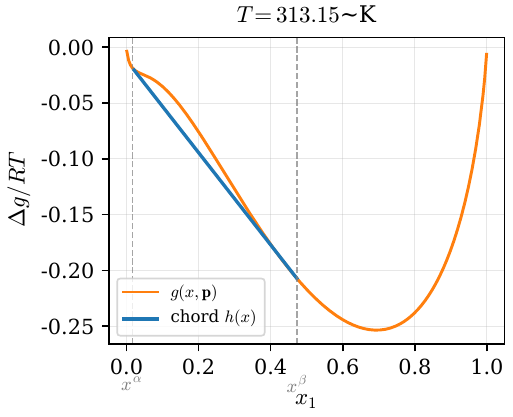}
\end{minipage}
\hfill
\begin{minipage}[t]{0.48\textwidth}
  \centering
  \includegraphics[width=\linewidth]{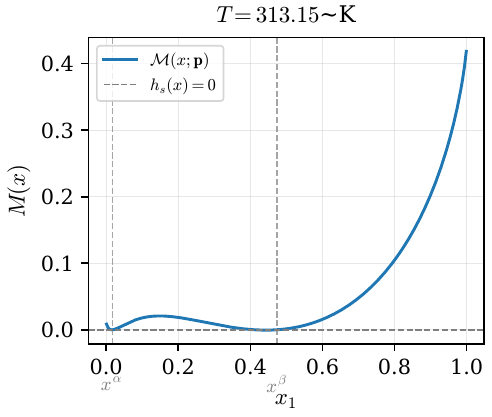}
\end{minipage}
\caption{Case 2 (n-butanol -- water) at $T = 313.15$~K. Left: Gibbs free energy surface $g(x, \mathbf{p}_{\mathrm{fit}})$; right: dual manifold $\mathcal{M}(x; \mathbf{p}_{\mathrm{fit}}) = g - h$.}
\label{fig:appendix-lle-butanol-water-T313_15K}
\end{figure}

\begin{figure}[ht]
\centering
\begin{minipage}[t]{0.48\textwidth}
  \centering
  \includegraphics[width=\linewidth]{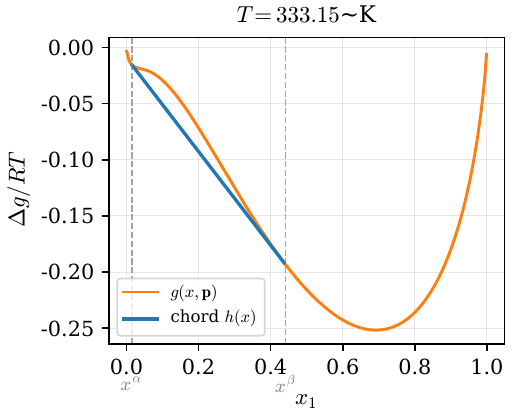}
\end{minipage}
\hfill
\begin{minipage}[t]{0.48\textwidth}
  \centering
  \includegraphics[width=\linewidth]{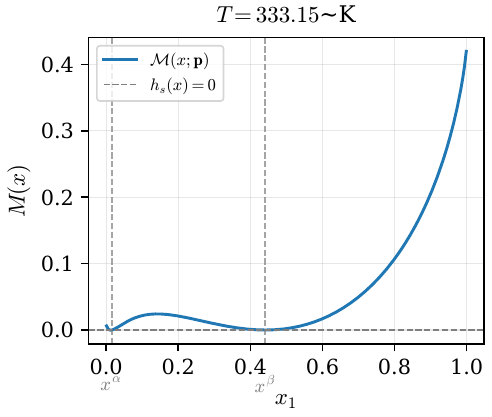}
\end{minipage}
\caption{Case 2 (n-butanol -- water) at $T = 333.15$~K. Left: Gibbs free energy surface $g(x, \mathbf{p}_{\mathrm{fit}})$; right: dual manifold $\mathcal{M}(x; \mathbf{p}_{\mathrm{fit}}) = g - h$.}
\label{fig:appendix-lle-butanol-water-T333_15K}
\end{figure}

\begin{figure}[ht]
\centering
\begin{minipage}[t]{0.48\textwidth}
  \centering
  \includegraphics[width=\linewidth]{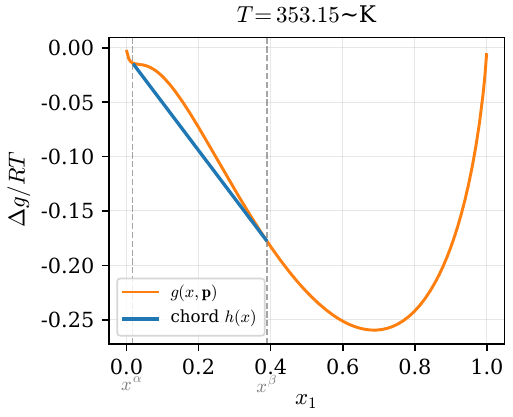}
\end{minipage}
\hfill
\begin{minipage}[t]{0.48\textwidth}
  \centering
  \includegraphics[width=\linewidth]{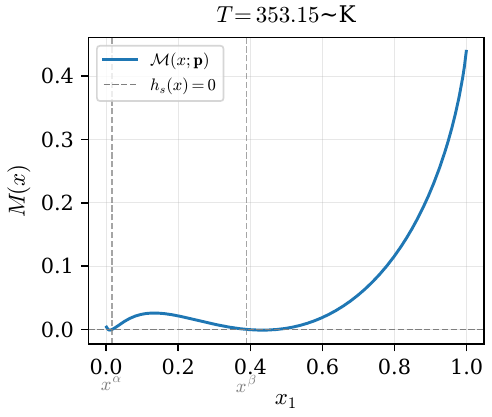}
\end{minipage}
\caption{Case 2 (n-butanol -- water) at $T = 353.15$~K. Left: Gibbs free energy surface $g(x, \mathbf{p}_{\mathrm{fit}})$; right: dual manifold $\mathcal{M}(x; \mathbf{p}_{\mathrm{fit}}) = g - h$.}
\label{fig:appendix-lle-butanol-water-T353_15K}
\end{figure}

\begin{figure}[ht]
\centering
\begin{minipage}[t]{0.48\textwidth}
  \centering
  \includegraphics[width=\linewidth]{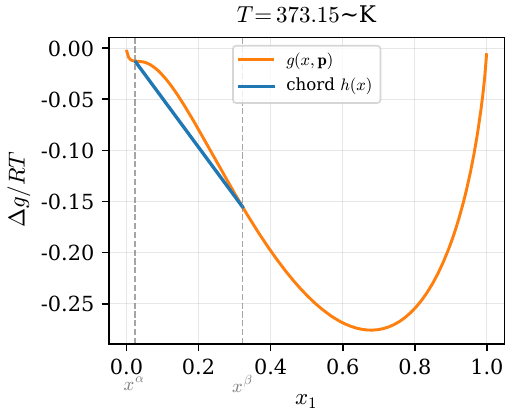}
\end{minipage}
\hfill
\begin{minipage}[t]{0.48\textwidth}
  \centering
  \includegraphics[width=\linewidth]{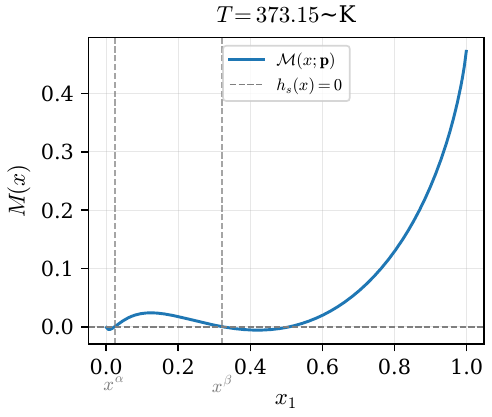}
\end{minipage}
\caption{Case 2 (n-butanol -- water) at $T = 373.15$~K. Left: Gibbs free energy surface $g(x, \mathbf{p}_{\mathrm{fit}})$; right: dual manifold $\mathcal{M}(x; \mathbf{p}_{\mathrm{fit}}) = g - h$.}
\label{fig:appendix-lle-butanol-water-T373_15K}
\end{figure}

\begin{figure}[ht]
\centering
\begin{minipage}[t]{0.48\textwidth}
  \centering
  \includegraphics[width=\linewidth]{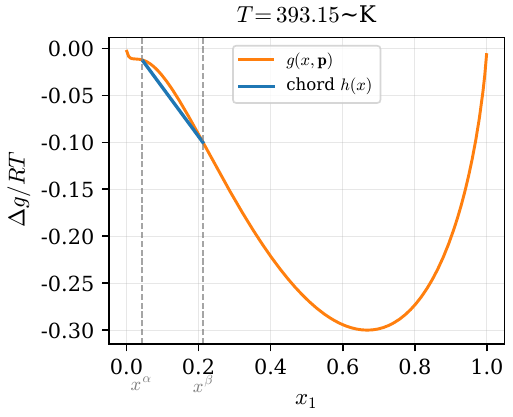}
\end{minipage}
\hfill
\begin{minipage}[t]{0.48\textwidth}
  \centering
  \includegraphics[width=\linewidth]{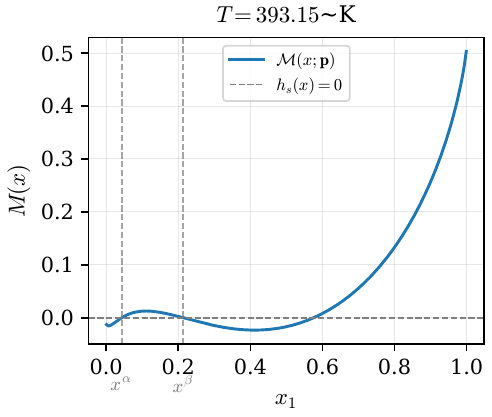}
\end{minipage}
\caption{Case 2 (n-butanol -- water) at $T = 393.15$~K. Left: Gibbs free energy surface $g(x, \mathbf{p}_{\mathrm{fit}})$; right: dual manifold $\mathcal{M}(x; \mathbf{p}_{\mathrm{fit}}) = g - h$.}
\label{fig:appendix-lle-butanol-water-T393_15K}
\end{figure}

\clearpage

\subsection{Case 3: n-octanol -- water}

This case was fitted with NRTL non-randomness parameter $\alpha = 0.2$ at
four experimental temperatures (293.15, 298.15, 313.15, 333.15~K).

\begin{figure}[ht]
\centering
\begin{minipage}[t]{0.48\textwidth}
  \centering
  \includegraphics[width=\linewidth]{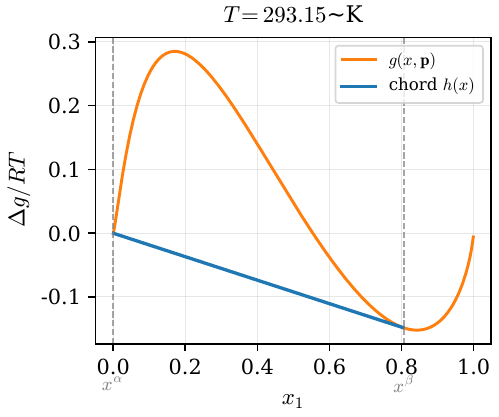}
\end{minipage}
\hfill
\begin{minipage}[t]{0.48\textwidth}
  \centering
  \includegraphics[width=\linewidth]{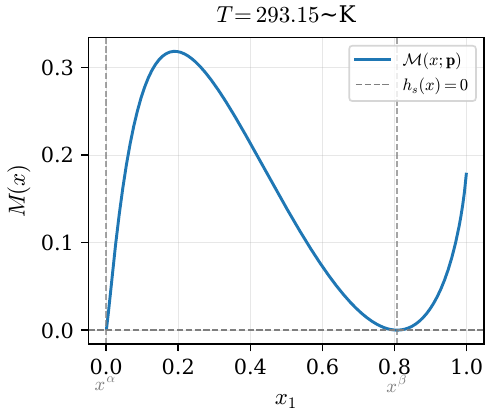}
\end{minipage}
\caption{Case 3 (n-octanol -- water) at $T = 293.15$~K. Left: Gibbs free energy surface $g(x, \mathbf{p}_{\mathrm{fit}})$; right: dual manifold $\mathcal{M}(x; \mathbf{p}_{\mathrm{fit}}) = g - h$.}
\label{fig:appendix-lle-octanol-water-T293_15K}
\end{figure}

\begin{figure}[ht]
\centering
\begin{minipage}[t]{0.48\textwidth}
  \centering
  \includegraphics[width=\linewidth]{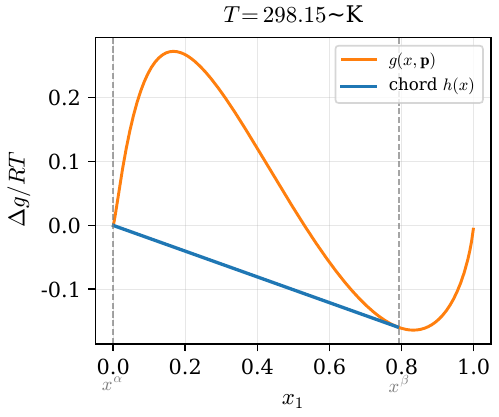}
\end{minipage}
\hfill
\begin{minipage}[t]{0.48\textwidth}
  \centering
  \includegraphics[width=\linewidth]{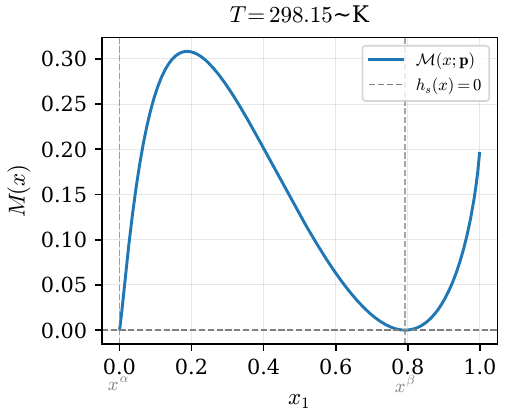}
\end{minipage}
\caption{Case 3 (n-octanol -- water) at $T = 298.15$~K. Left: Gibbs free energy surface $g(x, \mathbf{p}_{\mathrm{fit}})$; right: dual manifold $\mathcal{M}(x; \mathbf{p}_{\mathrm{fit}}) = g - h$.}
\label{fig:appendix-lle-octanol-water-T298_15K}
\end{figure}

\begin{figure}[ht]
\centering
\begin{minipage}[t]{0.48\textwidth}
  \centering
  \includegraphics[width=\linewidth]{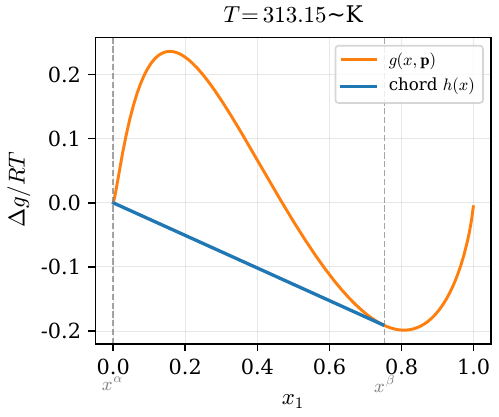}
\end{minipage}
\hfill
\begin{minipage}[t]{0.48\textwidth}
  \centering
  \includegraphics[width=\linewidth]{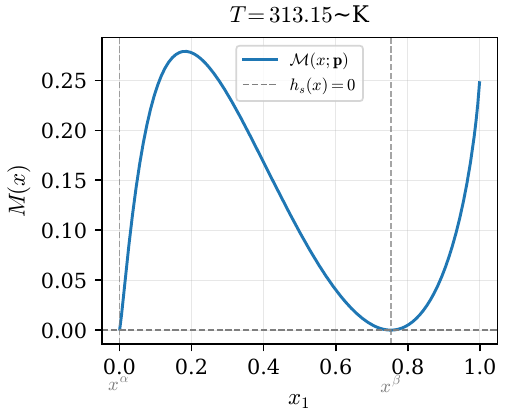}
\end{minipage}
\caption{Case 3 (n-octanol -- water) at $T = 313.15$~K. Left: Gibbs free energy surface $g(x, \mathbf{p}_{\mathrm{fit}})$; right: dual manifold $\mathcal{M}(x; \mathbf{p}_{\mathrm{fit}}) = g - h$.}
\label{fig:appendix-lle-octanol-water-T313_15K}
\end{figure}

\begin{figure}[ht]
\centering
\begin{minipage}[t]{0.48\textwidth}
  \centering
  \includegraphics[width=\linewidth]{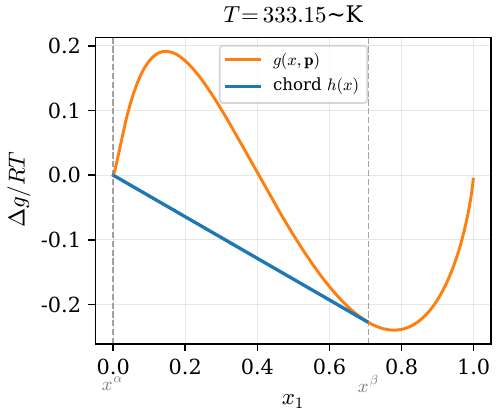}
\end{minipage}
\hfill
\begin{minipage}[t]{0.48\textwidth}
  \centering
  \includegraphics[width=\linewidth]{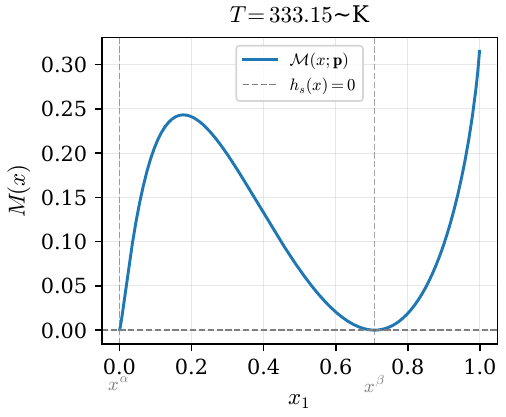}
\end{minipage}
\caption{Case 3 (n-octanol -- water) at $T = 333.15$~K. Left: Gibbs free energy surface $g(x, \mathbf{p}_{\mathrm{fit}})$; right: dual manifold $\mathcal{M}(x; \mathbf{p}_{\mathrm{fit}}) = g - h$.}
\label{fig:appendix-lle-octanol-water-T333_15K}
\end{figure}

\clearpage

\subsection{Case 4: furfural -- 2,2,5-trimethyl-hexane}

This case was fitted with NRTL non-randomness parameter $\alpha = 0.2$ at
ten experimental temperatures (293.15, 298.15, 303.15, 313.15, 323.15,
333.15, 343.15, 353.15, 363.15, 373.15~K).

\begin{figure}[ht]
\centering
\begin{minipage}[t]{0.48\textwidth}
  \centering
  \includegraphics[width=\linewidth]{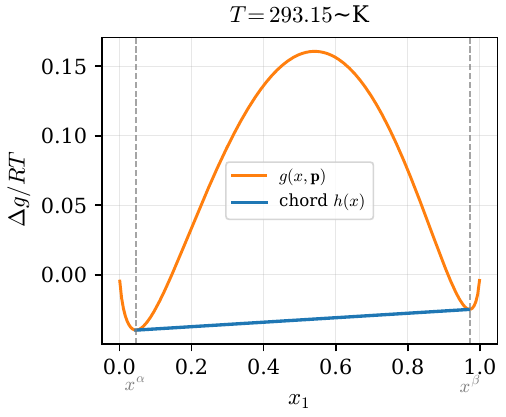}
\end{minipage}
\hfill
\begin{minipage}[t]{0.48\textwidth}
  \centering
  \includegraphics[width=\linewidth]{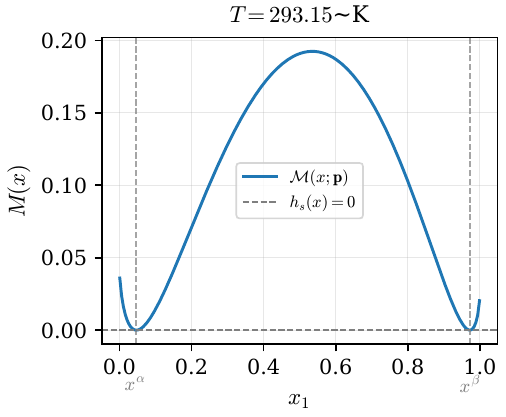}
\end{minipage}
\caption{Case 4 (furfural -- 2,2,5-trimethyl-hexane) at $T = 293.15$~K. Left: Gibbs free energy surface $g(x, \mathbf{p}_{\mathrm{fit}})$; right: dual manifold $\mathcal{M}(x; \mathbf{p}_{\mathrm{fit}}) = g - h$.}
\label{fig:appendix-lle-furfural-tmh-T293_15K}
\end{figure}

\begin{figure}[ht]
\centering
\begin{minipage}[t]{0.48\textwidth}
  \centering
  \includegraphics[width=\linewidth]{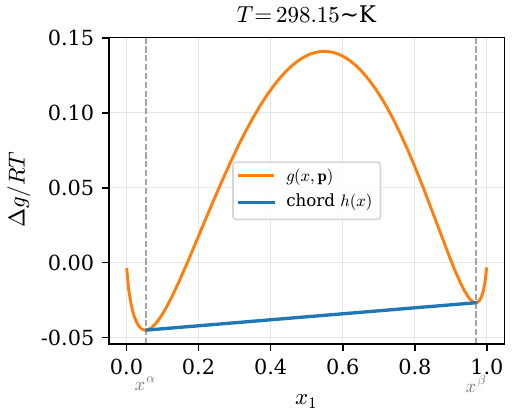}
\end{minipage}
\hfill
\begin{minipage}[t]{0.48\textwidth}
  \centering
  \includegraphics[width=\linewidth]{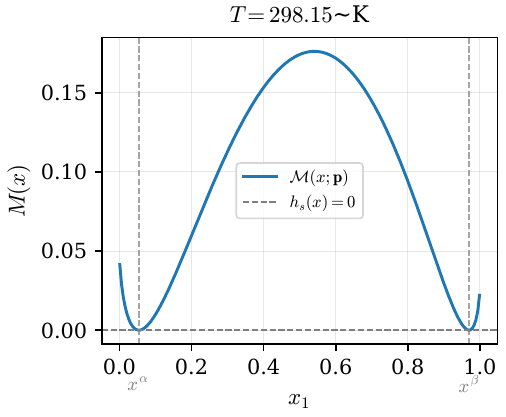}
\end{minipage}
\caption{Case 4 (furfural -- 2,2,5-trimethyl-hexane) at $T = 298.15$~K. Left: Gibbs free energy surface $g(x, \mathbf{p}_{\mathrm{fit}})$; right: dual manifold $\mathcal{M}(x; \mathbf{p}_{\mathrm{fit}}) = g - h$.}
\label{fig:appendix-lle-furfural-tmh-T298_15K}
\end{figure}

\begin{figure}[ht]
\centering
\begin{minipage}[t]{0.48\textwidth}
  \centering
  \includegraphics[width=\linewidth]{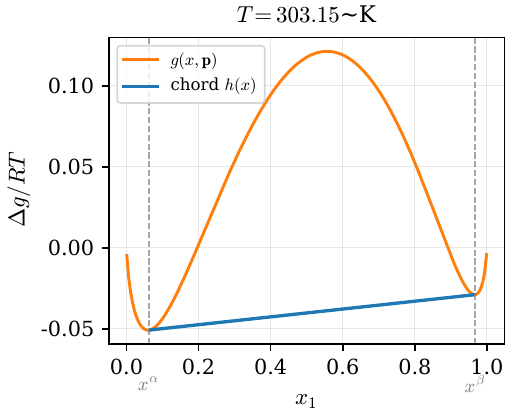}
\end{minipage}
\hfill
\begin{minipage}[t]{0.48\textwidth}
  \centering
  \includegraphics[width=\linewidth]{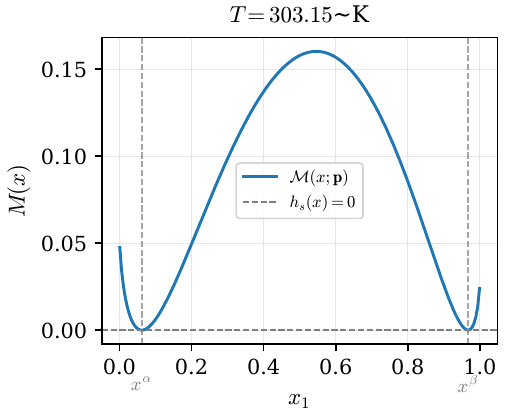}
\end{minipage}
\caption{Case 4 (furfural -- 2,2,5-trimethyl-hexane) at $T = 303.15$~K. Left: Gibbs free energy surface $g(x, \mathbf{p}_{\mathrm{fit}})$; right: dual manifold $\mathcal{M}(x; \mathbf{p}_{\mathrm{fit}}) = g - h$.}
\label{fig:appendix-lle-furfural-tmh-T303_15K}
\end{figure}

\begin{figure}[ht]
\centering
\begin{minipage}[t]{0.48\textwidth}
  \centering
  \includegraphics[width=\linewidth]{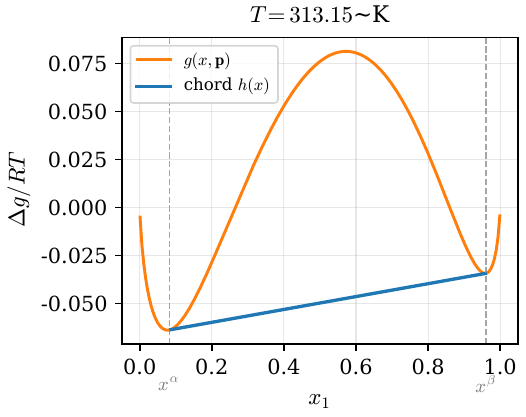}
\end{minipage}
\hfill
\begin{minipage}[t]{0.48\textwidth}
  \centering
  \includegraphics[width=\linewidth]{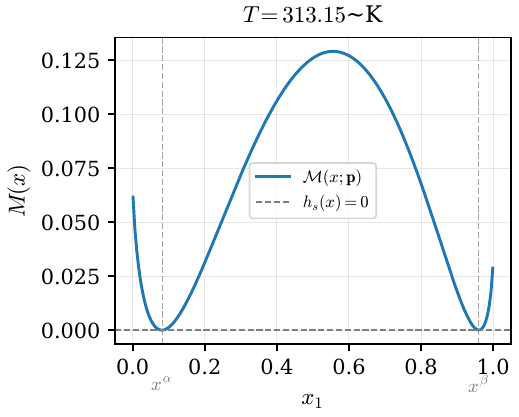}
\end{minipage}
\caption{Case 4 (furfural -- 2,2,5-trimethyl-hexane) at $T = 313.15$~K. Left: Gibbs free energy surface $g(x, \mathbf{p}_{\mathrm{fit}})$; right: dual manifold $\mathcal{M}(x; \mathbf{p}_{\mathrm{fit}}) = g - h$.}
\label{fig:appendix-lle-furfural-tmh-T313_15K}
\end{figure}

\begin{figure}[ht]
\centering
\begin{minipage}[t]{0.48\textwidth}
  \centering
  \includegraphics[width=\linewidth]{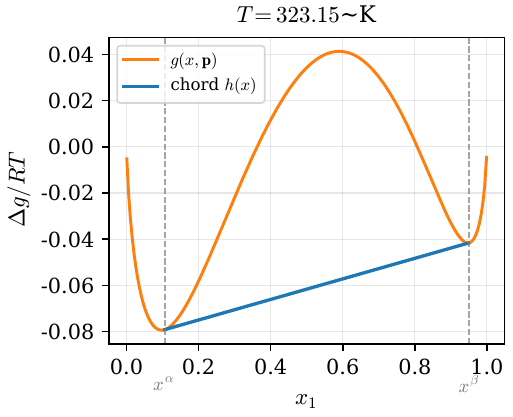}
\end{minipage}
\hfill
\begin{minipage}[t]{0.48\textwidth}
  \centering
  \includegraphics[width=\linewidth]{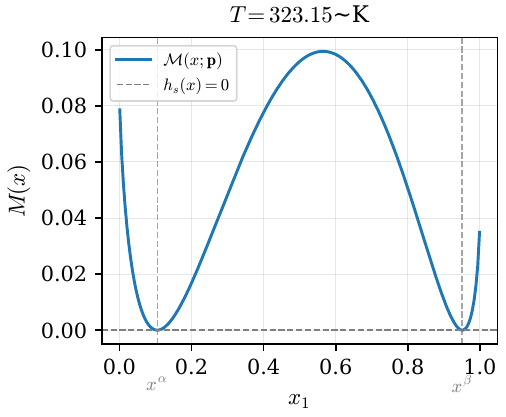}
\end{minipage}
\caption{Case 4 (furfural -- 2,2,5-trimethyl-hexane) at $T = 323.15$~K. Left: Gibbs free energy surface $g(x, \mathbf{p}_{\mathrm{fit}})$; right: dual manifold $\mathcal{M}(x; \mathbf{p}_{\mathrm{fit}}) = g - h$.}
\label{fig:appendix-lle-furfural-tmh-T323_15K}
\end{figure}

\begin{figure}[ht]
\centering
\begin{minipage}[t]{0.48\textwidth}
  \centering
  \includegraphics[width=\linewidth]{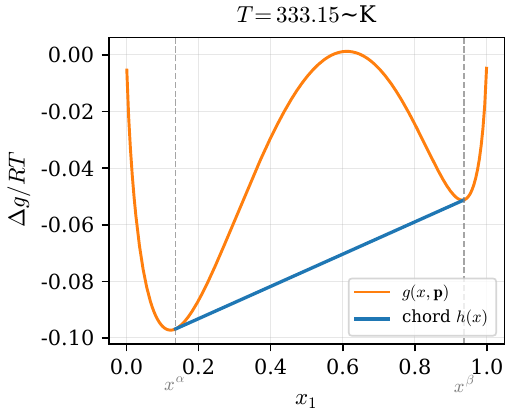}
\end{minipage}
\hfill
\begin{minipage}[t]{0.48\textwidth}
  \centering
  \includegraphics[width=\linewidth]{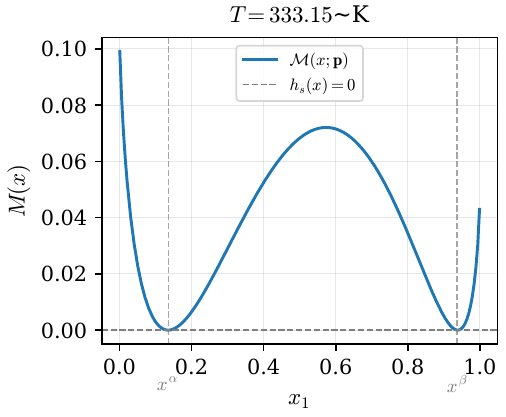}
\end{minipage}
\caption{Case 4 (furfural -- 2,2,5-trimethyl-hexane) at $T = 333.15$~K. Left: Gibbs free energy surface $g(x, \mathbf{p}_{\mathrm{fit}})$; right: dual manifold $\mathcal{M}(x; \mathbf{p}_{\mathrm{fit}}) = g - h$.}
\label{fig:appendix-lle-furfural-tmh-T333_15K}
\end{figure}

\begin{figure}[ht]
\centering
\begin{minipage}[t]{0.48\textwidth}
  \centering
  \includegraphics[width=\linewidth]{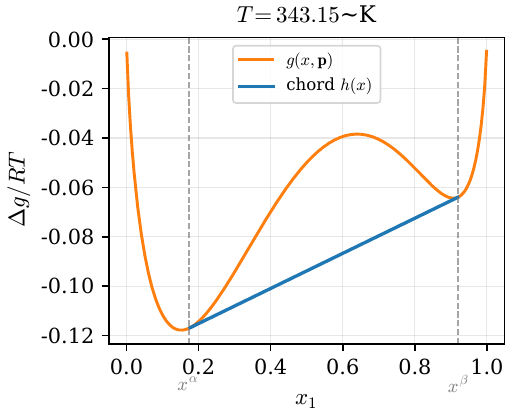}
\end{minipage}
\hfill
\begin{minipage}[t]{0.48\textwidth}
  \centering
  \includegraphics[width=\linewidth]{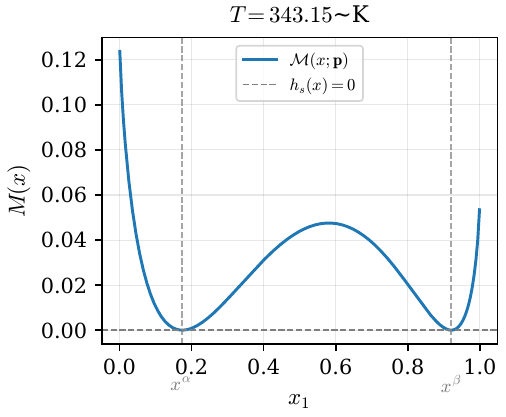}
\end{minipage}
\caption{Case 4 (furfural -- 2,2,5-trimethyl-hexane) at $T = 343.15$~K. Left: Gibbs free energy surface $g(x, \mathbf{p}_{\mathrm{fit}})$; right: dual manifold $\mathcal{M}(x; \mathbf{p}_{\mathrm{fit}}) = g - h$.}
\label{fig:appendix-lle-furfural-tmh-T343_15K}
\end{figure}

\begin{figure}[ht]
\centering
\begin{minipage}[t]{0.48\textwidth}
  \centering
  \includegraphics[width=\linewidth]{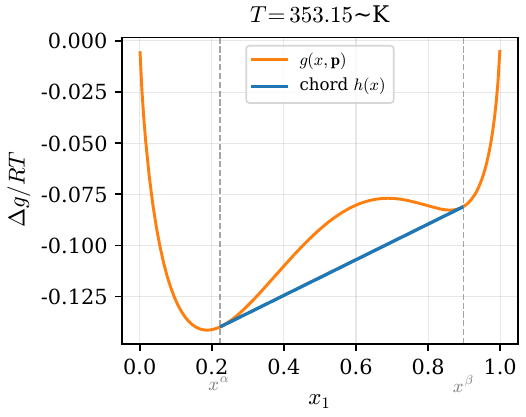}
\end{minipage}
\hfill
\begin{minipage}[t]{0.48\textwidth}
  \centering
  \includegraphics[width=\linewidth]{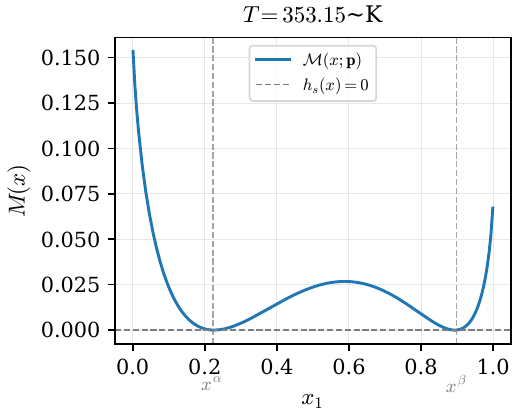}
\end{minipage}
\caption{Case 4 (furfural -- 2,2,5-trimethyl-hexane) at $T = 353.15$~K. Left: Gibbs free energy surface $g(x, \mathbf{p}_{\mathrm{fit}})$; right: dual manifold $\mathcal{M}(x; \mathbf{p}_{\mathrm{fit}}) = g - h$.}
\label{fig:appendix-lle-furfural-tmh-T353_15K}
\end{figure}

\begin{figure}[ht]
\centering
\begin{minipage}[t]{0.48\textwidth}
  \centering
  \includegraphics[width=\linewidth]{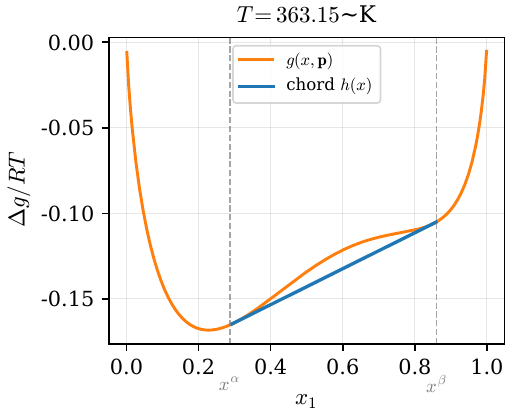}
\end{minipage}
\hfill
\begin{minipage}[t]{0.48\textwidth}
  \centering
  \includegraphics[width=\linewidth]{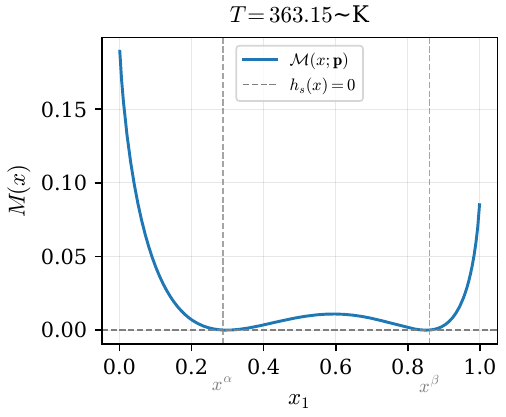}
\end{minipage}
\caption{Case 4 (furfural -- 2,2,5-trimethyl-hexane) at $T = 363.15$~K. Left: Gibbs free energy surface $g(x, \mathbf{p}_{\mathrm{fit}})$; right: dual manifold $\mathcal{M}(x; \mathbf{p}_{\mathrm{fit}}) = g - h$.}
\label{fig:appendix-lle-furfural-tmh-T363_15K}
\end{figure}

\begin{figure}[ht]
\centering
\begin{minipage}[t]{0.48\textwidth}
  \centering
  \includegraphics[width=\linewidth]{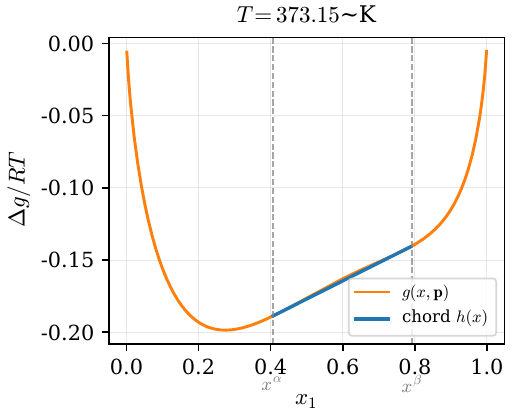}
\end{minipage}
\hfill
\begin{minipage}[t]{0.48\textwidth}
  \centering
  \includegraphics[width=\linewidth]{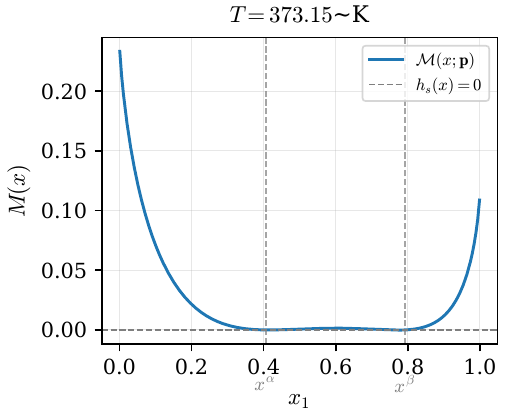}
\end{minipage}
\caption{Case 4 (furfural -- 2,2,5-trimethyl-hexane) at $T = 373.15$~K. Left: Gibbs free energy surface $g(x, \mathbf{p}_{\mathrm{fit}})$; right: dual manifold $\mathcal{M}(x; \mathbf{p}_{\mathrm{fit}}) = g - h$.}
\label{fig:appendix-lle-furfural-tmh-T373_15K}
\end{figure}